\documentclass[a4paper]{amsart}

\usepackage[left=2.5cm,right=2.5cm,top=2.5cm,bottom=2.5cm]{geometry}

\usepackage{amsmath}
\usepackage{amsthm}
\usepackage{amssymb}
\usepackage{mathtools}
\usepackage{enumerate}
\usepackage{hyperref}
\usepackage{color}
\usepackage{graphicx}
\usepackage{import}
\usepackage{xcolor}
\usepackage{esint}

\theoremstyle{plain}
\newtheorem{theorem}{Theorem}[section]
\newtheorem{proposition}[theorem]{Proposition}
\newtheorem{lemma}[theorem]{Lemma}
\newtheorem{corollary}[theorem]{Corollary}
\newtheorem{claim}{Claim}

\theoremstyle{definition}
\newtheorem{definition}[theorem]{Definition}
\newtheorem{remark}[theorem]{Remark}
\newtheorem{example}[theorem]{Example}

\numberwithin{equation}{section}
\numberwithin{figure}{section}

\newcommand{\BA}{\mathbb{A}}

\newcommand{\BD}{\mathbb{D}}

\newcommand{\BQ}{\mathbb{Q}}
\newcommand{\BR}{\mathbb{R}}

\newcommand{\BT}{\mathbb{T}}

\newcommand{\BZ}{\mathbb{Z}}

\newcommand{\MB}{\mathcal{B}}

\newcommand{\MD}{\mathcal{D}}
\newcommand{\ME}{\mathcal{E}}
\newcommand{\MF}{\mathcal{F}}
\newcommand{\MG}{\mathcal{G}}
\newcommand{\MH}{\mathcal{H}}
\newcommand{\MI}{\mathcal{I}}

\newcommand{\MK}{\mathcal{K}}
\newcommand{\ML}{\mathcal{L}}

\newcommand{\MN}{\mathcal{N}}

\newcommand{\MP}{\mathcal{P}}
\newcommand{\MQ}{\mathcal{Q}}
\newcommand{\MRR}{\mathcal{R}}

\newcommand{\MT}{\mathcal{T}}
\newcommand{\MU}{\mathcal{U}}
\newcommand{\MV}{\mathcal{V}}

\title{Smooth perfectness of Hamiltonian diffeomorphism groups}
\author{Oliver Edtmair}

\begin{document}

\maketitle

\begin{abstract}
A fundamental result of Banyaga states that the Hamiltonian diffeomorphism group of a closed symplectic manifold is perfect. We refine this result by proving that, locally in the \( C^\infty \) topology, the number of commutators needed to express a Hamiltonian diffeomorphism is bounded, and the commutators can be chosen to depend smoothly on the diffeomorphism. As a corollary, we show that any homogeneous quasimorphism on the Hamiltonian diffeomorphism group is continuous in the \( C^\infty \) topology.

We establish an analogous smooth perfectness result for compactly supported Hamiltonian diffeomorphisms of open symplectic manifolds, where Banyaga proved that the kernel of the Calabi homomorphism is perfect. For symplectic manifolds with boundary, we define a natural group of Hamiltonian diffeomorphisms that need not restrict to the identity on the boundary. We extend the Calabi homomorphism to this setting and prove a corresponding refined perfectness result.

These results are motivated by our work on symplectic packing stability for symplectic manifolds with boundary, where they play a key role in our symplectic embedding constructions.
\end{abstract}

\tableofcontents

\section{Introduction}

\subsection{Main results}

Let $(M^{2n},\omega)$ be a connected closed symplectic manifold. A classic result by Banyaga \cite{ban78} asserts that the Hamiltonian diffeomorphism group $\operatorname{Ham}(M,\omega)$ is \textit{perfect} and \textit{simple}; see also \cite[\S4]{ban97}. This means that $\operatorname{Ham}(M,\omega)$ is equal to its commutator subgroup and does not have any non-trivial proper normal subgroups. Banyaga's result was preceded by influential works of Epstein~\cite{eps70}, Herman~\cite{her71}, Thurston~\cite{thu74}, and Mather~\cite{mat74, mat75} on diffeomorphism groups.

Throughout this paper, we equip $\operatorname{Ham}(M,\omega)$ with the $C^\infty$ topology, unless stated otherwise. Our first result states that, locally, the factorization of a Hamiltonian diffeomorphism $\varphi\in \operatorname{Ham}(M,\omega)$ into commutators can be chosen to depend smoothly on $\varphi$. Here, the smooth dependence is understood in the diffeological sense; see Section~\ref{sec:preliminaries}.

\begin{theorem}[Smooth perfectness for closed manifolds]
\label{thm:smooth_perfectness_closed}
Let $n>0$ be an integer. Then there exists an integer $m>0$ such that the following is true: Let $(M^{2n},\omega)$ be a closed symplectic manifold and consider the map
\begin{equation*}
\Phi:\operatorname{Ham}(M,\omega)^{2m} \rightarrow \operatorname{Ham}(M,\omega) \quad (u_1,v_1,\dots,u_m,v_m)\mapsto \prod\limits_{j=1}^m [u_j,v_j].
\end{equation*}
Then there exist an open neighbourhood $\MN\subset \operatorname{Ham}(M,\omega)$ of the identity and a smooth map
\begin{equation*}
\Psi: \MN \rightarrow \operatorname{Ham}(M,\omega)^{2m}
\end{equation*}
which is a local right inverse of $\Phi$, that is $\Phi\circ \Psi = \operatorname{id}_{\MN}$. We can choose $\MN$ and $\Psi$ such that $\Psi(\operatorname{id})$ is arbitrarily close to the tuple $(\operatorname{id},\dots,\operatorname{id})$.
\end{theorem}

\begin{remark}
An analogous result outside the symplectic setting, for smooth diffeomorphism groups, was obtained in \cite{ht03, hrt13}. However, the constructions there do not appear to adapt well to the symplectic setting, and our proof is rather different.
\end{remark}

\begin{remark}
\label{rem:cannot_choose_right_inverse_to_map_id_to_id}
It is not possible to find a smooth local right inverse $\Psi$ of $\Phi$ such that $\Psi(\operatorname{id}) = (\operatorname{id},\dots,\operatorname{id})$. The reason is that the linearization of $\Phi$ at $(\operatorname{id},\dots,\operatorname{id})$ vanishes.
\end{remark}

Let $G$ be a group with commutator subgroup $[G,G]$. The \textit{commutator length}
\begin{equation*}
\operatorname{cl} : [G,G] \rightarrow \BZ_{\geq 0}
\end{equation*}
assigns to each element $g \in [G,G]$ the minimum number of commutators required to express $g$ as a product of commutators. For certain closed symplectic manifolds $(M,\omega)$, the commutator length is known to be unbounded on $\operatorname{Ham}(M,\omega)$; for instance, the existence of a non-trivial homogeneous quasimorphism implies such unboundedness. See, for example, \cite{ep03, ent04, gg04}. By contrast, Theorem~\ref{thm:smooth_perfectness_closed} shows that the commutator length is locally bounded with respect to the $C^\infty$ topology on $\operatorname{Ham}(M,\omega)$.

We recall that a map $q:G \rightarrow \BR$ defined on some group $G$ is called a \textit{quasimorphism} if there exists a constant $C\geq 0$ such that
\begin{equation*}
|q(ab) - q(a) - q(b)| \leq C \qquad \text{for all $a,b\in G$.}
\end{equation*}
It is called \textit{homogeneous} if in addition
\begin{equation*}
q(a^n) = nq(a) \qquad \text{for all $a \in G$ and $n\in \BZ$.}
\end{equation*}
Quasimorphisms have attracted considerable attention in symplectic geometry; see \cite{ent14} and \cite[\S3]{pr14} for surveys. There are numerous constructions of quasimorphisms, both with and without the use of Floer theory. As a consequence of Theorem~\ref{thm:smooth_perfectness_closed}, we obtain the following automatic continuity result.

\begin{corollary}[Automatic continuity]
\label{cor:homogeneous_quasimorphisms_continuity}
Let $(M,\omega)$ be a closed symplectic manifold. Then every homogeneous quasimorphism $q : \operatorname{Ham}(M,\omega)\rightarrow \BR$ is continuous with respect to the $C^\infty$ topology on $\operatorname{Ham}(M,\omega)$.
\end{corollary}

Now suppose that $(M,\omega)$ is a connected open symplectic manifold. In this case, the compactly supported Hamiltonian diffeomorphism group $\operatorname{Ham}_c(M,\omega)$ is neither simple nor perfect. Indeed, for open $M$ there exists a surjective group homomorphism
\begin{equation*}
\operatorname{Cal} : \operatorname{Ham}_c(M,\omega) \rightarrow \BR/\Lambda_\omega
\end{equation*}
called the \textit{Calabi homomorphism}. Here $\Lambda_\omega \subset \BR$ is a certain subgroup which vanishes whenever $(M,\omega)$ is exact; see Section~\ref{sec:preliminaries}. Banyaga \cite{ban78} proved that the commutator subgroup of $\operatorname{Ham}_c(M,\omega)$ agrees with the kernel of the Calabi homomorphism, which we abbreviate by $\operatorname{Ham}_c^0(M,\omega)$. Moreover, he proved that $\operatorname{Ham}_c^0(M,\omega)$ is perfect and simple.

We topologize the group $\operatorname{Diff}_c(M)$ of compactly supported diffeomorphisms of $M$ as the direct limit of the groups $\operatorname{Diff}_K(M)$, where $K$ ranges over all compact subsets of $M$, and $\operatorname{Diff}_K(M)$ denotes the group of diffeomorphisms supported in $K$, equipped with the $C^\infty$ topology. In favourable cases---such as when $(M,\omega)$ is exact---we equip $\operatorname{Ham}_c(M,\omega)$ and $\operatorname{Ham}_c^0(M,\omega)$ with the subspace topology inherited from $\operatorname{Diff}_c(M)$. For the general definition of the topologies on these groups, we refer to Section~\ref{sec:preliminaries}. The subtlety in the general case is related to the potential non-discreteness of $\Lambda_\omega$.

Note that Theorem~\ref{thm:smooth_perfectness_closed} is no longer valid if we replace $\operatorname{Ham}(M,\omega)$ by $\operatorname{Ham}_c^0(M,\omega)$ for some connected open symplectic manifold $(M,\omega)$. The reason is related to Remark~\ref{rem:cannot_choose_right_inverse_to_map_id_to_id}. Assume by contradiction that there is a local smooth right inverse $\Psi$ defined in an open neighbourhood of the identity in $\operatorname{Ham}_c^0(M,\omega)$. Write $\Psi(\operatorname{id}) = (u_1,v_1,\dots,u_m,v_m)$. Let $K\subset M$ be a compact subset containing the supports of all $u_i$ and $v_i$. Now consider an isotopy $\varphi_t$ in $\operatorname{Ham}_c^0(M,\omega)$ with $\varphi_0 = \operatorname{id}$. Since the linearization of $\Phi$ vanishes at $(\operatorname{id},\dots,\operatorname{id})$ and can moreover be computed locally in $M$, we see that $\partial_t|_{t=0} \Phi\circ\Psi(\varphi_t)$ vanishes outside $K$. But since $\Psi$ is assumed to be a right inverse of $\Phi$, we have $\Phi\circ \Psi(\varphi_t) = \varphi_t$. Now simply pick $\varphi_t$ such that $\partial_t|_{t=0}\varphi_t$ does not vanish outside $K$. This yields a contradiction.

In order to obtain an analogue of Theorem~\ref{thm:smooth_perfectness_closed} for open symplectic manifolds $(M,\omega)$, we therefore also need to consider Hamiltonian diffeomorphisms which are not compactly supported. We let $\operatorname{Ham}(M,\omega)$ denote the group of all Hamiltonian diffeomorphisms of $(M,\omega)$, not necessarily compactly supported. Elements $\varphi \in \operatorname{Ham}(M,\omega)$ are precisely the time-$1$ maps of Hamiltonian isotopies $(\varphi_H^t)_{t\in [0,1]}$ generated by Hamiltonians on $M$ whose induced Hamiltonian vectors fields $X_H$ are complete.

As before, smoothness of maps between diffeomorphism groups is understood in the diffeological sense; see Section~\ref{sec:preliminaries}.

\begin{theorem}[Smooth perfectness for open manifolds]
\label{thm:smooth_perfectness_open}
Let $n>0$ be an integer. Then there exists an integer $m>0$ such that the following is true for every connected open symplectic manifold $(M^{2n},\omega)$: There exist Hamiltonian diffeomorphisms $\varphi_1,\psi_1,\dots,\varphi_m,\psi_m \in \operatorname{Ham}(M,\omega)$, not necessarily compactly supported, which satisfy $\prod_j[\varphi_j,\psi_j] = \operatorname{id}_M$ and the following property: Consider the map
\begin{equation}
\label{eq:smooth_perfectness_open_map_Phi}
\Phi:\operatorname{Ham}_c^0(M,\omega)^{2m} \rightarrow \operatorname{Ham}_c^0(M,\omega) \quad (u_1,v_1,\dots,u_m,v_m)\mapsto \prod\limits_{j=1}^m [\varphi_j u_j,\psi_j v_j].
\end{equation}
Then there exist an open neighbourhood $\MN\subset \operatorname{Ham}_c^0(M,\omega)$ of the identity and a smooth map
\begin{equation*}
\Psi: \MN \rightarrow \operatorname{Ham}_c^0(M,\omega)^{2m}
\end{equation*}
which is a local right inverse of $\Phi$, that is $\Phi\circ \Psi = \operatorname{id}_\MN$, and satisfies $\Psi(\operatorname{id}) = (\operatorname{id},\dots,\operatorname{id})$.

Moreover, given an arbitrarily small open neighbourhood $\MU \subset \operatorname{Ham}_c^0(M,\omega)$ of the identity, it is possible to choose $\varphi_j$ and $\psi_j$ such that they are $C^\infty_{\operatorname{loc}}$ limits of elements of $\MU$.
\end{theorem}

We will also deduce the following corollary from Theorem~\ref{thm:smooth_perfectness_open}.

\begin{corollary}
\label{cor:smooth_perfectness_open}
Let $n>0$ be an integer. Then there exists an integer $m>0$ such that the following is true: Let $(M^{2n},\omega)$ be a connected open symplectic manifold and consider the map
\begin{equation*}
\Phi:\operatorname{Ham}_c^0(M,\omega)^{2m} \rightarrow \operatorname{Ham}_c^0(M,\omega) \quad (u_1,v_1,\dots,u_m,v_m)\mapsto \prod\limits_{j=1}^m [u_j,v_j].
\end{equation*}
Let $U\Subset M$ be a relatively compact open subset. Then there exist an open neighbourhood $\MN\subset \operatorname{Ham}_c^0(U,\omega)$ of the identity and a smooth map
\begin{equation*}
\Psi: \MN \rightarrow \operatorname{Ham}_c^0(M,\omega)^{2m}
\end{equation*}
which is a local right inverse of $\Phi$, that is $\Phi\circ \Psi = \operatorname{id}_{\MN}$, and which takes values in the closure of $\operatorname{Ham}_c^0(U,\omega)^{2m}$ inside $\operatorname{Ham}_c^0(M,\omega)^{2m}$. Moreover, we can choose $\MN$ and $\Psi$ such that $\Psi(\operatorname{id})$ is arbitrarily close to the tuple $(\operatorname{id},\dots,\operatorname{id})$.
\end{corollary}

In our work \cite{edt} on symplectic packing stability for symplectic manifolds with boundary, it is essential to have a version of Theorems~\ref{thm:smooth_perfectness_closed} and~\ref{thm:smooth_perfectness_open} for symplectic manifolds with boundary and Hamiltonian diffeomorphisms which are not required to restrict to the identity on the boundary. This is a slightly non-standard notion, but it has appeared before in the literature; see for example \cite[\S 2]{abhs18}.

Let $(M,\omega)$ be a connected symplectic manifold, possibly with boundary. We allow both $M$ and the boundary $\partial M$ to be non-compact. Consider a compactly supported Hamiltonian $H: [0,1]\times M\rightarrow \BR$ which vanishes on the boundary $\partial M$. Then the induced Hamiltonian vector field $X_H$ is tangent to the boundary and there is a well-defined Hamiltonian flow $(\varphi_H^t)_{t\in [0,1]}$. We call the time-one map $\varphi_H^1$ of such a Hamiltonian flow generated by a compactly supported Hamiltonian vanishing on the boundary $\partial M$ a \emph{compactly supported Hamiltonian diffeomorphism} of $(M,\omega)$. As usual, the set of such diffeomorphisms forms a group, which is denoted by $\operatorname{Ham}_c(M,\omega)$.

If $M$ is not closed, it is possible to define the Calabi homomorphism $\operatorname{Cal}:\operatorname{Ham}_c(M,\omega) \rightarrow \BR/\Lambda_\omega$ as in the open case without boundary. Again, we abbreviate its kernel by $\operatorname{Ham}_c^0(M,\omega)$. For closed $M$, it will be convenient to use the convention $\operatorname{Ham}_c^0(M,\omega) \coloneqq \operatorname{Ham}(M,\omega)$. We topologize $\operatorname{Ham}_c(M,\omega)$ and $\operatorname{Ham}_c^0(M,\omega)$ in a manner similar to the case of an open manifold without boundary; see Section~\ref{sec:preliminaries}.

As in the case of open symplectic manifolds, it is useful to define the group $\operatorname{Ham}(M,\omega)$ of all Hamiltonian diffeomorphisms which are not necessarily compactly supported. These are the time-$1$ maps of Hamiltonian isotopies $(\varphi_H^t)_{t\in [0,1]}$ generated by Hamiltonians $H$, not necessarily compactly supported, which vanish on the boundary of $M$ and have a complete Hamiltonian vector field $X_H$.

For more details concerning Hamiltonian diffeomorphisms of symplectic manifolds with boundary, we refer to Section~\ref{sec:preliminaries}.

With these definitions out of the way, we can formulate the following result, which simultaneously generalizes Theorems~\ref{thm:smooth_perfectness_closed} and~\ref{thm:smooth_perfectness_open}.

\begin{theorem}[Smooth perfectness for manifolds with boundary]
\label{thm:smooth_perfectness_general}
The assertion of Theorem~\ref{thm:smooth_perfectness_open} continues to hold for all connected symplectic manifolds $(M,\omega)$ which are allowed to be compact or non-compact and to have empty or non-empty boundary.
\end{theorem}

\begin{corollary}
\label{cor:smooth_perfectness_general}
The assertion of Corollary~\ref{cor:smooth_perfectness_open} continues to hold for all connected symplectic manifolds $(M,\omega)$ which are allowed to be compact or non-compact and to have empty or non-empty boundary.
\end{corollary}

\begin{remark}
If the boundary $\partial M$ is non-empty, then the group $\operatorname{Ham}_c^0(M,\omega)$ is clearly not simple since the subgroup of all diffeomorphisms restricting to the identity on the boundary is a non-trivial normal subgroup. Perfectness of $\operatorname{Ham}_c^0(M,\omega)$ is a straightforward consequence of Corollary \ref{cor:smooth_perfectness_general} above. In the non-conservative setting, a perfectness result for diffeomorphism groups of manifolds with boundary was obtained by Rybicki in \cite{ryb98}. Our methods are completely different from \cite{ryb98}, and to the best of our knowledge, it was unknown whether $\operatorname{Ham}_c^0(M,\omega)$ is perfect if $\partial M$ is non-empty.
\end{remark}

\subsection{Motivation and applications}

As we have already indicated, an important motivation for the results in this paper comes from our work \cite{edt}, where we establish that the phenomenon of \emph{symplectic ball packing stability}---first discovered by Biran \cite{bir99} for closed symplectic $4$-manifolds---holds much more generally for compact connected symplectic $4$-manifolds with smooth boundary. More precisely, in \cite{edt} we prove that for every compact connected symplectic $4$-manifold $(M^4,\omega)$ with smooth boundary, there exists a finite threshold $n_0>0$ such that, for every integer $n\geq n_0$, the volume of $(M,\omega)$ can be filled, up to arbitrarily small error, by $n$ disjoint symplectically embedded balls of equal size. The key novelty of this result is its applicability to arbitrary smooth symplectic boundary, whereas most known sharp symplectic embedding results apply only to closed manifolds or highly symmetric toric domains, such as symplectic ellipsoids.

A central insight of \cite{edt} is that symplectic packing stability in the presence of boundary is closely linked to the simplicity and perfectness of Hamiltonian diffeomorphism groups. Simplicity of the Hamiltonian diffeomorphism group implies that for any non-trivial Hamiltonian diffeomorphism $\varphi$ (even one that is dynamically simple, such as a rotation), any other, however dynamically complicated, Hamiltonian diffeomorphism can be expressed as a composition of conjugates of $\varphi$ and $\varphi^{-1}$. Roughly speaking, in \cite{edt} we use this observation to decompose a given symplectic manifold with boundary, whose characteristic foliation is potentially dynamically complicated, into balls, whose characteristic foliation on the boundary is extremely simple, namely the Hopf fibration.

To implement this strategy in \cite{edt}, quantitiative perfectness results are essential. Specifically, we require a uniform bound on the number of commutators needed to factorize a Hamiltonian diffeomorphism near the identity, and we also need these commutators to remain close to the identity themselves. Furthermore, it is important to handle Hamiltonian diffeomorphisms on manifolds with boundary that do not restrict to the identity on the boundary. This necessity provides motivation for the additional work required to establish Theorem~\ref{thm:smooth_perfectness_general} and Corollary~\ref{cor:smooth_perfectness_general}.

We refer the reader to \cite{edt} for further details on symplectic packing stability, our symplectic embedding constructions, and applications to symplectic Weyl laws and in particular to their subleading asymptotics.

\subsection{Overview and structure of the paper}

The first perfectness result for a smooth diffeomorphism group was obtained by Herman in \cite{her71}. Relying on an implicit function theorem for Fr\'echet spaces \cite{hs71}, he proved that the group $\operatorname{Diff}_0(T^n)$ of smooth diffeomorphisms of the $n$-dimensional torus $T^n$ isotopic to the identity is perfect. Later, Thurston \cite{thu74} used this result to deduce perfectness of $\operatorname{Diff}_0(M)$ for arbitrary $M$. In his influential paper \cite{ban78}, Banyaga adapted the Herman--Thurston approach to perfectness to the symplectic setting and showed that the group $\operatorname{Ham}(M,\omega)$ of Hamiltonian diffeomorphisms of a closed symplectic manifold is perfect. He also proved perfectness of the kernel of the Calabi homomorphism on an open symplectic manifold.

An independent proof of perfectness for diffeomorphism groups was found by Mather; see \cite{mat74, mat75, eps84}. However, as pointed out by Banyaga in \cite[\S 2, p.\ 24]{ban97}, Mather's approach does not work in the conservative setting, so it is not useful for our purposes.

The proof of our smooth perfectness results (Theorems~\ref{thm:smooth_perfectness_closed}, \ref{thm:smooth_perfectness_open}, \ref{thm:smooth_perfectness_general}) broadly follows the Herman--Thurston approach and Banyaga's adaptation to the symplectic setting; see \cite[\S 2 \& \S4]{ban97} for an exposition. Due to the delicate nature of the arguments involved, a significant amount of care is necessary to deduce our refined perfectness results.

It is relatively straightforward to prove Theorem~\ref{thm:smooth_perfectness_closed} in the case that $(M,\omega)$ is the torus $T^{2n}$ with the standard symplectic form $\omega_0$. This more or less follows from the observation that since the commutators in Herman's argument \cite{her71} for the torus are found using an implicit function theorem for Fr\'echet spaces, their smooth dependence on the diffeomorphism in question essentially comes for free. We explain this in more detail in Section~\ref{sec:application_herman_sergeraert_diophantine_torus_rotations}.

It is quite a bit more subtle to go from $T^{2n}$ to arbitrary closed or open symplectic manifolds because Thurston's and Banyaga's proofs are based on technically difficult fragmentation arguments. In addition, it is essential in our work to carefully keep track of the number of commutators we produce and their smooth dependence on the initial diffeomorphism.

The case of symplectic manifolds with boundary poses another challenge: It is no longer possible to reduce perfectness to the case of the torus. We overcome this difficulty by establishing a variant of the results of Herman \cite{her71} and Herman--Sergeraert \cite{hs71} for Hamiltonian diffeomorphisms of annuli. This relies on an application of the Nash--Moser implicit function theorem.\\

The rest of this paper is structured as follows.

In \S\ref{sec:preliminaries} we discuss preliminaries on Hamiltonian diffeomorphisms, such as the Calabi homomorphism and generating functions. We also introduce the group of Hamiltonian diffeomorphisms of a symplectic manifold with boundary.

In \S\ref{sec:reduction_to_balls} we formulate Theorem~\ref{thm:smooth_perfectness_pairs_of_balls}. This theorem is a weakened version of Corollaries~\ref{cor:smooth_perfectness_open} and~\ref{cor:smooth_perfectness_general} with $(M,U)$ given by a pair of open balls or half balls. We show that all our main results stated in the introduction can be deduced from Theorem~\ref{thm:smooth_perfectness_pairs_of_balls}. This involves establishing two fragmentation results for Hamiltonian diffeomorphisms (Lemmas~\ref{lem:fragmentation_basic} and~\ref{lem:fragmentation_with_Calabi_control}) which are also used in later sections.

Thurston's and Banyaga's arguments for reducing perfectness to the case of the torus involve the local homology of Hamiltonian diffeomorphism groups \cite[\S 2, p.\ 30-32]{ban97}. This can be viewed as a bookkeeping technique for the diffeomorphisms produced by various fragmentations. In the present paper, it will be more convenient to keep track of diffeomorphisms in a related but slightly different way. This involves cocycles on triangulated surfaces with values in Hamiltonian diffeomorphism groups. We introduce this formalism in \S\ref{sec:commutators_and_surfaces}.

In \S\ref{sec:application_herman_sergeraert_diophantine_torus_rotations} we establish Theorem~\ref{thm:smooth_perfectness_closed} in the case of the torus (Corollary~\ref{cor:smooth_perfectness_torus}). We also establish a corresponding more technical result (Proposition~\ref{prop:cocycles_torus}) formulated in the language of cocycles and triangulated surfaces.

In \S\ref{sec:fragmentation} we refine Proposition~\ref{prop:cocycles_torus} via constructions involving fragmentations of Hamiltonian diffeomorphisms. This is the content of Proposition~\ref{prop:cocycles_torus_fragmented}.

After constructing certain open covers of symplectic manifolds subject to technical conditions in \S\ref{sec:special_open_covers}, in \S\ref{sec:smooth_perfectness_full_balls} we use Proposition~\ref{prop:cocycles_torus_fragmented} to prove Theorem~\ref{thm:smooth_perfectness_pairs_of_balls} in the case of pairs of full balls. This completes the proof of all main results in the introduction for symplectic manifolds without boundary.

At this point, it remains to prove Theorem~\ref{thm:smooth_perfectness_pairs_of_balls} for pairs of half balls. To this end, we prove a variant of the Herman--Sergeraert theorem \cite{hs71} for symplectic manifolds with boundary in \S\ref{sec:herman_sergeraert_boundary}. This is the content of Theorem~\ref{thm:herman_with_boundary}. Our proof of Theorem~\ref{thm:herman_with_boundary} involves the Nash--Moser implicit function theorem. Most of the work in \S\ref{sec:herman_sergeraert_boundary} goes towards verifying the assumptions in the Nash--Moser theorem, in particular constructing right inverses for the linearized problem.

Finally, in \S\ref{sec:smooth_perfectness_half_balls}, we use Theorem~\ref{thm:herman_with_boundary} to complete the proof of Theorem~\ref{thm:smooth_perfectness_pairs_of_balls} for pairs of half balls.\\

\paragraph{\textbf{Acknowledgements}}

The author thanks Vincent Humili\`ere and Sobhan Seyfaddini for helpful comments, in particular suggesting Corollary \ref{cor:homogeneous_quasimorphisms_continuity}. He is supported by Dr.\ Max R\"{o}ssler, the Walter Haefner Foundation, and the ETH Z\"{u}rich Foundation.

\section{Preliminaries on Hamiltonian diffeomorphism groups}
\label{sec:preliminaries}

We review some preliminary material concerning Hamiltonian diffeomorphism groups and fix notation and conventions. We refer to \cite[\S 10]{ms17} or \cite[\S 4]{ban97} for more details.

\subsection{Hamiltonian diffeomorphism groups and their topology}

Let $(M^{2n},\omega)$ be a symplectic manifold. For now, we allow $M$ to be open or closed and assume that it does not have boundary. Every compactly supported Hamiltonian $H:[0,1]\times M \rightarrow \BR$ induces a time-dependent Hamiltonian vector field $X_H$. We adopt the sign convention that $X_H$ is characterized by the identity
\begin{equation*}
\iota_{X_{H_t}} \omega = dH_t.
\end{equation*}
The compactly supported Hamiltonian isotopy generated by $X_H$ is denoted by $(\varphi_H^t)_{t\in [0,1]}$ and the group of compactly supported Hamiltonian diffeomorphisms by $\operatorname{Ham}_c(M,\omega)$. If there is no danger of confusion, we sometimes drop the symplectic form from the notation and simply write $\operatorname{Ham}_c(M)$.

If $M$ is closed, then $\operatorname{Ham}_c(M,\omega)$ agrees with the group $\operatorname{Ham}(M,\omega)$ of all Hamiltonian diffeomorphisms. By default, we always endow $\operatorname{Ham}(M,\omega)$ with the $C^\infty$ topology, unless stated otherwise. With respect to this topology, the group $\operatorname{Ham}(M,\omega)$ is locally contractible. This follows from the local contractibility---also with respect to the $C^\infty$ topology---of the symplectomorphism group $\operatorname{Symp}(M,\omega)$ and the $C^1$ flux conjecture, which was resolved in full generality by Ono \cite{ono06}.

If $M$ is open, we topologize the compactly supported diffeomorphism group $\operatorname{Diff}_c(M)$ as the colimit of all $\operatorname{Diff}_K(M)$, where $K$ ranges over compact subsets of $M$ and $\operatorname{Diff}_K(M)$ is the group of diffeomorphisms supported in $K$, equipped with the $C^\infty$ topology. The induced subspace topology makes $\operatorname{Symp}_c(M,\omega)$ locally contractible. Unfortunately, we do not know whether $\operatorname{Ham}_c(M,\omega)$ equipped with the subspace topology is locally contractible in general. The reason is that we do not know whether the flux group $\Gamma_\omega$ showing up in the exact sequence
\begin{equation*}
1 \rightarrow \operatorname{Ham}_c(M,\omega) \rightarrow \operatorname{Symp}_{c,0}(M,\omega) \overset{\operatorname{Flux}}{\longrightarrow} H^1_c(M;\BR)/\Gamma_\omega \rightarrow 0
\end{equation*}
induced by the flux homomorphism is discrete for a general open symplectic manifold $(M,\omega)$. Note that this is the case for exact symplectic manifolds, where the flux group $\Gamma_\omega$ vanishes. In general, we use the following workaround to obtain a topology making $\operatorname{Ham}_c(M,\omega)$ locally contractible. Let
\begin{equation*}
\pi: \widetilde{\operatorname{Symp}}_{c,0}(M,\omega) \rightarrow \operatorname{Symp}_{c,0}(M,\omega)
\end{equation*}
be the universal cover of the identity component of the compactly supported symplectomorphism group. We define
\begin{equation*}
\widetilde{\operatorname{Ham}}_c(M,\omega) \coloneqq \operatorname{ker} (\widetilde{\operatorname{Flux}}:\widetilde{\operatorname{Symp}}_{c,0}(M,\omega) \rightarrow H^1_c(M;\BR))
\end{equation*}
to be the kernel of the flux homomorphism on the level of the universal cover. We equip $\widetilde{\operatorname{Ham}}_c(M,\omega)$ with the subspace topology. This makes $\widetilde{\operatorname{Ham}}_c(M,\omega)$ locally contractible. The image of $\widetilde{\operatorname{Ham}}_c(M,\omega)$ under the covering map $\pi$ is exactly $\operatorname{Ham}_c(M,\omega)$. Now equip $\operatorname{Ham}_c(M,\omega)$ with the quotient topology. This turns $\widetilde{\operatorname{Ham}}_c(M,\omega)$ into the universal cover of $\operatorname{Ham}_c(M,\omega)$.

Consider $\varphi \in \operatorname{Ham}_c(M,\omega)$. Then a diffeomorphism $\psi \in \operatorname{Ham}_c(M,\omega)$ is close to $\varphi$ with respect to the topology on $\operatorname{Ham}_c(M,\omega)$ precisely if it is close to $\varphi$ with respect to the topology on $\operatorname{Diff}_c(M)$ and if it can be connected to $\varphi$ through Hamiltonian diffeomorphisms in $\operatorname{Ham}_c(M,\omega)$ which stay close to $\varphi$ with respect to the topology on $\operatorname{Diff}_c(M)$.

Unless stated otherwise, we always equip $\operatorname{Ham}_c(M,\omega)$ with this topology.

\subsection{Symplectic manifolds with boundary}

As already mentioned in the introduction, for our applications in \cite{edt} it is important to also consider Hamiltonian diffeomorphisms of symplectic manifolds with boundary. Let $(M^{2n},\omega)$ be a symplectic manifold with boundary. We allow $M$ to be compact or non-compact. Consider a compactly supported Hamiltonian $H:[0,1]\times M \rightarrow \BR$ which vanishes on the boundary $\partial M$. The condition that $H$ vanishes on $\partial M$ implies that the induced Hamiltonian vector field $X_H$ is tangent to $\partial M$. In fact, it is tangent to the characteristic foliation on $\partial M$ induced by the symplectic form $\omega$. We call the induced isotopy $(\varphi_H^t)_{t\in [0,1]}$ a \emph{compactly supported Hamiltonian isotopy} and the time-$1$ map $\varphi_H^1$ a \emph{compactly supported Hamiltonian diffeomorphism}. The set of all such compactly supported Hamiltonian diffeomorphisms is denoted by $\operatorname{Ham}_c(M,\omega)$. As in the usual case without boundary, one can check that this forms a group. If $M$ is compact, we also write $\operatorname{Ham}(M,\omega) = \operatorname{Ham}_c(M,\omega)$.

Note that Hamiltonian diffeomorphisms $\varphi\in \operatorname{Ham}_c(M,\omega)$ must preserve characteristic leaves on the boundary $\partial M$ since the Hamiltonian vector fields $X_H$ we consider are tangent to the characteristic foliation.

\begin{example}
Consider the closed unit disk $\BD \subset \BR^2$ equipped with the standard area form $\omega$. Then the group $\operatorname{Ham}(\BD,\omega)$ is simply the group $\operatorname{Diff}(\BD,\omega)$ of all diffeomorphisms of $\BD$ (not necessarily equal to the identity on the boundary) which preserve the area form $\omega$.
\end{example}

\begin{example}
Let $\BA \coloneqq [0,1] \times\BT$ denote the annulus, where $\BT \coloneqq \BR/\BZ$ is the circle. Equip $\BA$ with the area form $\omega = dx \wedge dy$, where $(x,y)$ are the coordinates on $\BA$. Consider the group $\operatorname{Diff}_0(\BA,\omega)$, which is the identity component of the group of all diffeomorphisms of $\BA$ preserving $\omega$. There is a natural flux homomorphism
\begin{equation*}
\operatorname{Flux}:\operatorname{Diff}_0(\BA,\omega) \rightarrow \BT
\end{equation*}
which can be described as follows. Let $\gamma$ denote the oriented line segment $[0,1]\times \left\{ 0 \right\} \subset \BA$. Given $\varphi \in \operatorname{Diff}_0(\BA,\omega)$, pick a $2$-chain $Z$ in $\BA$ such that $\partial Z + \gamma - \varphi(\gamma)$ is a $1$-chain in $\partial \BA$. Then the flux of $\varphi$ is given by
\begin{equation*}
\operatorname{Flux}(\varphi) = \left[\int_Z \omega\right] \in \BT.
\end{equation*}
This is well-defined because any two choices of $Z$ differ by a cycle representing an element of $H_2(\BA,\partial\BA)$. The Hamiltonian diffeomorphism group $\operatorname{Ham}(\BA,\omega)$ is simply given by the kernel of $\operatorname{Flux}$.
\end{example}

Similarly to the open case without boundary, we topologize $\operatorname{Ham}_c(M,\omega)$ such that a small open neighbourhood of a Hamiltonian diffeomorphism $\varphi \in \operatorname{Ham}_c(M,\omega)$ consists of those $\psi \in \operatorname{Ham}_c(M,\omega)$ which are contained in a small open neighbourhood of $\varphi$ in $\operatorname{Diff}_c(M)$ and which can be connected to $\varphi$ via Hamiltonian diffeomorphisms in $\operatorname{Ham}_c(M,\omega)$ which stay near $\varphi$ with respect to the topology on $\operatorname{Diff}_c(M,\omega)$.

\subsection{Calabi homomorphism}

For open $M$ without boundary, the \textit{Calabi homomorphism} \cite{cal69} \cite[p.~102]{ban97} is a surjective group homomorphism
\begin{equation*}
\widetilde{\operatorname{Cal}} : \widetilde{\operatorname{Ham}}_c(M,\omega) \rightarrow \BR.
\end{equation*}
The value of the Calabi homomorphism on an element of $\widetilde{\operatorname{Ham}}_c(M,\omega)$ represented by a compactly suported Hamiltonian isotopy $(\varphi_H^t)_{t\in [0,1]}$ is given by
\begin{equation}
\label{eq:calabi_homomorphism_definition}
\widetilde{\operatorname{Cal}}([(\varphi_H^1)_{t\in [0,1]}]) = \int_{[0,1]\times M} H dt \wedge \omega^n.
\end{equation}
This definition extends to manifolds $M$ with boundary.

\begin{lemma}
\label{lem:calabi_with_boundary}
Let $(M,\omega)$ be a symplectic manifold which is allowed to have boundary and to be compact or non-compact. Assume that $M$ has no closed components. Then equation~\eqref{eq:calabi_homomorphism_definition} yields a well-defined group homomorphism $\widetilde{\operatorname{Cal}} : \widetilde{\operatorname{Ham}}_c(M,\omega) \rightarrow \BR$.
\end{lemma}

This can be shown by an argument similar to the one in the case without boundary. For completeness, we sketch a proof.

\begin{proof}
The essential point is to show that the value in equation~\eqref{eq:calabi_homomorphism_definition} is independent of the choice of representative $(\varphi_H^t)_{t\in [0,1]}$ of an element of $\widetilde{\operatorname{Ham}}_c(M,\omega)$. Once this is established, it is straightforward to check that \eqref{eq:calabi_homomorphism_definition} defines a group homomorphism.

Consider a family $(H^s)_{s\in [0,1]}$ of compactly supported Hamiltonians vanishing on the boundary such that $\varphi_{H^s}^1$ is independent of $s$. We need to verify that the integrals $\int_{[0,1]\times M} H^s dt \wedge \omega^n$ are also independent of $s$.

Set $\varphi_{t,s} \coloneqq \varphi_{H^s}^t$. This defines a $2$-parameter family of Hamiltonian diffeomorphisms parametrized by $(s,t) \in [0,1]^2$. Let $X_{t,s}$ and $Y_{t,s}$ be the vector fields characterized by $\partial_t\varphi_{t,s} = X_{t,s}\circ \varphi_{t,s}$ and $\partial_s\varphi_{t,s} = Y_{t,s}\circ \varphi_{t,s}$. By \cite[Prop. 3.1.5]{ban97} we have\footnote{Following \cite{ms17} we use the sign convention $[U,V] = -\ML_UV$ for the Lie bracket, which does not match the sign convention in \cite{ban97}.}
\begin{equation}
\label{eq:lie_bracket_vector_fields}
[X_{t,s},Y_{t,s}] = \partial_t Y_{t,s} - \partial_s X_{t,s}.
\end{equation}
For each $(t,s)$, there exists a unique compactly supported function $G_{t,s}$ on $M$ which vanishes on the boundary $\partial M$ and satisfies $X_{G_{t,s}} = Y_{t,s}$. We also have $X_{H_{t,s}} = X_{t,s}$, where we write $H_{t,s} = H^s(t,\cdot)$. It follows from identity~\eqref{eq:lie_bracket_vector_fields} that
\begin{equation*}
\left\{ H_{t,s},G_{t_s} \right\} = \partial_t G_{t,s} - \partial_s H_{t,s}.
\end{equation*}
We therefore obtain
\begin{eqnarray*}
\partial_s \int_{[0,1]\times M} H_{t,s} dt \wedge \omega^n &=& \int_{[0,1]\times M} (\partial_t G_{t,s} - \left\{ H_{t,s},G_{t,s} \right\}) dt \wedge \omega^n\\
&=& \int_M (G_{1,s} - G_{0,s}) \omega^n - \int_{[0,1]\times M} \left\{ H_{t,s}, G_{t,s} \right\} dt \wedge \omega^n \\
&=& -\int_{[0,1]\times M} \left\{ H_{t,s}, G_{t,s} \right\} dt \wedge \omega^n.
\end{eqnarray*}
Here the last equality uses that $G_{0,s} = G_{1,s} = 0$ because both $\varphi_{0,s} = \operatorname{id}$ and $\varphi_{1,s} = \varphi_{H^s}^1$ are independent of $s$. We have
\begin{equation*}
\left\{ H_{t,s}, G_{t,s} \right\} \omega^n = n dH_{t,s} \wedge dG_{t,s} \wedge \omega^{n-1} = n d(H_{t,s} dG_{t,s}\wedge \omega^{n-1}),
\end{equation*}
see for example \cite[Exercise 3.19]{ms17}. Since $H_{t,s}$ is compactly supported and vanishes on the boundary of $M$, it follows from Stokes' theorem that $\int_M \left\{ H_{t,s},G_{t,s} \right\} \omega^n$ must vanish. We deduce that $\int_{[0,1]\times M} H_{t,s} dt\wedge \omega^n$ is independent of $s$, as desired.
\end{proof}

The Calabi homomorphism descends to a homomorphism
\begin{equation*}
\operatorname{Cal} : \operatorname{Ham}_c(M,\omega) \rightarrow \BR/\Lambda_\omega,
\end{equation*}
where $\Lambda_\omega$ is the image of $\pi_1(\operatorname{Ham}_c(M,\omega))$ under $\widetilde{\operatorname{Cal}}$. Not much is known about $\Lambda_\omega$, except that it vanishes whenever $(M,\omega)$ does not have boundary and is exact; see, for instance, \cite[Remark~3.11]{mcd10}. Suppose for now that $M$ is connected. Then we abbreviate the kernel of $\operatorname{Cal}$ by
\begin{equation*}
\operatorname{Ham}_c^0(M,\omega) \coloneqq \operatorname{ker}(\operatorname{Cal}: \operatorname{Ham}_c(M,\omega) \rightarrow \BR/\Lambda_\omega).
\end{equation*}
Since we do not know whether $\Lambda_\omega$ is discrete for general $(M,\omega)$, we also do not know whether $\operatorname{Ham}_c^0(M,\omega)$ is always locally contractible when equipped with the subspace topology induced by $\operatorname{Ham}_c(M,\omega)$. For this reason, we topologize $\operatorname{Ham}_c^0(M,\omega)$ in a similar manner as $\operatorname{Ham}_c(M,\omega)$. First, we define
\begin{equation*}
\widetilde{\operatorname{Ham}}_c^0(M,\omega) \coloneqq \operatorname{ker}(\widetilde{\operatorname{Cal}} : \widetilde{\operatorname{Ham}}_c(M,\omega) \rightarrow \BR)
\end{equation*}
and endow it with the subspace topology coming from $\widetilde{\operatorname{Ham}}_c(M,\omega)$. The image of $\widetilde{\operatorname{Ham}}_c^0(M,\omega)$ under the natural projection map $\widetilde{\operatorname{Ham}}_c(M,\omega) \rightarrow \operatorname{Ham}_c(M,\omega)$ is precisely given by $\operatorname{Ham}_c^0(M,\omega)$, and we equip the latter group with the quotient topology. This makes $\operatorname{Ham}_c^0(M,\omega)$ locally contractible and turns $\widetilde{\operatorname{Ham}}_c^0(M,\omega)$ into the universal cover of $\operatorname{Ham}_c^0(M,\omega)$. We always consider this topology on $\operatorname{Ham}_c^0(M,\omega)$, unless stated otherwise.

\begin{remark}
\label{rem:local_lift_of_calabi}
In this paper, we are mostly concerned with Hamiltonian diffeomorphisms in a small neighbourhood $\MN \subset \operatorname{Ham}_c(M,\omega)$ of the identity. In view of our construction of the topology on $\operatorname{Ham}_c(M,\omega)$, we can identify $\MN$ with a small neighbourhood of the identity in $\widetilde{\operatorname{Ham}}_c(M,\omega)$. Near the identity, we can therefore lift the Calabi homomorphism $\operatorname{Cal}:\operatorname{Ham}_c(M,\omega) \rightarrow \BR/\Lambda_\omega$ to a real-valued map
\begin{equation*}
\operatorname{Cal} : \MN \subset \operatorname{Ham}_c(M,\omega) \rightarrow \BR.
\end{equation*}
The zero set of this lift forms an open neighourhood of the identity in $\operatorname{Ham}_c^0(M,\omega)$.
\end{remark}

If $M$ is disconnected, we define $\operatorname{Ham}_c^0(M,\omega)$ to be the subgroup of $\operatorname{Ham}_c(M,\omega)$ consisting of all compactly supported Hamiltonian diffeomorphisms whose restriction to any connected component $M'$ of $M$ is contained in $\operatorname{Ham}_c^0(M',\omega)$.

\subsection{Diffeology}

In this paper, a map between spaces of compactly supported diffeomorphisms, for example a map $\Phi: \operatorname{Ham}_c(M,\omega) \rightarrow \operatorname{Ham}_c(N,\sigma)$, is understood to be smooth if it is smooth in the diffeological sense. This means that $\Phi$ maps each smooth finite-parameter family $(\varphi_z)_{z\in \BR^k}$ in $\operatorname{Ham}_c(M,\omega)$ to a smooth finite-parameter family $(\Phi(\varphi_z))_{z\in \BR^k}$ in $\operatorname{Ham}_c(N,\sigma)$. Here a finite-parameter family $(\varphi_z)_{z\in \BR^k}$ is called smooth if, for every compact subset $K\subset \BR^k$, there exists a compact subset $L \subset M$ such that $\operatorname{supp}(\varphi_z) \subset L$ for all $z \in K$ and if, moreover, the induced map
\begin{equation*}
\BR^k \times M \rightarrow M \qquad (z,p) \mapsto \varphi_z(p)
\end{equation*}
is smooth.

If $M$ is compact, then $\operatorname{Ham}(M,\omega)$ carries the structure of a Fr\'echet manifold and one can ask for a map to be smooth as a map between Fr\'echet manifolds; see e.g.\ \cite{ham82}. In our setting, this notion of smoothness is equivalent to smoothness in the diffeological sense. Note that if $M$ is not compact, then $\operatorname{Ham}_c(M,\omega)$ is not naturally a Fr\'echet manifold.

\subsection{Generating functions}
\label{subsec:generating_functions}

Let $(M^{2n},\omega)$ be a symplectic manifold, which may be open or closed. Assume for now that $M$ does not have boundary. Consider the product $M\times \overline{M} = (M,\omega)\times (M,-\omega)$. The diagonal $\Delta \subset M\times \overline{M}$ is a Lagrangian submanifold. By the Weinstein neighbourhood theorem, a neighbourhood $N$ of $\Delta$ in $M\times \overline{M}$ can be symplectically identified with a neighbourhood of the zero section in the cotangent bundle $T^*\Delta \cong T^*M$. Let us fix such an identification. Let $\lambda$ be the $1$-form on $N$ corresponding to the canonical Liouville $1$-form on $T^*M$. Moreover, let $\operatorname{pr}:N\rightarrow\Delta$ be the map corresponding to the projection onto the zero section in $T^*M$.

Via the above identifications, graphs $\operatorname{gr}(\varphi) \subset M\times \overline{M}$ of compactly supported diffeomorphisms $\varphi$ of $M$ close to the identity correspond to graphs of compactly supported $1$-forms $\alpha$ on $M$ close to zero. Here the space $\Omega_c^1(M)$ of compactly supported $1$-forms is topologized as the colimit of all $\Omega_K^1(M)$, where $K$ ranges over compact subsets of $M$ and $\Omega_K^1(M)$ denotes the space of $1$-forms with support in $K$, equipped with the $C^\infty$ topology. A $1$-form $\alpha$ is in correspondence with a diffeomorphism $\varphi$ if and only if $\lambda|_{\operatorname{gr}(\varphi)} = (\operatorname{pr}|_{\operatorname{gr}(\varphi)})^*\alpha$.

A diffeomorphism $\varphi$ is a symplectomorphism if and only if its graph is a Lagrangian submanifold. This happens exactly if the corresponding $1$-form $\alpha$ is closed. We therefore have a bijective correspondence between compactly supported symplectomorphisms $\varphi$ of $(M,\omega)$ in a small neighbourhood of the identity and compactly supported closed $1$-forms $\alpha$ on $M$ in a small neighbourhood of zero. Compactly supported Hamiltonian diffeomorphisms $\varphi\in \operatorname{Ham}_c(M,\omega)$ close to the identity (with respect to the topology on $\operatorname{Ham}_c(M,\omega)$ introduced above) are in bijective correspondence with compactly supported $1$-forms $\alpha$ close to zero which are exact, i.e.\ which arise as the differential $\alpha = dW$ of a compactly supported smooth function $W \in C^\infty_c(M)$. Such a function $W$ is called a \emph{generating function} of $\varphi$. A function $W \in C^\infty_c(M)$ is a generating function of $\varphi$ if and only if $W\circ \operatorname{pr}|_{\operatorname{gr}(\varphi)}$ is a primitive of $\lambda|_{\operatorname{gr}(\varphi)}$.

If $M$ is open, then compactly supported generating functions are uniquely determined and there is a bijective correspondence between $\varphi\in \operatorname{Ham}_c(M,\omega)$ close to the identity and $W\in C^\infty_c(M)$ close to zero. This bijective correspondence is smooth in the diffeological sense. If $M$ is closed, then $W$ is only determined up to a locally constant function.

Now consider the case that $(M^{2n},\omega)$ has boundary. The restriction $T^*M|_{\partial M}$ of the cotangent bundle to the boundary of $M$ has two natural subbundles: First, there is the conormal bundle $N^*\partial M$ of the boundary, which consists of all cotangent vectors vanishing on all vectors tangent to the boundary. This is a $1$-dimensional vector bundle and its total space is a Lagrangian submanifold of $T^*M$. Second, there is the bundle $E \subset T^*M|_{\partial M}$ consisting of all cotangent vectors vanishing on vectors tangent to the characteristic foliation. The bundle $E$ contains $N^*\partial M$ and has codimension $1$ inside $T^*M|_{\partial M}$. The total space of $E$ is a coisotropic submanifold of $T^* M$ of codimension $2$. The $2$-dimensional coisotropic leaves of $E$ are precisely the sets of the form $N^*\partial M|_{L}$ for characteristic leaves $L \subset \partial M$.

The product $\partial M \times \partial\overline{M}$ is a coisotropic submanifold of $M\times \overline{M}$ of codimension $2$. The Coisotropic leaves are precisely of the form $L_1\times L_2$ for characteristic leaves $L_1$ and $L_2$ in the boundary of $M$. We can choose a symplectic identification of an open neighbourhood $N \subset M\times \overline{M}$ of the diagonal $\Delta$ with a suitable subset of the cotangent bundle $T^*M \cong T^*\Delta$ containing the zero section such that points in $E$ correspond to points in $\partial M \times \partial \overline{M}$. Under this identification, points in $N^*\partial M$ correspond to points $(u,v)$ in $\partial M\times \partial\overline{M}$ near $\partial\Delta$ which have the property that, inside some small neighbourhood of $u$, it is possible to connected $v$ to $u$ by moving along the characteristic foliation.

Via this identification, compactly supported diffeomorphisms of $M$ close to the identity correspond to compactly supported $1$-forms $\alpha$ on $M$ close to zero whose restriction to $\partial M$ takes values in the subbundle $E$. The restriction of $\alpha$ to $\partial M$ takes values in $N^*\partial M$ exactly if $\varphi$ preserves characteristic leaves on $\partial M$ (i.e.\ maps each characteristic leaf to itself) and can be connected to the identity through such diffeomorphisms. Note that the restriction of $\alpha$ to $\partial M$ taking values in $N^*\partial M$ is equivalent to $\alpha$ vanishing on vectors tangent to $\partial M$.

As before, a diffeomorphism $\varphi$ is symplectic if and only if the corresponding $1$-form $\alpha$ is closed. Moreover, if $M$ has no closed components, we have a bijective correspondence between compactly supported Hamiltonian diffeomorphisms $\varphi\in \operatorname{Ham}_c(M,\omega)$ near the identity and compactly supported generating functions $W\in C^\infty_c(M)$ near zero which vanish on the boundary.

We conclude this subsection with some lemmas which will be useful later on.

\begin{lemma}
\label{lem:difference_of_generating_functions}
Let $(M,\omega)$ be a connected symplectic manifold. Let $U\subset V \subset M$ be open and suppose that the closure of $U$ is contained in $V$. Then the following is true for all $\varphi,\varphi' \in \operatorname{Ham}_c(M,\omega)$ contained in a sufficiently small neighbourhood of the identity: Let $W,W' \in C^\infty_c(M)$ be generating functions for $\varphi$ and $\varphi'$, respectively.
\begin{enumerate}
\item \label{item:difference_of_generating_functions_gen_funct_to_diff} Suppose that there exists a constant $C\in \BR$ such that $\operatorname{supp}(C+W-W') \subset U$. Then
\begin{equation*}
\varphi \varphi'^{-1} ,\enspace \varphi'^{-1}\varphi \in \operatorname{Ham}_c(V).
\end{equation*}
\item \label{item:difference_of_generating_functions_diff_to_gen_funct} Suppose that $\varphi \varphi'^{-1} \in \operatorname{Ham}_c(U)$ or $\varphi'^{-1}\varphi \in \operatorname{Ham}_c(U)$. Then $\operatorname{supp}(C+W-W') \subset V$ for some constant $C\in \BR$.
\end{enumerate}
\end{lemma}

Note that the constant $C$ in the lemma is only relevant if $M$ is closed. For open $M$ or $M$ with non-empty boundary, the constant must vanish since generating functions are assumed to be compactly supported and vanish on the boundary.

Before turning to the proof of Lemma~\ref{lem:difference_of_generating_functions}, let us record the following elementary observation, which will also be useful in the proof of Lemma~\ref{lem:change_of_generating_function_under_chain_of_compactly_supported_ham_diffeos} below. Consider an exact symplectic manifold $(S,d\sigma)$ and an exact Lagrangian submanifold $L\subset S$, i.e.\ a Lagrangian such that $\sigma|_L = df$ for some function $f$ on $L$. Let $(\varphi_H^t)_{t\in [0,1]}$ be a Hamiltonian isotopy generated by a Hamiltonian $H$. Then the Lagrangian $L'\coloneqq \varphi_H^1(L)$ is also exact and a primitive $f'$ of $\sigma|_{L'}$ is given by
\begin{equation}
\label{eq:difference_of_generating_functions_proof_new_primitive}
f'(\varphi_H^1(x)) \ = f(x) + \int_{ \left\{ \varphi_H^t(x) \right\}_{t\in [0,1]} } \sigma + \int_0^1 H_t(\varphi_H^t(x)) dt \qquad \text{for $x \in L$}.
\end{equation}

\begin{proof}[Proof of Lemma~\ref{lem:difference_of_generating_functions}]
We prove statement~\eqref{item:difference_of_generating_functions_gen_funct_to_diff}. Suppose that $\operatorname{supp}(C+W-W') \subset U$. This implies that the graphs $\operatorname{gr}(\varphi)$ and $\operatorname{gr}(\varphi')$ agree outside the set $\operatorname{pr}^{-1}(U)$. Moreover, the functions $W\circ \operatorname{pr}|_{\operatorname{gr}(\varphi)}$ and $W'\circ\operatorname{pr}|_{\operatorname{gr}(\varphi')}$ agree up to an additive constant outside this set.

Now consider the Hamiltonian diffeomorphism $\varphi'\times \operatorname{id}$ of $M\times \overline{M}$. This Hamiltonian diffeomorphism maps $\operatorname{gr}(\varphi)$ to $\operatorname{gr}(\varphi \varphi'^{-1})$ and $\operatorname{gr}(\varphi')$ to $\Delta$. If $\varphi'$ is sufficiently close to the identity, then the fact that $\operatorname{gr}(\varphi)$ and $\operatorname{gr}(\varphi')$ agree outside $\operatorname{pr}^{-1}(U)$ implies that $\operatorname{gr}(\varphi\varphi'^{-1})$ and $\Delta$ agree outside $\operatorname{pr}^{-1}(V)$. Moreover, in view of identity~\eqref{eq:difference_of_generating_functions_proof_new_primitive} and using that $W\circ \operatorname{pr}|_{\operatorname{gr}(\varphi)}$ and $W'\circ\operatorname{pr}|_{\operatorname{gr}(\varphi')}$ agree up to an additive constant outside $\operatorname{pr}^{-1}(U)$, we see that $\lambda|_{\operatorname{gr}(\varphi\varphi'^{-1})}$ admits a primitive which agrees with a primitive of $\lambda|_{\Delta}$ outside of $\operatorname{pr}^{-1}(V)$. Since any primitive of $\lambda|_\Delta$ is constant, we can conclude that $\varphi\varphi'^{-1}$ admits a generating function which vanishes outside of $V$. Since we can run the above argument with a slightly smaller open set $V'$ whose closure is contained in $V$, we obtain that $\varphi\varphi'^{-1}$ has a generating function supported in $V$. Thus $\varphi\varphi'^{-1} \in \operatorname{Ham}_c(V)$. In order to show that $\varphi'^{-1}\varphi \in \operatorname{Ham}_c(V)$, one can run an analogous argument with $\operatorname{id}\times \varphi'^{-1}$ in place of $\varphi'\times \operatorname{id}$. We omit the details.

We prove statement~\eqref{item:difference_of_generating_functions_diff_to_gen_funct}. Suppose that $\varphi\varphi'^{-1} \in \operatorname{Ham}_c(U)$. The Hamiltonian diffeomorphism $\operatorname{id}\times \varphi\varphi'^{-1}$ maps $\operatorname{gr}(\varphi')$ to $\operatorname{gr}(\varphi)$. Hence $\operatorname{gr}(\varphi)$ and $\operatorname{gr}(\varphi')$ agree outside $U\times \overline{M}$. Moreover, since $\operatorname{id}\times \varphi\varphi'^{-1} \in \operatorname{Ham}_c(U\times \overline{M})$, we can see from identity~\eqref{eq:difference_of_generating_functions_proof_new_primitive} that $\lambda|_{\operatorname{gr}(\varphi)}$ and $\lambda|_{\operatorname{gr}(\varphi')}$ admit primitives which agree outside $U\times \overline{M}$. Since such primitives are related to generating functions via the projection $\operatorname{pr}$, it follows that if $\varphi$ and $\varphi'$ are sufficiently close to the identity, then they admit generating functions which agree outside $V$ (and in fact outside some slightly smaller set $V'$ with closure contained in $V$). Hence there exists a constant $C$ such that $\operatorname{supp}(C+W-W') \subset V$. The same conclusion can be obtained under the assumption $\varphi'^{-1}\varphi \in \operatorname{Ham}_c(U)$ via a similar argument.
\end{proof}

\begin{lemma}
\label{lem:change_of_generating_function_under_chain_of_compactly_supported_ham_diffeos}
Let $(M,\omega)$ be a connected closed symplectic manifold. Let $U_1,\dots,U_k$ be proper open subsets of $M$. Let $\psi_0,\psi_1,\dots,\psi_k= \psi_0 \in \operatorname{Ham}(M)$ be Hamiltonian diffeomorphisms close to the identity with generating functions $W_0,W_1,\dots, W_k$. Suppose that $\varphi_i \coloneqq \psi_i\psi_{i-1}^{-1} \in \operatorname{Ham}_c(U_i)$ and $\operatorname{supp}(W_i-W_{i-1}) \subset U_i$ for all $i$. Then
\begin{equation*}
\int_M (W_k-W_0) \omega^n = - \sum\limits_{i=1}^k \operatorname{Cal}_{U_i}(\varphi_i).
\end{equation*}
Here $\operatorname{Cal}_{U_i}$ denotes the local real-valued lift of the Calabi homomorphism near the identity in $\operatorname{Ham}_c(U_i)$; see Remark~\ref{rem:local_lift_of_calabi}.
\end{lemma}

\begin{proof}
For each $1\leq i\leq k$, pick a small compactly supported Hamiltonian $H_i:[0,1]\times U_i \rightarrow \BR$ generating $\varphi_i$. Let $H: [0,k]\times M\rightarrow \BR$ be the concatenation of all $H_i$ which restricts to $H_i$ on $[i,i-1]\times M$. Note that since $\psi_0 = \psi_k$, we have $\varphi_k\cdots \varphi_1 = \operatorname{id}$. Therefore $H$ generates a loop in $\operatorname{Ham}(M)$ which is based at the identity and stays in a small neighbourhood of the identity.

Define the Hamiltonian $G$ on $M\times \overline{M}$ by $G \coloneqq - \operatorname{pr}_2^* H$, where $\operatorname{pr}_2$ denotes the projection onto the second factor of $M\times \overline{M}$. The Hamiltonian flow of $G$ is given by $\varphi_G^t = \operatorname{id}\times \varphi_H^t$. Note that for each integer $0\leq i \leq k$, the Hamiltonian diffeomorphism $\varphi_G^i$ maps $\operatorname{gr}(\psi_0)$ to $\operatorname{gr}(\psi_i)$.

For each $i$, we have a primitive of $\lambda|_{\operatorname{gr}(\psi_i)}$ given by $W_i\circ \operatorname{pr}|_{\operatorname{gr}(\psi_i)}$. We claim that
\begin{equation*}
W_i\circ\operatorname{pr}|_{\operatorname{gr}(\psi_i)}(\varphi_G^i(x)) = W_{i-1}\circ \operatorname{pr}|_{\operatorname{gr}(\psi_{i-1})}(\varphi_G^{i-1}(x)) + \int_{ \left\{ \varphi_G^t(x) \right\}_{t\in [i-1,i]}}\lambda + \int_{i-1}^i G_t(\varphi_G^t(x))dt
\end{equation*}
for every $1\leq i \leq k$ and $x \in \operatorname{gr}(\psi_0)$. Indeed, by identity~\eqref{eq:difference_of_generating_functions_proof_new_primitive}, this holds up to some constant. It follows from the assumptions that $W_i-W_{i-1}$ is supported in $U_i$ and that the restriction of $G$ to the time interval $[i-1,i]$ is supported in $M\times U_i$ that this constant must be zero.

From the above we conclude that
\begin{equation*}
W_k\circ \operatorname{pr}|_{\operatorname{gr}(\psi_k)}(x) = W_0\circ\operatorname{pr}|_{\operatorname{gr}(\psi_0)}(x) + \int_{ \left\{ \varphi_G^t(x) \right\}_{t\in [0,k]}}\lambda + \int_0^k G_t(\varphi_G^t(x)) dt
\end{equation*}
for every $x \in \operatorname{gr}(\psi_0)$, where we use that $\varphi_G^t$ is a loop. We therefore need to show that
\begin{equation}
\label{eq:change_of_generating_function_under_chain_of_compactly_supported_ham_diffeos_proof_action}
\int_{ \left\{ \varphi_G^t(x) \right\}_{t\in [0,k]}}\lambda + \int_0^k G_t(\varphi_G^t(x)) dt = - \left( \int_M\omega^n \right)^{-1} \sum\limits_{i=1}^k \operatorname{Cal}_{U_i}(\varphi_i)
\end{equation}
for every $x$. Note that the expression on the left is exactly the action of $x$ viewed as a fixed point of the time-$k$ map of the Hamiltonian loop $(\varphi_G^t)_{t\in [0,k]}$ generated by $G$. It was proved by Oh in \cite{oh05} that this action does not change if we replace the loop $(\varphi_G^t)_{t\in [0,k]}$ by a loop $(\varphi_F^t)_{t\in [0,k]}$ which is homotopic to $(\varphi_G^t)_{t\in [0,k]}$ as a based loop and whose generating Hamiltonian $F$ is normalized such that
\begin{equation}
\label{eq:change_of_generating_function_under_chain_of_compactly_supported_ham_diffeos_proof_normalization_hamiltonian}
\int_{[0,k]\times M\times \overline{M}} F dt \wedge \Omega^{2n} = \int_{[0,k]\times M\times \overline{M}} G dt \wedge \Omega^{2n}.
\end{equation}
Here $\Omega$ abbreviates the symplectic form $\omega \oplus -\omega$ on $M\times \overline{M}$ and the dimension of $M$ is $2n$. Note that since the loop $(\varphi_G^t)_{t\in [0,k]}$ is contained in a small neighbourhood of the identity, it is homotopic as a based loop to the constant loop at the identity. It is straightforward from the definition of $G$ that
\begin{equation*}
\fint_{[i-1,i]\times M\times \overline{M}} G = - \fint_{[0,1]\times M} H_i,
\end{equation*}
where $\fint$ computes the mean value of a function. Now define $F$ to be the constant Hamiltonian equal to
\begin{equation*}
F \coloneqq -k^{-1} \sum\limits_{i=1}^k \fint_{[0,1]\times M} H_i = - \left( k \int_M \omega^n \right)^{-1}\sum\limits_{i=1}^k \operatorname{Cal}_{U_i}(\varphi_i). 
\end{equation*}
The Hamiltonian $F$ generates the constant loop at the identity and satisfies the normalization~\eqref{eq:change_of_generating_function_under_chain_of_compactly_supported_ham_diffeos_proof_normalization_hamiltonian}. We conclude that the action on the left hand side of equation~\eqref{eq:change_of_generating_function_under_chain_of_compactly_supported_ham_diffeos_proof_action} is given by
\begin{equation*}
\int_{ \left\{ \varphi_F^t (x)\right\}_{t\in [0,1]}} \lambda + \int_0^k F_t(\varphi_F^t(x)) dt = - \left(\int_M \omega^n \right)^{-1}\sum\limits_{i=1}^k \operatorname{Cal}_{U_i}(\varphi_i).
\end{equation*}
This finishes the proof of the lemma.
\end{proof}

\section{Reduction to pairs of balls}
\label{sec:reduction_to_balls}

Given an open ball $B\subset \BR^{2n}$ centered at the origin, let us define the half ball $B_+\coloneqq B\cap (\BR_{\geq 0}\times \BR^{2n-1})$. The half ball $B_+$ is a symplectic manifold with boundary given by $\partial B_+ = B \cap (\left\{ 0 \right\}\times \BR^{2n-1})$. Here the term ``boundary'' refers to the boundary of $B_+$ as a manifold, and not to the frontier of $B_+$ as a subspace of $\BR^{2n}$.

The goal of this section is to deduce our main results, Theorems~\ref{thm:smooth_perfectness_closed}, \ref{thm:smooth_perfectness_open}, \ref{thm:smooth_perfectness_general} and Corollaries~\ref{cor:homogeneous_quasimorphisms_continuity}, \ref{cor:smooth_perfectness_open}, \ref{cor:smooth_perfectness_general}, from the following theorem, whose proof is the subject of the forthcoming sections.

\begin{theorem}
\label{thm:smooth_perfectness_pairs_of_balls}
Let $n>0$ be an integer and fix two open balls $B'\Subset B\Subset \BR^{2n}$ centered at the origin. Then there exists an integer $m>0$ such that the following is true. Let $(V,U)$ be either the tuple $(B,B')$ or the tuple $(B_+,B_+')$ and consider the map
\begin{equation*}
\Phi:\operatorname{Ham}_c^0(V)^{2m} \rightarrow \operatorname{Ham}_c^0(V) \quad (u_1,v_1,\dots,u_m,v_m)\mapsto \prod\limits_{j=1}^m [u_j,v_j].
\end{equation*}
Then there exist an open neighbourhood $\MN\subset \operatorname{Ham}_c^0(U)$ of the identity and a smooth right inverse
\begin{equation*}
\Psi: \MN \rightarrow \operatorname{Ham}_c^0(V)^{2m}
\end{equation*}
of $\Phi$. We can choose $\MN$ and $\Psi$ such that $\Psi(\operatorname{id})$ is arbitrarily close to the tuple $(\operatorname{id},\dots,\operatorname{id})$.
\end{theorem}

\begin{remark}
\label{rmk:smooth_perfectness_paris_of_balls_m}
While the integer $m$ in Theorems~\ref{thm:smooth_perfectness_closed}, \ref{thm:smooth_perfectness_open}, and \ref{thm:smooth_perfectness_general} is only dependent on the dimension $n$, the integer $m$ in Theorem~\ref{thm:smooth_perfectness_pairs_of_balls} is allowed to depend on $n$ and the choice of balls $B'\Subset B$. An easy scaling argument shows that $m$ can be chosen to only depend on $n$ and the ratio of the radii of $B'$ and $B$.
\end{remark}

\begin{remark}
\label{rem:manifold_without_boundary_only_require_pairs_of_full_balls}
In order to deduce our main results for symplectic manifolds without boundary, one only needs Theorem~\ref{thm:smooth_perfectness_pairs_of_balls} for pairs of full balls.
\end{remark}

Before deducing our main results from Theorem~\ref{thm:smooth_perfectness_pairs_of_balls}, we state and prove some preliminary lemmas. The first lemma is concerned with fragmentation of Hamiltonian diffeomorphisms and will also be used in later sections.

\begin{lemma}[Fragmentation]
\label{lem:fragmentation_basic}
Let $(M^{2n},\omega)$ be a connected symplectic manifold, not necessarily compact and possibly with boundary. Suppose that $\MU = (U_i)_{1 \leq i \leq k}$ is a finite open cover of $M$. Then for every $\varphi \in \operatorname{Ham}_c(M)$ in a sufficiently small neighbourhood of the identity, there exist $\varphi_i \in \operatorname{Ham}_c(U_i)$ for $1 \leq i \leq k$ such that $\varphi = \varphi_1 \circ \cdots \circ \varphi_k$. The $\varphi_i$ can be chosen such that the assignment $\varphi \mapsto \varphi_i$ is smooth and maps $\operatorname{id}_M$ to $\operatorname{id}_{U_i}$. Moreover, the following properties hold:
\begin{enumerate}
\item \label{item:fragmentation_basic_closed_M_calabi_sum} If $M$ is closed and all open subsets $U_i$ are proper (so that $\operatorname{Cal}_{U_i}$ is defined), then
\begin{equation*}
\sum_i \operatorname{Cal}_{U_i}(\varphi_i) = 0.
\end{equation*}
\item \label{item:fragmentation_basic_open_V} Let $V\subset M$ be an open subset. If $M$ is closed and $\varphi \in \operatorname{Ham}_c^0(V)$ or if $M$ is not closed and $\varphi \in \operatorname{Ham}_c(V)$, then $\varphi_i \in \operatorname{Ham}_c(V\cap U_i)$ for all $i$.
\end{enumerate}
\end{lemma}

\begin{proof}
A well-known strategy to construct fragmentations of Hamiltonian diffeomorphisms is to use generating functions; see, for instance, \cite[\S 4]{ban97}.

Fix a partition of unity $(\lambda_i)_{1 \leq i \leq k}$ subordinate to the open cover $\MU$ and define
\begin{equation*}
\mu_0 \coloneqq 0 \qquad \text{and} \qquad \mu_i = \sum_{j\leq i} \lambda_j \quad \text{for $1 \leq i \leq k$}.
\end{equation*}

Let us first assume that $M$ is open. Given $\varphi \in \operatorname{Ham}_c(M)$, let $W$ be the compactly supported generating function of $\varphi$. Set $W_i \coloneqq \mu_i W$. The assignment $\varphi \mapsto W_i$ is smooth and maps $\operatorname{id}_M$ to $0$. Let $\psi_i \in \operatorname{Ham}_c(M)$ denote the Hamiltonian diffeomorphism generated by $W_i$. By construction, we have
\begin{equation*}
\operatorname{supp}(W_i - W_{i-1}) \subset \operatorname{supp}(\lambda_i) \subset U_i.
\end{equation*}
By Lemma~\ref{lem:difference_of_generating_functions}, this implies that
\begin{equation*}
\varphi_i \coloneqq \psi_{i-1}^{-1}\psi_i \in \operatorname{Ham}_c(U_i)
\end{equation*}
whenever $\varphi$ is sufficiently close to $\operatorname{id}$. Clearly, we have $\varphi = \varphi_1 \circ \cdots \circ \varphi_k$, and $\varphi \mapsto \varphi_i$ is smooth and maps $\operatorname{id}_M$ to $\operatorname{id}_{U_i}$. Suppose that $\varphi \in \operatorname{Ham}_c(V)$. Then $W$ and all $W_i$ are supported in $V$. Hence the differences $W_i-W_{i-1}$ are supported in $V\cap U_i$ and we can deduce via Lemma~\ref{lem:difference_of_generating_functions} that $\varphi_i \in \operatorname{Ham}_c(V \cap U_i)$. This confirms property~\eqref{item:fragmentation_basic_open_V} in the open case.

Now suppose that $M$ is closed. In this case, the generating function $W$ of $\varphi$ is only determined up to an additive constant. For a choice of $W$ close to zero, we can define $W_i$, $\psi_i$, and $\varphi_i$ as in the open case. This yields a fragmentation $\varphi = \varphi_1 \cdots \varphi_k$ with $\varphi_i \in \operatorname{Ham}_c(U_i)$. We write $\varphi_i = \varphi_i(W)$ to indicate the dependence of $\varphi_i$ on $W$. If $W'$ is another choice of generating function of $\varphi$, then it follows from Lemma~\ref{lem:change_of_generating_function_under_chain_of_compactly_supported_ham_diffeos} that
\begin{equation*}
\int_M (W - W')\omega^n = -\sum\limits_{i=1}^k \operatorname{Cal}_{U_i}(\varphi_i(W)) + \sum\limits_{i=1}^k \operatorname{Cal}_{U_i}(\varphi_i(W')).
\end{equation*}
It follows from this identity that there exists a unique normalization $W_0$ of the generating function of $\varphi$ such that the desired identity $\sum_i \operatorname{Cal}_{U_i}(\varphi_i(W_0)) = 0$ holds. We define the fragmentation $\varphi = \varphi_1 \cdots \varphi_k$ via this normalization, i.e.\ we set $\varphi_i \coloneqq \varphi_i(W_0)$. Note that this yields a smooth dependence of $\varphi_i$ on $\varphi$. If $\varphi = \operatorname{id}$, then the generating function $W=0$ gives rise to $\varphi_i(0) = \operatorname{id}$. Since $\sum_i \operatorname{Cal}_{U_i}(\operatorname{id}) = 0$, we have $W_0 = 0$ and we deduce that the assignment $\varphi\mapsto \varphi_i$ maps $\operatorname{id}_M$ to $\operatorname{id}_{U_i}$. Clearly, property~\eqref{item:fragmentation_basic_closed_M_calabi_sum} holds by construction of our fragmentation.

It remains to verify property~\eqref{item:fragmentation_basic_open_V}. If $V = M$ there is nothing to do, so we assume that $V$ is a proper open subset. Suppose that $\varphi \in \operatorname{Ham}_c^0(V)$. Let $W$ be the unique generating function of $\varphi$ which is supported in $V$. As in the open case, we have $\varphi_i(W) \in \operatorname{Ham}_c(V\cap U_i)$. We compute
\begin{equation*}
\sum_i\operatorname{Cal}_{U_i}(\varphi_i(W)) = \sum_i\operatorname{Cal}_{V\cap U_i}(\varphi_i(W)) = \operatorname{Cal}_V(\varphi) = 0.
\end{equation*}
Hence we see that $W_0 = W$, which implies that $\varphi_i \in \operatorname{Ham}_c(V\cap U_i)$ and confirms property~\eqref{item:fragmentation_basic_open_V}.
\end{proof}

\begin{lemma}[Fragmentation with Calabi control]
\label{lem:fragmentation_with_Calabi_control}
Let $(M^{2n},\omega)$ be a connected symplectic manifold, not necessarily compact and possibly with boundary. Suppose that $\MU = (U_i)_{1 \leq i \leq k}$ is a finite open cover of $M$. Then for every $\varphi \in \operatorname{Ham}_c^0(M)$ sufficiently close to the identity, there exist $\varphi_i \in \operatorname{Ham}_c^0(U_i)$ for $1 \leq i \leq k$ such that $\varphi = \varphi_1 \circ \cdots \circ \varphi_k$. Moreover, the $\varphi_i$ can be chosen such that the assignment $\varphi \mapsto \varphi_i$ is smooth and maps $\operatorname{id}_M$ to $\operatorname{id}_{U_i}$.
\end{lemma}

\begin{proof}
Let us assume that all $U_i$ are proper subsets of $M$. If not, the lemma holds true trivially by setting $\varphi_i \coloneqq \varphi$ for some $i$ with $U_i = M$ and $\varphi_j \coloneqq \operatorname{id}_{U_j}$ for $j \neq i$.

Let $\varphi \in \operatorname{Ham}_c^0(M)$ be close to the identity. By Lemma~\ref{lem:fragmentation_basic} there exists a fragmentation $\varphi = \tilde{\varphi}_1 \circ \cdots \circ \tilde{\varphi}_k$ with $\tilde{\varphi}_i \in \operatorname{Ham}_c(U_i)$ such that $\varphi\mapsto \tilde{\varphi}_i$ is smooth and maps $\operatorname{id}_M$ to $\operatorname{id}_{U_i}$. Moreover, this fragmentation satisfies
\begin{equation*}
\sum_i \operatorname{Cal}_{U_i}(\tilde{\varphi}_i) = 0.
\end{equation*}
If $M$ is not closed, this follows from the fact that Calabi is a group homomorphism and that $\operatorname{Cal}(\varphi) = 0$ by assumption. If $M$ is closed, this is property~\eqref{item:fragmentation_basic_closed_M_calabi_sum} in Lemma~\ref{lem:fragmentation_basic}.

Our goal is to modify the fragmentation $\varphi = \tilde{\varphi}_1 \circ \cdots \circ \tilde{\varphi}_k$ such that the factors $\tilde{\varphi}_i$ are contained in $\operatorname{Ham}_c^0(U_i)$. The basic idea is the following: Consider $1 \leq i < j \leq k$ and a non-empty open set $B \subset U_i \cap U_j$ disjoint from any $U_\ell$ for $i<\ell<j$. Pick an autonomous Hamiltonian $H$ compactly supported in $B$ and normalized such that $\operatorname{Cal}(\varphi_H^1) = 1$. By our assumptions on $B$, the Hamiltonian diffeomorphism $\varphi_H^t$ commutes with $\tilde{\varphi}_\ell$ for all $t\in \BR$ and $i< \ell < j$. Hence we have
\begin{equation*}
\varphi = \tilde{\varphi}_1 \circ \cdots \circ \tilde{\varphi}_{i-1} \circ (\tilde{\varphi}_i \circ\varphi_H^{-t}) \circ \tilde{\varphi}_{i+1} \circ \cdots \circ \tilde{\varphi}_{j-1} \circ (\varphi_H^{t} \circ \tilde{\varphi}_j) \circ \tilde{\varphi}_{j+1} \circ \cdots \circ \tilde{\varphi}_k.
\end{equation*}
This allows us to modify the Calabi invariant of the factors $\tilde{\varphi}_i$ and $\tilde{\varphi}_j$ by $-t$ and $t$, respectively. Since $\sum_i \operatorname{Cal}_{U_i}(\tilde{\varphi}_i) = 0$, it is possible to make the Calabi invariants of all $\tilde{\varphi}_i$ vanish by inserting factors as above. In the following, we explain this strategy in more detail.

For each $1 \leq i \leq k$, let $(U_i^\alpha)_{\alpha \in A_i}$ denote the set of path components of $U_i$. Moreover, let $A$ denote the set of all pairs $(i,\alpha)$ where $\alpha \in A_i$. We build a graph $G$ with vertex set $A$ by inserting, for each $1 \leq i < j \leq k$, $\alpha \in A_i$, and $\beta \in A_j$, an edge $e$ between the vertices $(i,\alpha)$ and $(j,\beta)$ if
\begin{equation*}
V_e \coloneqq(U_i^\alpha \cap U_j^\beta) \setminus \bigcup_{i<\ell<j} \overline{U}_\ell \neq \emptyset.
\end{equation*}
Let $E$ denote the set of all edges.

We claim that the graph $G$ is connected. Indeed, consider arbitrary vertices $(i,\alpha)\neq (j,\beta)$. Since $M$ is connected by assumption, there exists a finite sequence
\begin{equation*}
(i,\alpha) = (i_0,\alpha_0) \neq (i_1,\alpha_1) \neq \dots \neq (i_s,\alpha_s) = (j,\beta)
\end{equation*}
such that $U_{i_{\nu-1}}^{\alpha_{\nu-1}} \cap U_{i_\nu}^{\alpha_\nu}\neq \emptyset$ for all $\nu$. Note that this implies $i_\nu \neq i_{\nu-1}$. Consider consecutive vertices $(i_{\nu -1},\alpha_{\nu-1})$ and $(i_\nu,\alpha_\nu)$. If there is no edge between $(i_{\nu-1},\alpha_{\nu-1})$ and $(i_\nu,\alpha_\nu)$, then this means that there exists $\ell$ strictly between $i_{\nu-1}$ and $i_\nu$ such that $U_{i_{\nu-1}}^{\alpha_{\nu-1}} \cap U_{i_\nu}^{\alpha_\nu} \cap \overline{U}_\ell \neq \emptyset$. This further implies that there exists $\gamma \in A_\ell$ such that $U_{i_{\nu-1}}^{\alpha_{\nu-1}} \cap U_{i_\nu}^{\alpha_\nu} \cap U_\ell^\gamma \neq \emptyset$. Now insert the vertex $(\ell,\gamma)$ between the vertices $(i_{\nu-1},\alpha_{\nu-1})$ and $(i_\nu,\alpha_\nu)$. Note that this operation replaces the index jump $i_\nu - i_{\nu-1}$ by two smaller index jumps $\ell - i_{\nu-1}$ and $i_\nu-\ell$ of the same sign and strictly smaller absolute value. Therefore, after repeating the above vertex insertion finitely many times, there must be an edge between any two consecutive vertices. This confirms connectedness of $G$.

Let us orient each edge $e$ as $e = ( (i,\alpha),(j,\beta))$ with $i<j$. Let $C_0(G;\BR)$ and $C_1(G;\BR)$ denote the $\BR$-vector spaces freely generated by the vertices and edges of $G$, respectively, and consider the boundary operator
\begin{equation*}
\partial : C_1(G;\BR) \rightarrow C_0(G;\BR) \qquad e = (a,b) \mapsto b-a.
\end{equation*}
Since $G$ is connected, the image $B_0(G;\BR)$ of $\partial$ consists of precisely those chains $\sum_{a \in A} x_a \cdot a$ such that $\sum_{a\in A} x_a = 0$. Fix a linear right inverse
\begin{equation*}
\Xi : B_0(G;\BR) \rightarrow C_1(G;\BR)
\end{equation*}
of $\partial$.

For each edge $e \in E$, we choose a non-empty open ball $B_e \subset V_e$. This can be done in such a way that the collection of balls $B_e$ is pairwise disjoint. Indeed, consider the set $M' \coloneqq M \setminus \bigcup_i \operatorname{fr}(U_i)$, where $\operatorname{fr}(U_i)$ is the frontier of the open subset $U_i$ in $M$. We observe that this set is open and dense in $M$. Moreover, every point in $M'$ has an open neighbourhood intersecting only finitely many of the sets $U_i^\alpha$. This implies that every $V_e$ contains a point with a neighbourhood intersecting only finitely many of the other sets $V_{e'}$. This clearly makes it possible to choose the balls $B_e$ pairwise disjoint.

Let us now choose, for every edge $e \in E$, an autonomous Hamiltonian $H_e$ compactly supported inside $B_e$ and normalized such that $\operatorname{Cal}(\varphi_{H_e}^1) = 1$. This implies that $\operatorname{Cal}(\varphi_{H_e}^t) = t$ for all $t \in \BR$.

For each $(i,\alpha) \in A$, let $\tilde{\varphi}_i^\alpha$ denote the restriction of $\tilde{\varphi}_i$ to the path component $U_i^\alpha$. Note that since $\sum_i\operatorname{Cal}(\tilde{\varphi}_i) = 0$, it follows that
\begin{equation}
\label{eq:fragmentation_with_calabi_control_proof_def_of_x}
x \coloneqq \sum_{(i,\alpha)\in A} \operatorname{Cal}(\tilde{\varphi}_i^\alpha) \cdot (i,\alpha) \in B_0(G;\BR).
\end{equation}
Let us write
\begin{equation*}
\Xi(x) = \sum_{e\in E} t_e\cdot e.
\end{equation*}
For each $(i,\alpha) \in A$, let $E_{i,\alpha}^+$ denote the set of oriented edges of the form $( (j,\beta), (i,\alpha))$. Similarly, let $E_{i,\alpha}^-$ be the set of oriented edges of the form $( (i,\alpha), (j,\beta))$. Let $E_i^+$ be the union of all $E_{i,\alpha}^+$ and $E_i^-$ the union of all $E_{i,\alpha}^-$. We define
\begin{equation*}
\varphi_i \coloneqq \left(\prod_{e \in E_i^+} \varphi_{H_e}^{-t_e}\right) \circ \tilde{\varphi}_i \circ \left( \prod_{e\in E_i^-} \varphi_{H_e}^{t_e} \right).
\end{equation*}
Note that the order within the products on the left and right of $\tilde{\varphi}_i$ does not matter since the sets $B_e$ are pairwise disjoint. Moreover, note that the product is finite since only finitely many $t_e$ are non-zero.

It remains to check that the diffeomorphisms $\varphi_i$ have all the required properties. First, we verify that $\varphi_i \in \operatorname{Ham}_c^0(U_i)$. Since $B_e \subset U_i$ for every $e \in E_i^\pm$, we have $\varphi_i \in \operatorname{Ham}_c(U_i)$. Let $\varphi_i^\alpha$ be the restriction of $\varphi_i$ to $U_i^\alpha$. We compute
\begin{equation}
\label{eq:fragmentation_with_calabi_control_proof_cal_phi_i_alpha}
\operatorname{Cal}_{U_i^\alpha}(\varphi_i^\alpha) = \operatorname{Cal}_{U_i^\alpha}(\tilde{\varphi}_i^\alpha)  -\sum_{e\in E_{i,\alpha}^+} t_e + \sum_{e\in E_{i,\alpha}^-} t_e.
\end{equation}
Note that by our choice of $\Xi$, the $(i,\alpha)$-component of $\partial\circ \Xi(x)$ is given by $\operatorname{Cal}_{U_i^\alpha}(\tilde{\varphi}_i^\alpha)$. Alternatively, the $(i,\alpha)$-component of $\partial\circ \Xi(x)$ can be computed as
\begin{equation*}
\sum_{e\in E_{i,\alpha}^+} t_e - \sum_{e\in E_{i,\alpha}^-} t_e.
\end{equation*}
We deduce that the right hand side of \eqref{eq:fragmentation_with_calabi_control_proof_cal_phi_i_alpha} vanishes. Hence $\operatorname{Cal}_{U_i^\alpha}(\varphi_i^\alpha) = 0$ for all $(i,\alpha)$, which implies that $\varphi_i \in \operatorname{Ham}_c^0(U_i)$.

Next, let us show that $\varphi = \varphi_1 \circ \cdots \circ \varphi_k$. Consider the expansion
\begin{equation*}
\varphi_1  \cdots  \varphi_k = \tilde{\varphi}_1 \left(\prod_{e\in E_1^-}\varphi_{H_e}^{t_e}\right) \cdots \left( \prod_{e\in E_i^+}\varphi_{H_e}^{-t_e} \right) \tilde{\varphi}_i \left( \prod_{e\in E_i^-}\varphi_{H_e}^{t_e} \right) \cdots \left( \prod_{e\in E_k^+}\varphi_{H_e}^{-t_e} \right)\tilde{\varphi}_k
\end{equation*}
For each edge $e\in E$, each of the two diffeomorphisms $\varphi_{H_e}^{t_e}$ and $\varphi_{H_e}^{-t_e}$ show up once in this expansion. Write $e= ( (i,\alpha),(j,\beta))$. Then both $\varphi_{H_e}^{\pm t_e}$ show up between $\tilde{\varphi}_i$ and $\tilde{\varphi}_j$. By construction of $B_e$, the diffeomorphisms $\varphi_{H_e}^{\pm t_e}$ have disjoint support from any other diffeomorphism showing up between $\tilde{\varphi}_i$ and $\tilde{\varphi}_j$. Hence they commute with every diffeomorphism between $\tilde{\varphi}_i$ and $\tilde{\varphi}_j$, and after rearranging the product, we can cancel $\varphi_{H_e}^{\pm t_e}$. This shows that
\begin{equation*}
\varphi_1 \cdots \varphi_k = \tilde{\varphi}_1 \cdots \tilde{\varphi}_k = \varphi.
\end{equation*}

It is straightforward to see that $\varphi_i = \operatorname{id}_{U_i}$ if $\varphi = \operatorname{id}_M$. Finally, suppose that $(\varphi(z))_{z\in \BR^r}$ is a smooth finite parameter family in $\operatorname{Ham}_c^0(M)$. Then, as established above, $(\tilde{\varphi}_i(z))_z$ is a smooth family in $\operatorname{Ham}_c(U_i)$. This implies that for each compact subset $K\subset \BR^r$, there exists a finite subset $A_i'\subset A_i$ such that $\tilde{\varphi}_i(z)$ is supported in $\bigcup_{\alpha \in A_i'} U_i^\alpha$ for each $z \in K$. Hence the family of boundaries $x(z)$ defined in equation \eqref{eq:fragmentation_with_calabi_control_proof_def_of_x} is contained in a finite dimensional subspace of $B_0(G;\BR)$ for $z \in K$. In the definition of $\varphi_i(z)$, we can therfore restrict the linear map $\Xi$ to a linear map between finite dimensional subspaces of $B_0(G;\BR)$ and $C_1(G;\BR)$ as $z$ ranges over $K$. As a result, it can be readily checked that $\varphi_i(z)$ is a smooth family.
\end{proof}

\begin{lemma}
\label{lem:covering_and_coloring}
For every integer $n>0$, there exists an integer $N>0$ such that the following is true: For every symplectic manifold $(M^{2n},\omega)$, not necessarily compact and possibly with boundary, there exists an open cover $\MB$ of $M$ by countably many open subsets symplectomorphic to symplectic balls or half balls such that:
\begin{enumerate}
\item For every (half) ball $B \in \MB$, there exists a larger symplectically embedded (half) ball $2 B \subset M$ such that the pair $(2 B,B)$ is conformally symplectomorphic to the pair $(B^{2n}_{(+)}(2),B^{2n}_{(+)}(1))$. Here $B^{2n}_{(+)}(r)$ denotes the (half) ball in $\BR^{2n}$ of radius $r$ centered at the origin.
\item There exists a coloring of $\MB$ by $N$ colors such that, for any two distinct (half) balls $B \neq B' \in \MB$ of the same color, the larger (half) balls $2 B$ and $2 B'$ are disjoint.
\end{enumerate}
\end{lemma}

\begin{proof}
Define $\BR^{2n}_+ \coloneqq \BR_{\geq 0}\times \BR^{2n-1}$. Let $\delta>0$. For $\varepsilon>0$, let $\MB_\varepsilon$ be the collection of all balls of radius $\varepsilon$ centered at lattice points $\varepsilon \BZ_{\geq 2}\times \delta\varepsilon\BZ^{2n-1}$ and all half balls of radius $2\varepsilon$ centered at lattice points $\left\{ 0 \right\}\times \delta\varepsilon\BZ^{2n-1}$. Note that for each (half) ball $B\in \MB$, the (half) ball $2B$ of double radius with the same center is contained in $\BR^{2n}_{+}$. It is possible to fix $\delta=\delta(n) >0$ sufficiently small such that $\MB_\varepsilon$ forms an open cover of $\BR^{2n}_+$ for every $\varepsilon>0$. Observe that with $\delta(n)$ fixed, one can find $C(n)>0$ such that there exists a coloring of $\MB_\varepsilon$ with $C(n)$ colors such that, for any two distinct $B,B'\in \MB_\varepsilon$ of the same color, we have $2B \cap 2B' = \emptyset$.

Let $U\Subset V \Subset \BR^{2n}_+$ be relatively compact open subsets. Consider the set of all $B \in \MB_\varepsilon$ such that $2B \subset V$. If $\varepsilon>0$ is sufficiently small, then the set of such $B$ covers $U$. This shows that there exists a $C(n)$-colored covering of $U$ by open (half) balls $B\subset V$ such that all $2B$ are contained in $V$ and for any two distinct $B$ and $B'$ of the same color we have $2B\cap 2B'=\emptyset$.

Let $(M^{2n},\omega)$ be a symplectic manifold. In order to deduce the assertion of the lemma from the above observation, note that there exist a constant $D(n)>0$ only depending on $n$ and a countable $D(n)$-colored cover of $M$ by relatively compact open subsets $V\Subset M$ symplectomorphic to relatively compact open subsets of $\BR^{2n}_+$ such that any two distinct $V$ and $V'$ of the same color are disjoint. For each $V$, we can pick $U\Subset V$ relatively compact such that the collection of all $U$ still covers $M$. Applying the above observation for each pair $U\Subset V$, we see that the statement of the lemma holds with $N = C(n)D(n)$.
\end{proof}

For the remainder of this section we assume that Theorem~\ref{thm:smooth_perfectness_pairs_of_balls} is valid and use it to show all the main results of this paper stated in the introduction. Note that Theorem~\ref{thm:smooth_perfectness_general} generalizes Theorems~\ref{thm:smooth_perfectness_closed} and~\ref{thm:smooth_perfectness_open}, so it suffices to prove the former.

\begin{proof}[Proof of Theorem \ref{thm:smooth_perfectness_general} assuming Theorem \ref{thm:smooth_perfectness_pairs_of_balls}]

Fix $n>0$. Pick an integer $N>0$ satisfying the assertion of Lemma \ref{lem:covering_and_coloring}. Let us apply Theorem \ref{thm:smooth_perfectness_pairs_of_balls} to the pair of (half) balls $(B^{2n}_{(+)}(2), B^{2n}_{(+)}(1))$. Let $m>0$ be the resulting integer. Note that $N$ and $m$ only depend on $n$.

Now consider an arbitrary connected symplectic manfold $(M^{2n},\omega)$, possibly with boundary and not necessarily compact. Let $\MB$ be a colored open covering of $M$ by (half) balls as in Lemma~\ref{lem:covering_and_coloring}. Let us identify the set of colors with $\left\{ 1,\dots,N \right\}$. For every color $1\leq \mu \leq N$, we define $U_\mu$ to be the union of all the (half) balls $B\in \MB$ whose color is $\mu$. Note that this union is disjoint.

Let $\varphi\in \operatorname{Ham}_c^0(U)$ be close to the identity. Applying Lemma \ref{lem:fragmentation_with_Calabi_control} to the cover $(U_\mu)_\mu$ of $M$ gives us a factorization $\varphi = \varphi_1 \cdots \varphi_N$ with $\varphi_\mu \in \operatorname{Ham}_c^0(U_\mu)$. Now consider one color $\mu$. For each (half) ball component $B$ of $U_\mu$, the restriction $\varphi_\mu|_{B}$ is contained in $\operatorname{Ham}_c^0(B)$ and close to the identity. We can therefore apply Theorem \ref{thm:smooth_perfectness_pairs_of_balls} for the pair of (half) balls $(2 B,B)$  and write
\begin{equation}
\label{eq:reduction_to_pairs_of_balls_application_of_pair_theorem}
\varphi_\mu|_B = \prod\limits_{j=1}^m [u_j^{B},v_j^{B}]
\end{equation}
for $u_j^B, v_j^B \in \operatorname{Ham}_c^0(2 B)$. Here we recall that the number of commutators $m$ necessary for the pair $(B^{2n}_{(+)}(2),B^{2n}_{(+)}(1))$ is the same as for the pair $(2 B,B)$ (see Remark \ref{rmk:smooth_perfectness_paris_of_balls_m}). Let us define $\overline{u}_j^\mu$ to be the composition of all $u_j^B$ where $B$ ranges over the (half) ball components of $U_\mu$. Note that since the (half) balls $2 B$ of color $\mu$ are disjoint, the order does not matter in this composition. Similarly, we define $\overline{v}_j^\mu$ to be the composition of all $v_j^B$. Again using disjointness of the (half) balls $2 B$ of color $\mu$, we see that
\begin{equation*}
\varphi_i = \prod\limits_{j=1}^m [\overline{u}_j^\mu,\overline{v}_j^\mu].
\end{equation*}
Therefore, we can write $\varphi$ as a product of $mN$ commutators:
\begin{equation*}
\varphi = \prod\limits_{\mu=1}^N \prod\limits_{j=1}^m [\overline{u}_j^\mu,\overline{v}_j^\mu].
\end{equation*}
Note, however, that while $\overline{u}_j^\mu$ and $\overline{v}_j^\mu$ are Hamiltonian diffeomorphisms of $(M,\omega)$, they are not necessarily compactly supported. The problem is that the smooth local right inverse in Theorem~\ref{thm:smooth_perfectness_pairs_of_balls} does not map $\operatorname{id}$ to the tuple $(\operatorname{id},\dots, \operatorname{id})$, but only close to it. For this reason, let us define $\varphi_j^\mu \in \operatorname{Ham}(M,\omega)$ to be the diffeomorphism $\overline{u}_j^\mu$ obtained via the above construction if we start with $\varphi = \operatorname{id}$. Similarly, we define $\psi_j^\mu \in \operatorname{Ham}(M,\omega)$ to be the diffeomorphism $\overline{v}_j^\mu$ for $\varphi = \operatorname{id}$. For a general input $\varphi \in \operatorname{Ham}_c^0(M,\omega)$, we can then write
\begin{equation*}
\overline{u}_j^\mu = \varphi_j^\mu u_j^\mu \qquad \text{and} \qquad \overline{v}_j^\mu = \psi_j^\mu v_j^\mu
\end{equation*}
for Hamiltonian diffeomorphisms $u_j^\mu$ and $v_j^\mu$. Clearly, we have
\begin{equation*}
\varphi = \prod\limits_{\mu=1}^N \prod\limits_{j=1}^m [\varphi_j^\mu u_j^\mu,\psi_j^\mu v_j^\mu].
\end{equation*}
Note that since each $\varphi_\mu$ is compactly supported in $U_\mu$, there are only finitely many (half) ball components $B$ of $U_\mu$ on which $\varphi_\mu$ is not equal to the identity. Therefore, there are only finitely many open sets $2B$ on which $\overline{u}_j^\mu$ and $\overline{v}_j^\mu$ differ from $\varphi_j^\mu$ and $\psi_j^\mu$, respectively. Hence we see that $u_j^\mu$ and $v_j^\mu$ are compactly supported and contained in $\operatorname{Ham}_c^0(M,\omega)$. The diffeomorphisms $u_j^\mu$ and $v_j^\mu$ depend smoothly on $\varphi$, and the diffeomorphisms associated to $\varphi = \operatorname{id}$ are equal to the identity by construction. This concludes the construction of the desired smooth local right inverse $\Psi$ of the map $\Phi$ defined in equation~\eqref{eq:smooth_perfectness_open_map_Phi}.

It remains to argue that the diffeomorphisms $\varphi_j^\mu$ and $\psi_j^\mu$ can be arranged to lie in the $C^\infty_{\operatorname{loc}}$ closure of an arbitrarily small open neighbourhood $\MU \subset \operatorname{Ham}_c^0(M,\omega)$ of the identity. In order to see this, note that when applying Theorem~\ref{thm:smooth_perfectness_pairs_of_balls} to obtain the factorization~\eqref{eq:reduction_to_pairs_of_balls_application_of_pair_theorem}, we can make sure that the diffeomorphisms $u_j^B$ and $v_j^B$ obtained for $\varphi = \operatorname{id}$ are contained in $\MU$. In fact, we can make sure that the composition of $u_j^B$ (or $v_j^B$) for any finite collection of components $B$ of $U_\mu$ is contained in $\MU$. This guarantees that $\varphi_j^\mu$ and $\psi_j^\mu$ are $C^\infty_{\operatorname{loc}}$ limits of elements of $\MU$.
\end{proof}

Note that since Corollary~\ref{cor:smooth_perfectness_general} generalizes Corollary~\ref{cor:smooth_perfectness_open}, it suffices to prove the latter.

\begin{proof}[Proof of Corollary~\ref{cor:smooth_perfectness_general}]
Let $(M,\omega)$ be a connected symplectic manifold, possibly non-compact and possibly with boundary. Let $U\Subset M$ be a relatively compact open subset. We observe that it is possible to find a sufficiently small open neighbourhood $\MU \subset \operatorname{Ham}_c^0(U)$ of the identity with the following property: Consider a sequence $(\varphi_j)_j$ in $\MU$. Suppose that $\varphi_j$ converges in $C^\infty_{\operatorname{loc}}$ to a not necessarily compactly supported diffeomorphism $\varphi$ of $U$. Then $(\varphi_j)_j$ is actually convergent in $\operatorname{Ham}_c^0(M,\omega)$ and $\varphi$ is the restriction of an element of the closure of $\operatorname{Ham}_c^0(U)$ in $\operatorname{Ham}_c^0(M,\omega)$. This can be easily deduced from our definition of the topology on $\operatorname{Ham}_c^0(U)$ involving a colimit over compact subsets of $U$. Using this observation, it is straightforward to obtain Corollary~\ref{cor:smooth_perfectness_general} from Theorem~\ref{thm:smooth_perfectness_general}.
\end{proof}

It remains to prove Corollary~\ref{cor:homogeneous_quasimorphisms_continuity}.

\begin{proof}[Proof of Corollary~\ref{cor:homogeneous_quasimorphisms_continuity}]
Let $(M,\omega)$ be a closed symplectic manifold. By Theorem~\ref{thm:smooth_perfectness_closed}, the commutator length on $\operatorname{Ham}(M,\omega)$ is locally bounded. Corollary~\ref{cor:homogeneous_quasimorphisms_continuity} therefore follows from the general observation that any homogeneous quasimorphism $q$ on a topological group $G$ with locally bounded commutator length must be continuous.

In order to see this, note that the local boundedness of the commutator length implies that $|q|$ must be bounded in some open neighbourhood $U\subset G$ of the identity. Let $E>0$ be a constant such that $|q(g)| \leq E$ for all $g \in U$. Moreover, let $D\geq 0$ be the defect of $q$.

Now fix $g \in G$ and $\varepsilon >0$. We need to show that there exists an open neighbourhood $V\subset G$ of $g$ such that $|q(g') - q(g)| <\varepsilon$ for all $g' \in V$. Choose a positive integer $N>0$ such that $\frac{D+E}{N} <\varepsilon$. Let $V\subset G$ be an open neighboorhood of $g$ such that $g^{-N}g'^{N} \in U$ for all $g' \in V$. For every $g' \in V$, we estimate:
\begin{equation*}
|q(g') - q(g)| = \frac{|q(g'^N) - q(g^N)|}{N} \leq \frac{|q(g^{-N}g'^N)| + D}{N} \leq \frac{E+D}{N} < \varepsilon.
\end{equation*}
Here the equality uses homogeneity of $q$. Thin concludes the proof of continuity of $q$.
\end{proof}

\section{Commutators and surfaces}
\label{sec:commutators_and_surfaces}

We recall that a \textit{$\Delta$-complex} $\MK$ consists of a sequence of sets $(K_n)_{n\geq 0}$ together with face maps
\begin{equation*}
\partial_i:K_{n+1}\rightarrow K_n \qquad \text{for $0\leq i \leq n+1$}
\end{equation*}
satisfying the relation
\begin{equation*}
\partial_i\partial_j = \partial_{j-1}\partial_i \qquad \text{for $i<j$.}
\end{equation*}
The \textit{geometric realization} $|\MK|$ of $\MK$ is obtained by taking the disjoint union over all $n\geq 0$ of one copy of the standard simplex $\Delta^n$ for every element of $K_n$ and identifying faces of these simplices according to the face maps $\partial_i$. A \textit{triangulation} of a topological space $X$ consists of a $\Delta$-complex $\MT$ together with a homeomorphism $|\MT|\cong X$.

\begin{definition}
A \textit{cocycle} $c$ on a $\Delta$-complex $\MK$ with values in a group $G$ is an assignment of an element $c(e)\in G$ to every $1$-simplex $e$ of $\MK$ equipped with a choice of orientation. This assignment is subject to the following conditions:
\begin{enumerate}
\item If $e$ is an oriented $1$-simplex of $\MK$ and $\overline{e}$ is the same $1$-simplex equipped with the opposite orientation, then $c(e) = c(\overline{e})^{-1}$.
\item Consider a $2$-simplex $a$ of $\MK$ and let $e_0,e_1,e_2$ be the sequence of oriented $1$-simplices obtained by going around the boundary of $a$ in either direction and starting at an arbitrary vertex. Then $c(e_2)c(e_1)c(e_0)=\operatorname{id}$.
\end{enumerate}
\label{def:cocycle}
\end{definition}

\begin{remark}
Recall that associated to a $\Delta$-complex $\MK$ and an abelian group $A$, we have a simplicial cochain complex $C^\bullet(\MK;A)$ with coefficients in $A$ whose cohomology agrees with the singular cohomology $H^*(|\MK|;A)$. If we replace the abelian group $A$ with a possibly non-abelian group $G$, the differential of this cochain complex is no longer well-defined. Nevertheless, it is still possible to make sense of the differentials $d^0$ in degree $0$ and $d^1$ in degree $1$ and we have $d^1\circ d^0 = 0$. The reason for this is that the set of boundary simplicies of an oriented simplex of dimension $1$ or $2$ comes equipped with a natural cyclic order. If $\MK$ is connected, the resulting cohomology groups are given by $G$ in degree $0$ and $\operatorname{Hom}(\pi_1(\MK),G)$ in degree $1$. Cocycles in the sense of Definition~\ref{def:cocycle} are precisely the elements of the kernel of the differential $d^1$ in degree $1$, i.e. $1$-cocycles.
\end{remark}

Consider a path $\alpha$ in a $\Delta$-complex $\MK$. Suppose that $\alpha$ is the concatenation of finitely many oriented $1$-simplices $e_0,\dots,e_k$ of $\MK$. These $1$-simplices are indexed in the order $\alpha$ traverses them, i.e.\ $\alpha$ starts with $e_0$ and ends with $e_k$. If $c$ is a cocycle on $\MK$ with values in $G$, then we define the value of $c$ on $\alpha$ to be the product $c(\alpha)\coloneqq c(e_k)\cdots c(e_0)$. Note that the value of $c$ on $\alpha$ only depends on the homotopy class of $\alpha$ with fixed end points. Let $\alpha$ and $\beta$ be paths such that the end point of $\alpha$ agrees with the starting point of $\beta$. Let us use the convention that the concatenation of the paths $\alpha$ and $\beta$ obtained by first following $\alpha$ and then following $\beta$ is denoted by $\beta\alpha$. Clearly, we have $c(\beta\alpha) = c(\beta)c(\alpha)$ for every cocycle $c$.

The following simple lemma explains the connection between products of commutators in $G$ and cocycles on triangulated surfaces with boundary. In the following sections, we will make extensive use of this connection. All surfaces are understood to be compact, connected, oriented, and to possibly have boundary.

\begin{lemma}
\label{lem:commutators_surfaces}
Let $G$ be a group and let $g\in G$ be an element. Moreover, let $k\geq 0$ be a non-negative integer. Then the following statements are equivalent:
\begin{enumerate}
\item $g$ can be written as a product of $k$ commutators in $G$.
\item There exist a triangulated surface $(\Sigma,\MT)$ of genus $k$ with exactly one boundary component, a cocycle $c$ on $\MT$, and a based loop $\gamma$ going around the boundary of $\Sigma$ such that $c(\gamma) = g$.
\end{enumerate}
\end{lemma}

\begin{proof}
We show that statement~(1) implies statement~(2). Write
\begin{equation}
\label{eq:commutators_surfaces_proof_a}
g = \prod\limits_{j=1}^k [u_j,v_j]
\end{equation}
for elements $u_j,v_j\in G$. Consider an oriented $(4k+1)$-gon $P$. Pick an arbitrary vertex $p_0$ of $P$. Starting at this vertex, we go around the boundary of $P$ once, following the boundary orientation. This yields an ordered list of $4k+1$ oriented edges of $P$. We assign the following elements of $G$ to these oriented edges in the order they appear in the list:
\begin{equation}
\label{eq:commutators_surfaces_proof_b}
g, u_k^{-1}, v_k^{-1}, u_k, v_k, \dots , u_1^{-1}, v_1^{-1}, u_1, v_1.
\end{equation}
We triangulate $P$ by drawing a straight line between the chosen vertex $p_0$ and every other vertex not already adjacent to $p_0$. Let $\MT$ denote the resulting triangulation. It is an easy consequence of identity~\eqref{eq:commutators_surfaces_proof_a} that if we go around the boundary of $P$ in positive direction and take the product of the labels~\eqref{eq:commutators_surfaces_proof_b}, we obtain the identity element of $G$. Here the order of the product is the same as in the definition of the evaluation of a cocycle on a path. Hence it is possible to extend the above assignment of group elements to oriented edges of $P$ to a cocycle $c$ on $\MT$. In fact, this extension is unique because of the special structure of $\MT$. We turn $P$ into a surface of genus $k$ by identifying $2k$ pairs of edges of $P$ via orientation reversing affine maps. More precisely, we identify the pairs of edges labelled by $u_k$ and $u_k^{-1}$ and the pairs of edges labelled by $v_k$ and $v_k^{-1}$. Let $\Sigma$ denote the resulting surface of genus $k$ with one boundary component. Both the triangulation $\MT$ of $P$ and the cocycle $c$ on $\MT$ descend to a triangulation and cocycle on $\Sigma$. We will abbreviate them by the same symbols $\MT$ and $c$. By construction, if we evaluate $c$ on the based loop going around the boundary of $\Sigma$ once in positive direction, we obtain $g$. This concludes the proof that statement~(1) implies statement~(2).

Next, we show that statement~(2) implies statement~(1). Let $\gamma$ be a loop in $\Sigma$ based at $p_0\in \partial\Sigma$ and going around $\partial\Sigma$ once such that $c(\gamma)=g$. In the fundamental group $\pi_1(\Sigma,p_0)$, the loop $\gamma$ is a product of $k$ commutators, i.e.\ we can write
\begin{equation}
\label{eq:commutators_surfaces_proof_c}
\gamma = \prod\limits_{j=1}^k [\alpha_j,\beta_j]
\end{equation}
for based loops $\alpha_j$ and $\beta_j$. Define $u_j \coloneqq c(\alpha_j)$ and $v_j\coloneqq c(\beta_j)$. Then evaluating $c$ on both sides of equation~\eqref{eq:commutators_surfaces_proof_c} exhibits $g$ as the product of the $k$ commutators $[u_j,v_j]$.
\end{proof}

\section{An application of Herman--Sergeraert's theorem on Diophantine torus rotations}
\label{sec:application_herman_sergeraert_diophantine_torus_rotations}

Let $T^{2n} \coloneqq \BR^{2n}/\BZ^{2n}$ denote the $2n$-dimensional torus. We equip $T^{2n}$ with the symplectic form induced by the standard symplectic form $\omega_0$ on $\BR^{2n}$. In this section we prove Theorem~\ref{thm:smooth_perfectness_closed} for $(M,\omega) = (T^{2n},\omega_0)$. More precisely, we prove a more technical statement, Proposition~\ref{prop:cocycles_torus} below, which will be useful in the following sections. The special case $(T^{2n},\omega_0)$ of Theorem~\ref{thm:smooth_perfectness_closed} easily follows from Proposition \ref{prop:cocycles_torus} in combination with the discussion in Section~\ref{sec:commutators_and_surfaces}, in particular the construction in the proof of Lemma~\ref{lem:commutators_surfaces}.

\begin{proposition}
\label{prop:cocycles_torus}
Let $n>0$ be a positive integer. Then there exists a triangulated surface $(\Sigma,\MT)$ with one boundary component and with one marked $0$-simplex $p_0\in \partial\Sigma$ such that the following is true. For every Hamiltonian diffeomorphism $\varphi\in \operatorname{Ham}(T^{2n})$ which is sufficienty close to the identity, there exists a cocycle $c$ on $\MT$ such that:
\begin{enumerate}
\item \label{item:cocycles_torus_boundary_evaluation} We have $c(\partial\Sigma,p_0) = \varphi$, where $(\partial\Sigma,p_0)$ is interpreted as the loop based at $p_0$ which goes around the boundary once in positive direction.
\item \label{item:cocycles_torus_smooth_dependence} The assignment $\varphi\mapsto c$ is smooth. By this we mean that the assignment $\varphi \mapsto c(e)$ is smooth for every oriented $1$-simplex $e$ of $\MT$.
\item \label{item:cocycles_torus_image_identity} We can arrange $\varphi = \operatorname{id}$ to be mapped arbitrarily close to the identity cocycle, i.e.\ the cocycle assigning $\operatorname{id}$ to every oriented $1$-simplex of $\MT$.
\item \label{item:cocycles_torus_Cal_zero} Suppose that $V\subset T^{2n}$ is a proper open subset and $\varphi \in \operatorname{Ham}_c^0(V)$. Then the restriction of $c$ to $\partial \Sigma$ takes values in $\operatorname{Ham}_c^0(V)$.
\end{enumerate}
\end{proposition}

\begin{corollary}
\label{cor:smooth_perfectness_torus}
The assertion of Theorem~\ref{thm:smooth_perfectness_closed} holds for $(M,\omega) = (T^{2n},\omega_0)$.
\end{corollary}

Our proof of Proposition~\ref{prop:cocycles_torus} has two main ingredients. The first is a deep theorem due to Herman and Sergeraert on Diophantine torus rotations \cite{hs71} (see also \cite[Theorems~2.3.3~and~4.4.2]{ban97}). The second ingredient is Lemma~\ref{lem:fragmentation_basic} above on fragmentation of Hamiltonian diffeomorphisms.

Given a vector $\alpha \in \BR^{2n}$, let
\begin{equation*}
R_\alpha: T^{2n} \rightarrow T^{2n} \qquad [x] \mapsto [x+\alpha]
\end{equation*}
denote the associated rotation. A vector $\alpha\in \BR^{2n}$ is said to be \textit{Diophantine} if it is badly approximable by rational vectors in $\BQ^{2n}$ in a quantitative sense (see e.g.\ \cite[Definition~2.3.2]{ban97}). The precise Diophantine condition will not be important for our discussion here, but let us point out that the set of Diophantine vectors has full Lebesgue measure.

\begin{theorem}[Herman--Sergeraert]
\label{thm:herman}
Let $\alpha\in \BR^{2n}$ be a Diophantine vector. Then there exist an open neighbourhood $\MN\subset\operatorname{Ham}(T^{2n})$ of the identity and a map
\begin{equation*}
s:\MN\rightarrow \operatorname{Ham}(T^{2n})
\end{equation*}
satisfying the following properties:
\begin{enumerate}
\item \label{item:herman_continuity} $s$ is smooth.
\item \label{item:herman_identity} $s(\operatorname{id}) = \operatorname{id}$
\item \label{item:herman_char_equation} For all $\phi\in\MN$, we have
\begin{equation*}
R_\alpha \phi = s(\phi)  R_\alpha s(\phi)^{-1}.
\end{equation*}
\end{enumerate}
\end{theorem}

\begin{remark}
In the Herman--Sergeraert theorem, smoothness is usually understood in the sense of maps between Fr\'echet manifolds. For diffeomorphism groups of closed manifolds, this is equivalent to smoothness in the diffeological sense, which is the default notion of smoothness in this paper.
\end{remark}

\begin{proof}[Proof of Proposition~\ref{prop:cocycles_torus}]
Fix a finite open cover $B_1,\dots,B_m$ of $T^{2n}$ by balls. Let $\alpha\in \BR^{2n}$ be Diophantine. By the isotopy extension theorem, there exists, for every $1\leq j\leq m$, a Hamiltonian diffeomorphism $R_j\in \operatorname{Ham}(T^{2n})$ such that
\begin{equation}
\label{eq:smooth_perfectness_torus_proof_Ri}
R_j|_{R_\alpha^{-1}(B_j)}=R_\alpha|_{R_\alpha^{-1}(B_j)}.
\end{equation}
We can choose $R_j$ to be arbitrarily $C^\infty$ close to the identity if we pick $\alpha$ sufficiently close to $0$.

Let $\varphi\in \operatorname{Ham}(T^{2n})$ be close to the identity. By Theorem~\ref{thm:herman} there exists $\psi\in \operatorname{Ham}(T^{2n})$, smoothly dependent on $\varphi$, such that $R_\alpha \varphi = \psi R_\alpha \psi^{-1}$. By Lemma~\ref{lem:fragmentation_basic} we can write $\psi = \psi_1\cdots \psi_m$, where $\psi_j\in \operatorname{Ham}_c(B_j)$ depends smoothly on $\psi$ and hence on $\varphi$. Using identity~\eqref{eq:smooth_perfectness_torus_proof_Ri} and the fact that the support of $\psi_j$ is contained in $B_j$, we deduce that $R_\alpha^{-1} \psi_j R_\alpha = R_j^{-1} \psi_j R_j$. We may therefore compute
\begin{eqnarray}
\label{eq:smooth_perfectness_torus_proof_phi_product_computation}
\varphi &=& R_\alpha^{-1}\psi R_\alpha\psi^{-1} \\
&=& R_\alpha^{-1}\psi_1\cdots\psi_m R_\alpha\psi_m^{-1}\cdots \psi_1^{-1} \nonumber\\
&=& (R_\alpha^{-1}\psi_1R_\alpha)\cdots(R_\alpha^{-1}\psi_mR_\alpha)\psi_m^{-1}\cdots\psi_1^{-1} \nonumber\\
&=& (R_1^{-1}\psi_1R_1)\cdots(R_m^{-1}\psi_mR_m)\psi_m^{-1}\cdots\psi_1^{-1}. \nonumber
\end{eqnarray}
Now consider an oriented $(4m+1)$-gon $P$ with one marked vertex $p_0$. Starting at $p_0$, we traverse the boundary in positive direction. We label the oriented edges we meet along the way by the following diffeomorphisms, in the order we encounter them (see Figure~\ref{fig:cocycle_torus_polygon}):
\begin{equation}
\label{eq:smooth_perfectness_torus_proof_label_list}
\varphi, R_1, \psi_1^{-1}, R_1^{-1}, \dots, R_m, \psi_m^{-1}, R_m^{-1}, \psi_m, \dots, \psi_1.
\end{equation}
\begin{figure}
\includegraphics[width=0.5\textwidth]{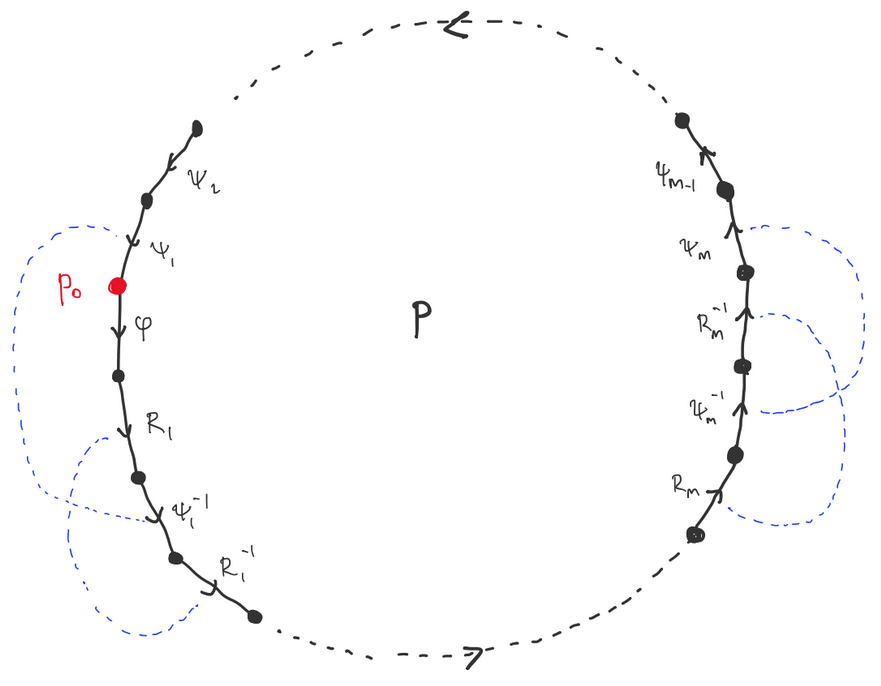}
\caption{We assign the diffeomorphisms from the list~\eqref{eq:smooth_perfectness_torus_proof_label_list} to the edges of $P$ with the orientation displayed. Pairs of edges are identified via orientation reversing maps as indicated by the blue dashed lines.}
\label{fig:cocycle_torus_polygon}
\end{figure}
Let us construct a triangluation $\MT$ of $P$ by connecting the vertex $p_0$ of $P$ to every other vertex it is not already adjacent to (see Figure~\ref{fig:cocycle_torus_triangulation}).
\begin{figure}
\includegraphics[width=0.3\textwidth]{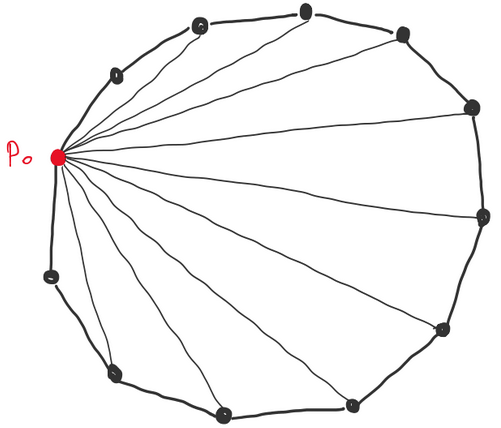}
\caption{Triangulation $\MT$ of $P$.}
\label{fig:cocycle_torus_triangulation}
\end{figure}
Note that it follows from equation~\eqref{eq:smooth_perfectness_torus_proof_phi_product_computation} that if we go around the boundary of $P$ and take the composition of all the diffeomorphisms along the way, we obtain the identity diffeomorphism. Hence the assignment of diffeomorphisms to oriented boundary edges of $P$ displayed in Figure~\ref{fig:cocycle_torus_polygon} extends to a cocycle $c$ on the triangulation $\MT$. In fact, it follows from the special structure of $\MT$ that this extension is unique. For every $1\leq j \leq m$, we identify the edges of $P$ labelled by $R_j$ and $R_j^{-1}$ using an affine map which is orientation reversing (see Figure~\ref{fig:cocycle_torus_polygon}). Similarly, we identify the edges of $P$ labelled by $\psi_j$ and $\psi_j^{-1}$ using an orientation reversing affine map. The result of these identifications is an oriented surface $\Sigma$ of genus $m$ with one boundary component. The triangulation and cocycle on $P$ descend to a triangulation and cocycle on $\Sigma$, still denoted by $\MT$ and $c$. We need to check that the triangulated surface $(\Sigma,\MT)$ and the assignment $\varphi\mapsto c$ satisfy all the properties listed in Proposition~\ref{prop:cocycles_torus}.

Property~\eqref{item:cocycles_torus_boundary_evaluation} is clear because $\MT$ has exactly one boundary $1$-simplex, which, if equipped with the boundary orientation, is assigned the diffeomorphism $\varphi$. This also verifies property~\eqref{item:cocycles_torus_Cal_zero}. Property~\eqref{item:cocycles_torus_smooth_dependence} holds because all the diffeomorphisms $\psi_j$ depend smoothly on $\psi$ and hence on $\varphi$. If $\varphi=\operatorname{id}$, then $\psi = \operatorname{id}$ and therefore $\psi_i = \operatorname{id}$. Thus the associated cocycle $c$ takes values in the set $\left\{\operatorname{id}, R_j, R_{j}^{-1} \mid 1\leq j\leq m \right\}$. Since these diffeomorphisms can be arranged to be arbitrarily close to the identity, the same is true for $c$. This verifies property~\eqref{item:cocycles_torus_image_identity}.
\end{proof}

\section{Fragmentation}
\label{sec:fragmentation}

The goal of this section is to prove the following refinement of Proposition~\ref{prop:cocycles_torus} in the presence of an open covering.

\begin{proposition}
\label{prop:cocycles_torus_fragmented}
Let $n>0$ be a positive integer and let $\MU$ be a finite open cover of $T^{2n}$. Then there exists a triangulated surface $(\Sigma,\MT)$ with exactly one boundary component such that the total number of simplices of $\MT$ is bounded by a constant only depending on $n$ and $|\MU|$ and such that the following is true: For every Hamiltonian diffeomorphism $\varphi\in \operatorname{Ham}(T^{2n})$ sufficiently close to the identity, there exists a cocycle $c$ on $\MT$ such that:
\begin{enumerate}
\item \label{item:cocycles_torus_fragmented_boundary_evaluation} We have $c(\partial\Sigma,p_0) = \varphi$. Here $p_0\in \partial\Sigma$ is a marked $0$-simplex independent of $\varphi$ and $(\partial\Sigma,p_0)$ refers to the loop based at $p_0$ going around the boundary once in positive direction.
\item \label{item:cocycles_torus_fragmented_support} For each $2$-simplex $\sigma$ of $\MT$, there exist open sets $U,U'\in \MU$ (only depending on $\sigma$ but not on $\varphi$) such that $U\cap U'\neq \emptyset$ and the restriction of $c$ to $\sigma$ takes values in $\operatorname{Ham}_c(U\cup U')$.
\item \label{item:cocycles_torus_fragmented_smooth_dependence} The assignment $\varphi\mapsto c$ is smooth.
\item \label{item:cocycles_torus_fragmented_identity} We can arrange the identity to be mapped arbitrarily close to the identity cocycle.
\item \label{item:cocycles_torus_fragmented_support_zero_cal} Suppose that $V\subset T^{2n}$ is a proper open subset and $\varphi\in \operatorname{Ham}_c^0(V)$. Then the restriction of $c$ to $\partial\Sigma$ takes values in $\operatorname{Ham}_c(V)$.
\end{enumerate}
\end{proposition}

\subsection{\texorpdfstring{Fragmentation of $1$-simplices}{Fragmentation of 1-simplices}}
\label{subsec:fragmentation_1_simplices}

Fix a connected symplectic manifold $(M^{2n},\omega)$, possibly with boundary and not necessarily compact. Let $\MU = (U_i)_{1\leq i \leq k}$ be an open cover of $M$. Assume that all $U_i$ are proper subsets of $M$, i.e. $M\setminus U_i\neq \emptyset$. Let us identify the standard $1$-simplex $\Delta^1$ with the interval $[0,k]$. Let $\tilde{\Delta}^1$ denote the subdivision of $\Delta^1$ into the $k$ intervals $[i-1,i]$ for $1\leq i \leq k$.

\begin{lemma}
\label{lem:fragmentation_1d}
For every cocycle $c$ on $\Delta^1$ which is sufficiently close to the identity cocycle, there exists a fragmented cocycle $\tilde{c}$ on $\tilde{\Delta}^1$ satisfying the following properties:
\begin{enumerate}
\item \label{item:fragmentation_1d_total_eval} The values of $c$ and $\tilde{c}$ agree on the path traversing $[0,k]\cong \Delta^1$ from $0$ to $k$.
\item \label{item:fragmentation_1d_support} We have $\tilde{c}([i-1,i])\in \operatorname{Ham}_c(U_i)$ for $1\leq i \leq k$. Here the $1$-simplex $[i-1,i]$ is oriented to point from $i-1$ to $i$.
\item \label{item:fragmentation_1d_smooth_dependence} The assignment $c\mapsto \tilde{c}$ is smooth. Moreover, the identity cocycle is mapped to the identity cocycle.
\item \label{item:fragmentation_1d_Cal_zero} Let $V\subset M$ be an open subset. If $M$ is closed and $c$ takes values in $\operatorname{Ham}_c^0(V)$ or if $M$ is not closed and $c$ takes values in $\operatorname{Ham}_c(V)$, then $\tilde{c}([i-1,i])\in \operatorname{Ham}_c(V\cap U_i)$.
\item \label{item:fragmentation_1d_Cal_sum} If $M$ is closed, we have $\sum_{1\leq i\leq k} \operatorname{Cal}_{U_i}(\tilde{c}([i-1,i])) = 0$ where $\operatorname{Cal}_{U_i}$ denotes the Calabi homomorphism on $\operatorname{Ham}_c(U_i)$.
\end{enumerate}
\end{lemma}

\begin{proof}
This lemma is a staightforward consequence of Lemma~\ref{lem:fragmentation_basic}. Indeed, set $\varphi \coloneqq c([0,k])$. Lemma~\ref{lem:fragmentation_basic} then yields a fragmentation $\varphi = \varphi_k \circ \cdots \circ \varphi_1$ with $\varphi_i \in \operatorname{Ham}_c(U_i)$. Set $\tilde{c}([i-1,i]) \coloneqq \varphi_i$ for $1 \leq i \leq k$. Clearly, this fragmented $1$-simplex $\tilde{c}$ satisfies properties~\eqref{item:fragmentation_1d_total_eval} and~\eqref{item:fragmentation_1d_support}. Properties~\eqref{item:fragmentation_1d_smooth_dependence}-\eqref{item:fragmentation_1d_Cal_sum} are immediate from the corresponding statements in Lemma~\ref{lem:fragmentation_basic}.
\end{proof}

\subsection{\texorpdfstring{Fragmentation of $2$-simplices}{Fragmentation of 2-simplices}}

As in Section~\ref{subsec:fragmentation_1_simplices}, let $(M^{2n},\omega)$ be a connected symplectic manifold, possibly with boundary and not necessarily compact, and let $\MU=(U_i)_{1\leq i\leq k}$ be an open cover of $M$ by proper open subsets $U_i$. Let us regard the $2$-simplex $\Delta^2$ as the subset
\begin{equation*}
\Delta^2 = \left\{ (s,t)\in \BR^2 \mid 0\leq t\leq s\leq k \right\}\subset \BR^2.
\end{equation*}
We consider the subdivision $\tilde{\Delta}^2$ of $\Delta^2$ into $k^2$ $2$-simplices displayed in Figure~\ref{fig:subdivision}. The $0$-simplices of $\tilde{\Delta}^2$ are the points $(i,j)$ for all integers $0\leq j\leq i\leq k$. The $1$-simplices can be grouped into horizontal $1$-simplices $h_{ij}$ connecting $(i-1,j)$ to $(i,j)$, vertical $1$-simplices $v_{ij}$ connecting $(i,j-1)$ to $(i,j)$ and diagonal $1$-simplices $d_{ij}$ connecting $(i-1,j-1)$ to $(i,j)$. Note that if we intersect the subdivision $\tilde{\Delta}^2$ with a boundary $1$-simplex $\partial_i\Delta^2$ of $\Delta^2$, then we obtain the subdivision of $\Delta^1\cong \partial_i\Delta^2$ considered in Section~\ref{subsec:fragmentation_1_simplices}.
\begin{figure}
\includegraphics[width=0.5\textwidth]{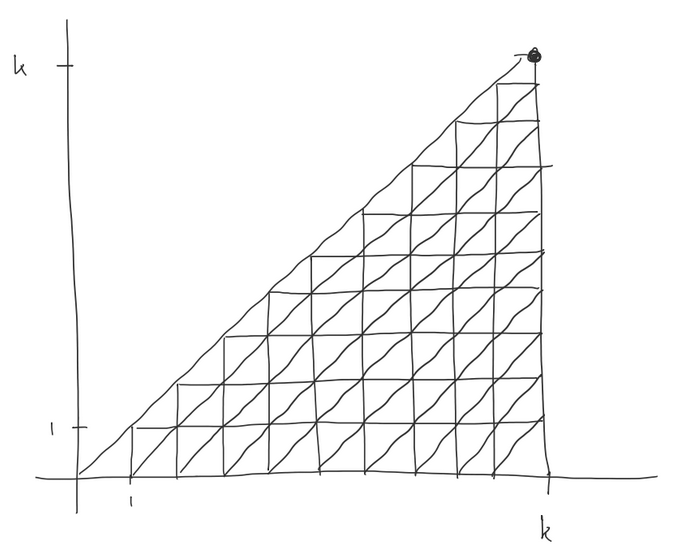}
\caption{Subdivision $\tilde{\Delta}^2$ of $\Delta^2$.}
\label{fig:subdivision}
\end{figure}

\begin{lemma}
\label{lem:fragmentation_2d}
For every cocycle $c$ on $\Delta^2$ all of whose values are sufficiently close to the identity, there exists a fragmented cocycle $\tilde{c}$ on $\tilde{\Delta}^2$ with the following properties:
\begin{enumerate}
\item \label{item:fragmentation_2d_commutativity_boundary} Fragmentation commutes with boundary maps, i.e.\ for every $0\leq i \leq 2$, the restriction of $\tilde{c}$ to $\partial_i\Delta^2$ agrees with the fragmentation of the restriction of $c$ to $\partial_i\Delta^2$ provided by Lemma~\ref{lem:fragmentation_1d}.
\item \label{item:fragmentation_2d_support} We have
\begin{equation*}
\tilde{c}(h_{ij})\in \operatorname{Ham}_c(U_i), \quad \tilde{c}(v_{ij})\in \operatorname{Ham}_c(U_j) \quad\text{and}\quad \tilde{c}(d_{ij})\in \operatorname{Ham}_c(U_i\cup U_j).
\end{equation*}
\item \label{item:fragmentation_2d_smooth_dependence} The assignment $c\mapsto c'$ is smooth. Moreover, the identity cocycle is mapped to the identity cocycle.
\end{enumerate}
\end{lemma}

\begin{proof}
Let $c$ be a cocycle on $\Delta^2$. We will construct Hamiltonian diffeomorphisms $\psi_{ij}\in \operatorname{Ham}(M)$ for all integers $0\leq j\leq i\leq k$ such that
\begin{equation}
\label{eq:fragmentation_2d_proof_desired_supports_a}
\psi_{ij}\psi_{i-1,j}^{-1} \in \operatorname{Ham}_c(U_i) \quad\text{and}\quad \psi_{ij}\psi_{i,j-1}\in \operatorname{Ham}_c(U_j).
\end{equation}
Given $\psi_{ij}$, we define the fragmented cocycle $\tilde{c}$ on $\tilde{\Delta}^2$ by setting
\begin{equation}
\label{eq:fragmentation_2d_proof_definition_cocycle}
\tilde{c}(h_{ij})\coloneqq \psi_{ij} \psi_{i-1,j}^{-1} \qquad \tilde{c}(v_{ij}) \coloneqq \psi_{ij} \psi_{i,j-1}^{-1} \qquad \tilde{c}(d_{ij})\coloneqq \psi_{ij} \psi_{i-1,j-1}^{-1}.
\end{equation}
If we normalize $\psi_{00}= \operatorname{id}$, then the identities~\eqref{eq:fragmentation_2d_proof_definition_cocycle} and property~\eqref{item:fragmentation_2d_commutativity_boundary} in Lemma~\ref{lem:fragmentation_2d} fully determine $\psi_{ij}$ for $(i,j)\in \partial\Delta^2$. For every $(i,j)\in \partial\Delta^2$, let $W_{ij}$ be a generating function for $\psi_{ij}$. These generating functions are only determined up to an additive constant.

If $M$ is not closed, let us normalize all these generating functions to be compactly supported and vanish on $\partial M$. It is a consequence of property~\eqref{item:fragmentation_1d_support} in Lemma~\ref{lem:fragmentation_1d} and Lemma~\ref{lem:difference_of_generating_functions} that
\begin{equation}
\label{eq:fragmentation_2d_proof_normalization_generating_functions}
\operatorname{supp}(W_{i0} - W_{i-1,0}) \subset U_i, \quad \operatorname{supp}(W_{kj}-W_{k,j-1})\subset U_j \quad\text{and}\quad \operatorname{supp}(W_{jj}-W_{j-1,j-1})\subset U_j
\end{equation}
for all $i$ and $j$.

We claim that if $M$ is closed, then there is a unique way of normalizing the generating functions $W_{ij}$ for $(i,j)\in \partial\Delta^2$ such that $W_{00}=0$ and such that all inclusions in \eqref{eq:fragmentation_2d_proof_normalization_generating_functions} hold. First, we observe that there is a unique way of normalizing the generating functions such that $W_{00}=0$ and all inclusions \eqref{eq:fragmentation_2d_proof_normalization_generating_functions} hold, except possibly for the inclusion $\operatorname{supp}(W_{kk}-W_{k-1,k-1})\subset U_k$. Let $W_{kk}'$ be the unique generating function for $\psi_{kk}$ such that $\operatorname{supp}(W_{kk}'-W_{k-1,k-1})\subset U_k$. We need to verify that $W_{kk}' = W_{kk}$. Using Lemma~\ref{lem:change_of_generating_function_under_chain_of_compactly_supported_ham_diffeos}, we compute
\begin{equation*}
\int_M (W_{kk}-W_{kk}') \omega^n = - \sum\limits_{i=1}^k \left( \operatorname{Cal}_{U_i}(\psi_{i-1,i-1}\psi_{ii}^{-1}) + \operatorname{Cal}_{U_i}(\psi_{i0}\psi_{i-1,0}^{-1}) + \operatorname{Cal}_{U_i}(\psi_{ki}\psi_{k,i-1}^{-1}) \right) = 0.
\end{equation*}
Here the last equality uses property~\eqref{item:fragmentation_1d_Cal_sum} in Lemma~\ref{lem:fragmentation_1d}. Since $W_{kk}$ and $W_{kk}'$ differ by a constant, this implies that $W_{kk} = W_{kk}'$, as desired. This concludes our proof that there exists a normalization of the generating functions satisfying $W_{00}=0$ and \eqref{eq:fragmentation_2d_proof_normalization_generating_functions} in the case that $M$ is closed. From now on, let us fix this normalization.

Our next step is to construct functions $W_{ij}$ for $(i,j)\notin \partial\Delta^2$ such that
\begin{equation}
\label{eq:fragmentation_2d_proof_support_condition_induction}
\operatorname{supp}(W_{ij}-W_{i-1,j}) \subset U_i \quad\text{and}\quad \operatorname{supp}(W_{ij}-W_{i,j-1})\subset U_j
\end{equation}
hold for all $i$ and $j$. We construct such functions inductively on $j$. Let $0<j<k$ and suppose that $W_{i\ell}$ for all $\ell<j$ have already been constructed subject to conditions~\eqref{eq:fragmentation_2d_proof_support_condition_induction}. Set $U^{(j)}\coloneqq \bigcup_{i>j}U_i$ and choose a smooth partition of unity $(\lambda_i^{(j)})_{j<i\leq k}$ on $U^{(j)}$ subordinate to the open cover $(U_i)_{j<i\leq k}$. Note that while the functions $\lambda_i^{(j)}$ are contained in $C^\infty(U^{(j)})$, they might not extend to smooth functions on $M$. Set
\begin{equation*}
\mu_i^{(j)} \coloneqq \sum\limits_{\ell = j+1}^i \lambda_{\ell}^{(j)}.
\end{equation*}

We abbreviate
\begin{equation*}
W^{(j)} \coloneqq (W_{kj} - W_{jj}) - (W_{k,j-1} - W_{j,j-1}).
\end{equation*}
We claim that $W^{(j)}$ is supported in $U^{(j)}\cap U_j$. Indeed, we can write
\begin{equation*}
W_{kj}-W_{jj} = \sum_{j<\mu\leq k}(W_{\mu\mu}-W_{\mu-1,\mu-1}) - (W_{k\mu}-W_{k,\mu-1}).
\end{equation*}
All terms in this sum are supported in $U^{(j)}$ by \eqref{eq:fragmentation_2d_proof_normalization_generating_functions}. Hence $W_{kj}-W_{jj}$ is supported in $U^{(j)}$. A similar argument using the inductive hypothesis shows that $W_{k,j-1}-W_{j,j-1}$ is supported in $U^{(j)}$. We conclude that $W^{(j)}$ is supported in $U^{(j)}$ as well. Next, we observe that
\begin{equation*}
W^{(j)} = (W_{kj} - W_{k,j-1}) + (W_{j,j-1}-W_{j-1,j-1}) - (W_{jj}-W_{j-1,j-1}).
\end{equation*}
The first and the last term are supported in $U_j$ by our normalization~\eqref{eq:fragmentation_2d_proof_normalization_generating_functions}. The second term is supported in $U_j$ by the inductive hypothesis. Thus $W^{(j)}$ is supported in $U_j$, proving the claim.

For $j<i<k$, we set
\begin{equation*}
W_{ij} \coloneqq W_{jj} + (W_{i,j-1}-W_{j,j-1}) + \mu_i^{(j)}\cdot W^{(j)}.
\end{equation*}
Since $\operatorname{supp}W^{(j)}\subset U^{(j)}$, this is a well-defined smooth function on $M$, even though $\mu_i^{(j)}$ is only defined on $U^{(j)}$ and might not extend to a smooth function on $M$.

Let us check that the functions $W_{ij}-W_{i-1,j}$ are supported in $U_i$. We compute
\begin{equation*}
W_{ij}-W_{i-1,j} = (W_{i,j-1}-W_{i-1,j-1}) + \lambda_i^{(j)}\cdot W^{(j)}.
\end{equation*}
The first term is supported in $U_i$ by the inductive hypothesis. The second term is supported in $U_i$ because the support of $\lambda_i^{(j)}$ (viewed as a function on $U^{(j)}$) is contained in $U_i$ and $W^{(j)}$ has compact support in $U^{(j)}$. This shows that $W_{ij}-W_{i-1,j}$ is supported in $U_i$.

Next, we check that $W_{ij}-W_{i,j-1}$ is supported in $U_j$. We compute
\begin{equation*}
W_{ij}-W_{i,j-1} = (W_{jj}-W_{j-1,j-1}) - (W_{j,j-1}-W_{j-1,j-1}) + \mu_i^{(j)}\cdot W^{(j)}.
\end{equation*}
The first two terms are supported in $U_j$ by our normalization \eqref{eq:fragmentation_2d_proof_normalization_generating_functions} and the indutive hypothesis. The last term is supported in $U_j$ because $W^{(j)}$ is, as shown above. Hence $W_{ij}-W_{i,j-1}$ is supported in $U_j$. This concludes our construction of functions $W_{ij}$ satisfying \eqref{eq:fragmentation_2d_proof_support_condition_induction}.

We define $\psi_{ij}$ to be the Hamiltonian diffeomorphism with generating function $W_{ij}$. It is straightforward to see from our construction that the resulting cocylce $\tilde{c}$ satisfies all desired properties in Lemma~\ref{lem:fragmentation_2d}.
\end{proof}

\subsection{Proof of Proposition \ref{prop:cocycles_torus_fragmented}}

We begin by applying Proposition~\ref{prop:cocycles_torus}. This yields a triangulated surface $(\Sigma,\MT)$ with a marked boundary $0$-simplex $p_0$. Given a Hamiltonian diffeomorphism $\varphi\in \operatorname{Ham}(T^{2n})$ sufficiently close to the identity, we get a cocycle $c$ on $\MT$. Suppose that $\MU=(U_j)_{1\leq j\leq k}$ is an open cover of $T^{2n}$. Let us subdivide every $1$-simplex of $\MT$ into $k$ $1$-simplices and every $2$-simplex of $\MT$ into $k^2$ $2$-simplices as indicated in Figure~\ref{fig:subdivision}. Let $\tilde{\MT}$ be the resulting triangulation of $\Sigma$. Observe that the number of simplices in $\tilde{\MT}$ is bounded by a constant only depending on the number of simplices in $\MT$ and $k$. Applying Lemma~\ref{lem:fragmentation_2d} to the restriction of the cocycle $c$ to every $2$-simplex of $\MT$ yields a fragmented cocycle $\tilde{c}$ on $\tilde{\MT}$. Note that property~\eqref{item:fragmentation_2d_commutativity_boundary} in Lemma~\ref{lem:fragmentation_2d} ensures compatibility between adjacent $2$-simplices of $\MT$.

We claim that the triangulated surface $(\Sigma,\tilde{\MT})$ and the assingment $\varphi\mapsto \tilde{c}$ satisfy all properties listed in Propsiton \ref{prop:cocycles_torus_fragmented}, except possibly for property \eqref{item:cocycles_torus_fragmented_support}. Indeed, it is immediate from property~\eqref{item:fragmentation_1d_total_eval} in Lemma~\ref{lem:fragmentation_1d} and property~\eqref{item:cocycles_torus_boundary_evaluation} in Proposition~\ref{prop:cocycles_torus} that $\tilde{c}$ evaluates to $\varphi$ on the boundary of $\Sigma$. Smoothness of the assignment $\varphi\mapsto \tilde{c}$ follows from property~\eqref{item:fragmentation_2d_smooth_dependence} in Lemma~\ref{lem:fragmentation_2d} and property~\eqref{item:cocycles_torus_smooth_dependence} in Proposition~\ref{prop:cocycles_torus}. It is a consequence of property~\eqref{item:fragmentation_2d_smooth_dependence} in Lemma~\ref{lem:fragmentation_2d} and property~\eqref{item:cocycles_torus_image_identity} in Proposition~\ref{prop:cocycles_torus_fragmented} that we can arrange $\varphi=\operatorname{id}$ to be mapped arbitrarily close to the identity cocycle on $\tilde{\MT}$. Finally, if $\varphi\in \operatorname{Ham}_c^0(V)$ for a proper open subset $V$, then property~\eqref{item:fragmentation_1d_Cal_zero} in Lemma~\ref{lem:fragmentation_1d} and property~\eqref{item:cocycles_torus_Cal_zero} in Propsition~\ref{prop:cocycles_torus} imply that the restriction of $\tilde{c}$ to $\partial\Sigma$ takes values in $\operatorname{Ham}_c(V)$.

Observe that it follows from property~\eqref{item:fragmentation_2d_support} in Lemma~\ref{lem:fragmentation_2d} that the restriction of $\tilde{c}$ to any $2$-simplex is supported in the union $U_i\cup U_j$ of two open subsets contained in the cover. However, it is possible that $U_i\cap U_j = \emptyset$ for some $2$-simplices. In order to also guarantee property \eqref{item:cocycles_torus_fragmented_support}, we will cut out these $2$-simplices and close the resulting holes in $\Sigma$ by identifying pairs of boundary edges via orientation reversing affine maps. Let us now explain this process in more detail. Consider a $2$-simplex of $\MT$ and identify it with $\Delta^2$. By construction, the restriction of $\tilde{\MT}$ to this $2$-simplex is simply given by $\tilde{\Delta}^2$. For every pair of indices $0< j < i\leq k$ such that $U_i\cap U_j = \emptyset$, consider the square $Q_{ij}$ with vertices $(i,j),(i-1,j),(i-1,j-1),(i,j-1)$ (see Figure \ref{fig:remove_bad_square}).
\begin{figure}
\includegraphics[width=0.5\textwidth]{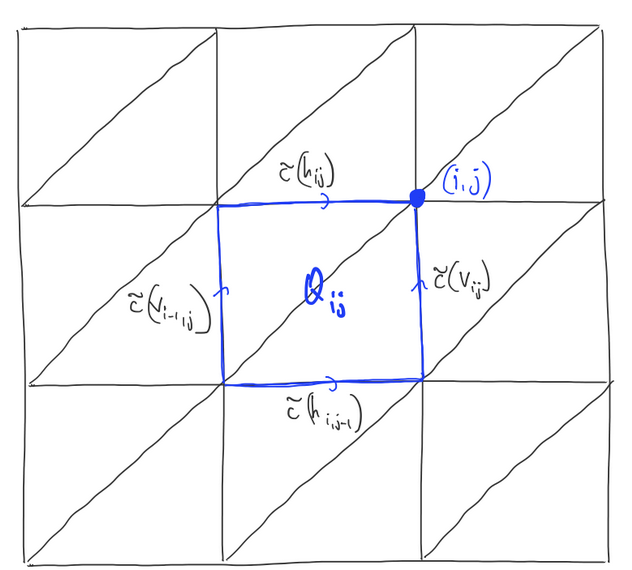}
\caption{The square $Q_{ij}$.}
\label{fig:remove_bad_square}
\end{figure}
Note that $Q_{ij}$ consists of two $2$-simplices of $\tilde{\Delta}^2$. These are precisely the $2$-simplices potentially violating property \eqref{item:cocycles_torus_fragmented_support}. In the notation of Lemma \ref{lem:fragmentation_2d}, we have
\begin{equation}
\label{eq:cocycles_torus_fragmented_proof_a}
\tilde{c}(v_{ij})\tilde{c}(h_{i,j-1}) = \tilde{c}(h_{ij})\tilde{c}(v_{i-1,j}).
\end{equation}
The diffeomorphisms $\tilde{c}(v_{ij})$ and $\tilde{c}(v_{i-1,j})$ are supported in $U_j$ and the diffeomorphisms $\tilde{c}(h_{i,j-1})$ and $\tilde{c}(h_{ij})$ are supported in $U_j$. Since $U_i$ and $U_j$ are disjoint, equality \eqref{eq:cocycles_torus_fragmented_proof_a} implies that
\begin{equation}
\label{eq:cocycles_torus_fragmented_proof_b}
\tilde{c}(v_{ij}) = \tilde{c}(v_{i-1,j})\quad \text{and} \quad \tilde{c}(h_{i,j-1}) = \tilde{c}(h_{ij}).
\end{equation}
Let us now cut the square $Q_{ij}$ out of $\Sigma$ (in the case $Q_{ij}$ touches the boundary of $\Sigma$, we do not remove the boundary edge). The resulting hole has four boundary edges and we identify the two pairs of opposite edges via orietation reversing affine maps. By equation~\eqref{eq:cocycles_torus_fragmented_proof_b} the cocycle $\tilde{c}$ descends to the quotient. Now repeat this process for all pairs of indices $i$ and $j$ as above and all $2$-simplices of $\MT$. Let $(\Sigma',\MT')$ be the resulting triangulated surface and $c'$ the resulting cocycle. Since $\MT'$ is obtained from $\tilde{\MT}$ be removing $2$-simplices, it is clear that the number of simplices in $\MT'$ is still bounded by a constant only depending on $n$ and $k$. Moreover, the cocycle $c'$ satisfies all the desired properties listed in Proposition~\ref{prop:cocycles_torus_fragmented}.

This concludes our proof of Proposition~\ref{prop:cocycles_torus_fragmented}. \qed

\section{Special open covers}
\label{sec:special_open_covers}

Let us begin with the following straightforward lemma concerning special refinements of an open cover.

\begin{lemma}
\label{lem:special_cover_refinement}
Let $\MV$ be a finite open cover of a normal topological space $X$. Then there exists a finite open cover $\MU$ of $X$ such that:
\begin{enumerate}
\item \label{item:special_cover_refinement_cardinality} $|\MU|$ only depends on $|\MV|$;
\item \label{item:special_cover_refinement_inclusion} for any $U,U'\in \MU$ such that $U\cap U'\neq \emptyset$, there exists $V\in \MV$ such that $U\cup U'\subset V$.
\end{enumerate}
\end{lemma}

\begin{proof}
Let us write $\MV = (V_i)_{i\in I}$ for some finite index set $I$. Since $X$ is a normal topological space, there exists an open cover $(W_i)_{i\in I}$ of $X$ such that $\overline{W}_i\subset V_i$. For every subset $J\subset I$, let us define the closed subset
\begin{equation*}
A_J \coloneqq \left( \bigcap\limits_{j\in J} \overline{W}_j \right) \setminus \left( \bigcup\limits_{i\notin J} W_i \right).
\end{equation*}
We claim that the sets $A_J$ for all $J\subset I$ cover $X$. Indeed, given $x\in X$, simply choose $J\subset I$ maximal with the property that $x\in \bigcap_{j\in J} W_j$. Since $x$ cannot be contained in $W_i$ for $i\notin J$, we see that $x\in A_J$.

Now suppose that $A_{J_1}\cap A_{J_2}\neq \emptyset$ for two subsets $J_1,J_2\subset I$. We claim that this implies $J_1\cap J_2 \neq \emptyset$. Indeed, pick $x \in A_{J_1} \cap A_{J_2}$. Clearly, $x$ is not contained in the union $\bigcup_{i\in J_1^c \cup J_2^c} W_i$. Since the $W_i$ form an open cover of $X$, we see that $J_1^c\cup J_2^c \neq I$, or in other words $J_1\cap J_2\neq \emptyset$.

For every $J\subset I$, pick an open set $U_J$ such that:
\begin{enumerate}
\item $A_J\subset U_J\subset \bigcap_{j\in J} V_j$
\item For any two $J_1$ and $J_2$ such that $A_{J_1}\cap A_{J_2} = \emptyset$, we have $U_{J_1}\cap U_{J_2} = \emptyset$.
\end{enumerate}
This is possible because $X$ is normal. Since the sets $A_J$ form a cover of $X$, so do the sets $U_J$. Suppose that $U_{J_1}\cap U_{J_2} \neq \emptyset$. By construction, we must have $A_{J_1}\cap A_{J_2} \neq \emptyset$. Recall that this implies that there exists $i\in J_1\cap J_2$. The union $U_{J_1}\cup U_{J_2}$ must be contained in $V_i$. Thus $(U_J)_{J\subset I}$ is an open cover of $X$ satisfying property~\eqref{item:special_cover_refinement_inclusion}. Clearly, the cardinality $|\MP(I)|$ of the cover only depends on $|I|$.
\end{proof}

\begin{proposition}
\label{prop:special_covers}
Let $n>0$ be a positive integer and consider balls $B'\Subset B\Subset T^{2n}$. Then for every $\varepsilon>0$, there exist:
\begin{itemize}
\item a finite open cover $\MV$ of $T^{2n}$;
\item a Hamiltonian diffeomorphism $\varphi_V\in \operatorname{Ham}(T^{2n})$ for every $V\in \MV$;
\item a compactly supported Hamiltonian diffeomorphism $\varphi_{V,W}\in \operatorname{Ham}_c(B)$ for every pair $V,W\in \MV$
\end{itemize}
satisfying the following properties:
\begin{enumerate}
\item \label{item:special_covers_bound_V} $|\MV|$ is bounded from above by a constant independent of $\varepsilon$;
\item \label{item:special_covers_Bprime} $B'\in \MV$ and $\varphi_{B'} = \operatorname{id}$;
\item \label{item:special_covers_inclusion} for all $V\in \MV$, we have $\varphi_V(V)\subset B$;
\item \label{item:special_covers_transition_diffeomorphisms} for all $V,W\in \MV$, we have $\varphi_{W,V}\circ \varphi_V|_{V\cap W} = \varphi_W|_{V\cap W}$;
\item \label{item:special_covers_Cinfty_control} for all $V,W\in \MV$, we have $\operatorname{dist}_{C^{\infty}}(\varphi_{V,W},\operatorname{id})\leq \varepsilon$.
\end{enumerate} 
\end{proposition}

\begin{remark}
Note that Proposition~\ref{prop:special_covers} is a refinement of Lemma~4.5.2 in Banyaga's book~\cite{ban97}, which does not include properties~\eqref{item:special_covers_bound_V} and~\eqref{item:special_covers_Cinfty_control}. It is relatively easy to achieve property~\eqref{item:special_covers_Cinfty_control} if $|\MV|$ is allowed to depend on $\varepsilon$. What makes the proof of Proposition~\ref{prop:special_covers} more subtle is that we need to keep $|\MV|$ bounded as $\varepsilon$ goes to zero.
\end{remark}

\begin{proof}
Let us first construct an open cover $\MV$ and diffeomorphisms $\varphi_V$ and $\varphi_{V,W}$ satisfying all the properties listed in Proposition \ref{prop:special_covers} except possibly for property \eqref{item:special_covers_Bprime}. Once we are done with this, we explain how to adapt the construction in order to also achieve property \eqref{item:special_covers_Bprime}.

Let us call a simplex in $T^{2n}$ affine if it is the image of some non-degenerate affine simplex in $\BR^{2n}$ under the natural covering map $\BR^{2n}\rightarrow T^{2n}$. We fix a triangulation $\MT$ of $T^{2n}$ by affine simplices. For every $k$, let $T_k$ denote the set of $k$-dimensional simplices of the triangulation $\MT$. For every map $f:T_0\rightarrow B$, there is a unique extension $F:T^{2n}\rightarrow B$ whose restriction to any simplex of $\MT$ is affine linear. For every simplex $\sigma$ in $\MT$, let $dF_\sigma$ denote the linearization of the restriction of $F$ to the simplex $\sigma$. Note that this linearization is the same for all points in $\sigma$. If $f$ is chosen generically, then $dF_\sigma$ has full rank for all $\sigma$. We fix such a generic choice of $f$.

Consider a top dimensional simplex $\sigma\in T_{2n}$. For every positive integer $k>0$, let $\Lambda_k(\sigma)$ denote the set of points in $\sigma$ which correspond to lattice points in $k^{-1} \BZ^{2n}$ under an affine identification of $\sigma$ with the simplex
\begin{equation}
\label{eq:special_covers_proof_simplex}
\left\{ (x_1,\dots,x_{2n})\in \BR_{\geq 0}^{2n} \mid \sum_j x_j \leq 1 \right\} \subset \BR^{2n}.
\end{equation}
Since $dF_\sigma$ has full rank, we may choose a constant $r>0$ such that, for every $k>0$ and any two distinct points $p\neq q \in \Lambda_k(\sigma)$, we have
\begin{equation}
\label{eq:special_covers_proof_distance_lattice_points}
\operatorname{dist}(p,q) > 10rk^{-1} \quad \text{and} \quad \operatorname{dist}(F(p),F(q)) > 10rk^{-1}.
\end{equation}
Since there are only finitely many simplices, we can choose the constant $r$ uniform among all $\sigma\in T_{2n}$. For every sufficiently large integer $m>0$, we have
\begin{equation}
\label{eq:special_covers_proof_simplex_covering}
\sigma \subset \bigcup\limits_{p\in \Lambda_{mk}(\sigma)} B_{rk^{-1}}(p)
\end{equation}
We fix such an integer $m>0$.

Consider a simplex $\sigma\in T_{2n}$. For every $k>0$, we choose a coloring of $\Lambda_{mk}(\sigma)$ by $m^{2n}$ colors as follows. First, consider a coloring of $(mk)^{-1}\BZ^{2n}$ by $m^{2n}$ distinct colors which is invariant under translations by $k^{-1}\BZ^{2n}$. Then use an affine identification of $\sigma$ with the simplex \eqref{eq:special_covers_proof_simplex} to obtain a coloring of $\Lambda_{mk}(\sigma)$ (see Figure \ref{fig:m2n_coloring}).
\begin{figure}
\includegraphics[width=0.5\textwidth]{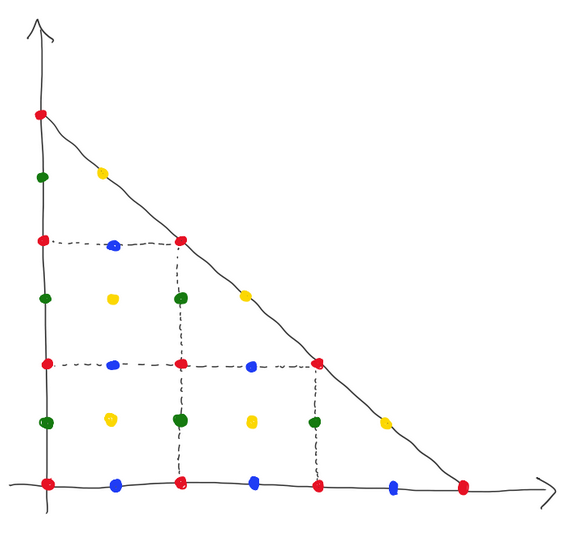}
\caption{A coloring of $\Lambda_{mk}(\sigma)$ by $m^{2n}$ colors in the case $n=1$, $m=2$, $k=3$.}
\label{fig:m2n_coloring}
\end{figure}
For every simplex $\sigma\in T_{2n}$, let us color $\Lambda_{mk}(\sigma)$ by a different set of $m^{2n}$ colors. The total set of colors can be identified with $T_{2n} \times \left\{ 1,\dots,m \right\}^{2n}$.

In the following we construct, for every sufficiently large integer $k>0$, an open cover $\MV^k$ of $T^{2n}$ indexed by the set of colors $T_{2n}\times \left\{ 1,\dots,m \right\}^{2n}$, which is independent of $k$. Moreover, we construct Hamiltonian diffeomorphisms $\varphi_{\mu}^k$ and $\varphi_{\mu\nu}^k$ for all colors $\mu$ and $\nu$. Let $\mu$ be any color and let $\sigma\in T_{2n}$ be the simplex the color belongs to. We define $V_\mu^k$ by
\begin{equation*}
V_\mu^k \coloneqq \bigcup\limits_{\substack{p\in \Lambda_{mk}(\sigma) \\ \text{$p$ has color $\mu$}}} B_{rk^{-1}}(p).
\end{equation*}
By the first inequality in \eqref{eq:special_covers_proof_distance_lattice_points}, this union is disjoint. Moreover, by the inclusion \eqref{eq:special_covers_proof_simplex_covering} the collection of sets $V_\mu^k$ for all colors $\mu$ covers $T^{2n}$. It follows from the second inequality in \eqref{eq:special_covers_proof_distance_lattice_points} that the union of all the balls $B_{rk^{-1}}(F(p))$ where $p$ has color $\mu$ is disjoint as well. We may therefore choose, for every color $\mu$, a Hamiltonian diffeomorphism $\varphi_\mu^k\in \operatorname{Ham}(T^{2n})$ with the following property. Let $p\in \Lambda_{mk}(\sigma)$ be a point of color $\mu$. Then the restriction of $\varphi_\mu^k$ to the ball $B_{rk^{-1}}(p)$ is a translation and maps $p$ to $F(p)$. See Figure \ref{fig:def_phi_mu} for an illustriation.
\begin{figure}
\includegraphics[width=0.7\textwidth]{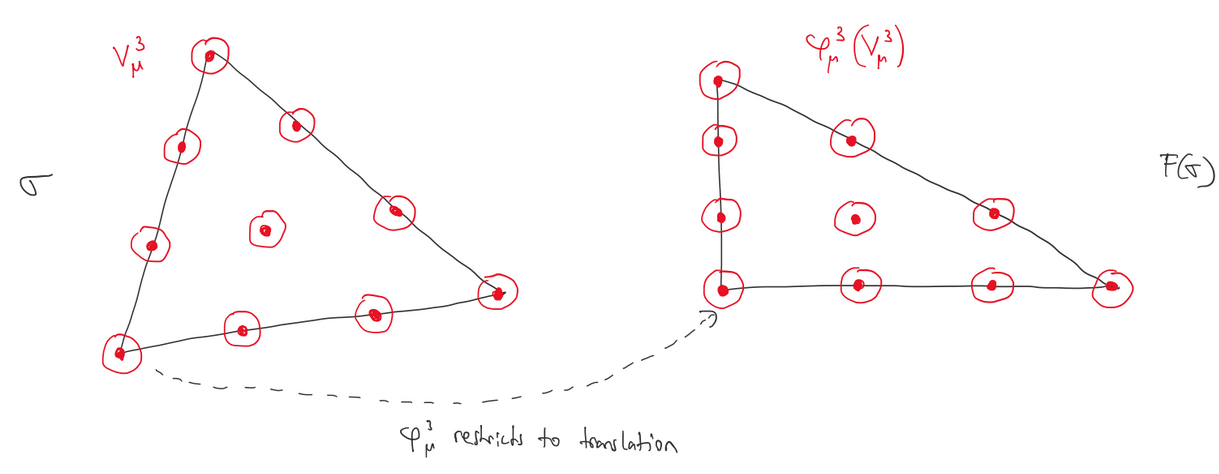}
\caption{An illustration of $\varphi_{\mu}^k$ in the case $n=1$ and $k=3$.}
\label{fig:def_phi_mu}
\end{figure}
If $k$ is sufficiently large, we clearly have $\varphi_\mu^k(V_\mu^k)\subset B$.

Next, consider two colors $\mu$ and $\nu$. If $\mu=\nu$ or $V_\mu$ and $V_\nu$ are disjoint, we simply set $\varphi_{\mu\nu}^k\coloneqq \operatorname{id}$. So assume that $\mu$ and $\nu$ are distinct and that $V_\mu^k\cap V_\nu^k\neq \emptyset$. We need to construct a Hamiltonian diffeomorphism $\varphi_{\mu\nu}^k$ satisfying property \eqref{item:special_covers_transition_diffeomorphisms}. Let us first treat the case that the colors $\mu$ and $\nu$ belong to the same simplex in $\sigma \in T_{2n}$. Consider balls $B_{rk^{-1}}(p)$ belonging to $V_{\mu}^k$ and $B_{rk^{-1}}(q)$ belonging to $V_{\nu}^k$ with non-empty intersection. We observe (see Figure \ref{fig:transition_translation}) that the restrictions of $\varphi_\mu^k$ and $\varphi_\nu^k$ to the intersection $B_{rk^{-1}}(p)\cap B_{rk^{-1}}(q)$ differ by a translation by
\begin{equation*}
F(p) - F(q) - (p-q) = (dF_\sigma-\operatorname{id})(p-q).
\end{equation*}
\begin{figure}
\includegraphics[width=0.7\textwidth]{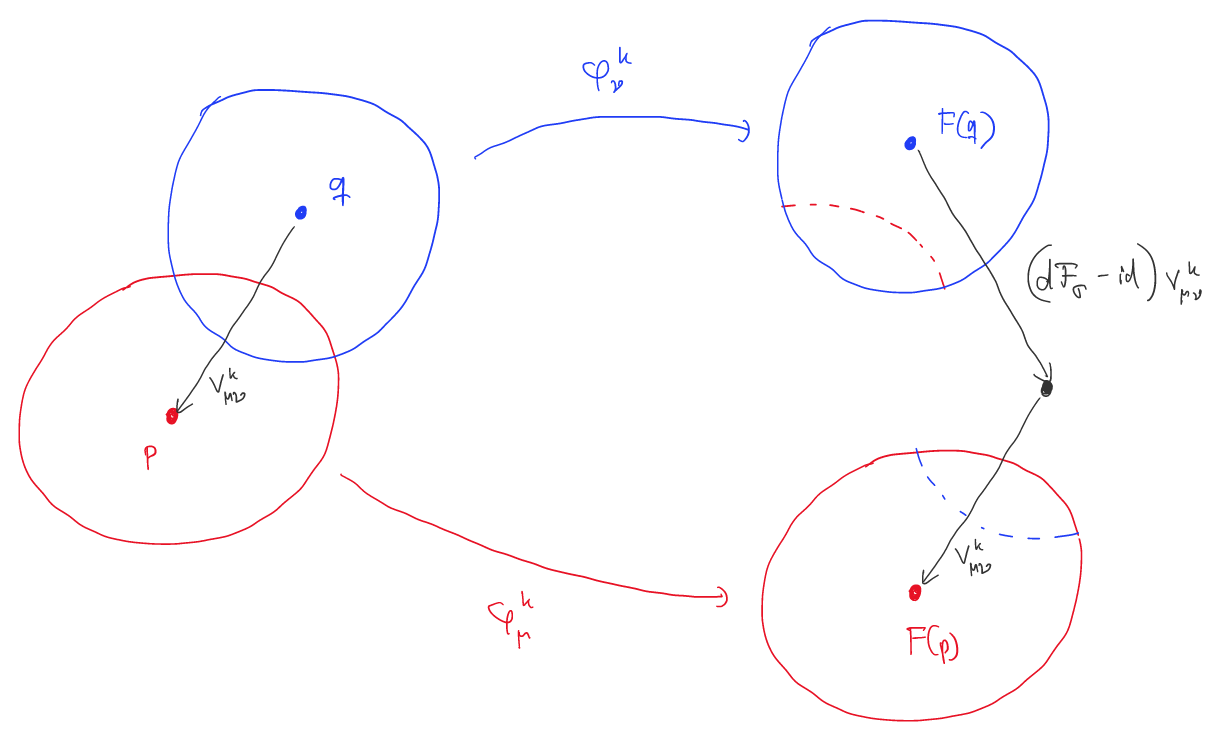}
\caption{Difference between $\varphi_{\mu}^k$ and $\varphi_\nu^k$.}
\label{fig:transition_translation}
\end{figure}
In other words, a diffeomorphism $\varphi_{\mu\nu}^k$ satisfies property \eqref{item:special_covers_transition_diffeomorphisms} exactly if its restriction to the set $\varphi_{\nu}^k(B_{rk^{-1}}(p)\cap B_{rk^{-1}}(q))$ is a translation by $(dF_\sigma-\operatorname{id})(p-q)$. It follows from our construction of the coloring that, for any other choice of balls $B_{rk^{-1}}(p')$ belonging to $V_\mu^k$ and $B_{rk^{-1}}(q')$ belonging to $V_\nu^k$ with non-empty intersection, we have $p'-q' = p-q$ (see Figure \ref{fig:vector_v_mu_nu}).
\begin{figure}
\includegraphics[width=0.4\textwidth]{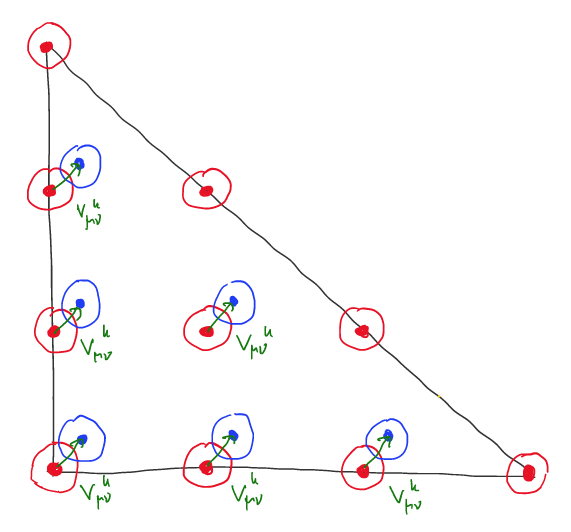}
\caption{The vector $v_{\mu\nu}^k$.}
\label{fig:vector_v_mu_nu}
\end{figure}
Let us denote this common vector by $v_{\mu\nu}^k$. Note that $|v_{\mu\nu}^k| = O(k^{-1})$. Let us choose a Hamiltonian diffeomorphism $\varphi_{\mu\nu}^k$ which is compactly supported in $B$ and whose restriction to the $rk^{-1}$-neighborhood of the image if $F$ is a translation by $(dF_{\sigma}-\operatorname{id})v_{\mu\nu}^k$. This is possible as long as $k$ is sufficiently large. Clearly, we can choose $\varphi_{\mu\nu}^k$ such that the $C^\infty$ distance between $\varphi_{\mu\nu}^k$ and $\operatorname{id}$ goes to zero as $k$ tends to infinity. 

Next, consider the case that $\mu$ and $\nu$ belong to two different simplices in $\sigma$ and $\tau$ in $T_{2n}$. Again, it follows from our construction of the coloring that there exist vectors $v_{\mu\nu}^k$ and $w_{\mu\nu}^k$ of lengths $O(k^{-1})$ such that, for any balls $B_{rk^{-1}}(p)$ belonging to $V_\mu^k$ and $B_{rk^{-1}}(q)$ belonging to $V_\nu^k$ with non-empty intersection, we have (see Figure \ref{fig:vectors_v_w_mu_nu})
\begin{equation*}
o\coloneqq q+w_{\mu\nu}^k \in \sigma\cap\tau \qquad \text{and} \qquad o + v_{\mu\nu}^k = p.
\end{equation*}
\begin{figure}
\includegraphics[width=0.4\textwidth]{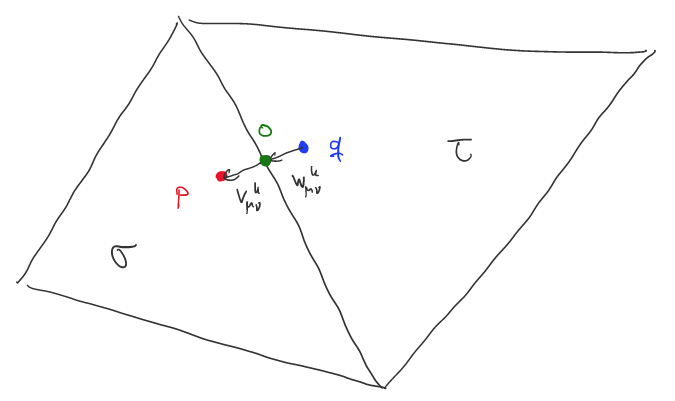}
\caption{The vectors $v_{\mu\nu}^k$ and $w_{\mu\nu}^k$.}
\label{fig:vectors_v_w_mu_nu}
\end{figure}
A diffeomorphism $\varphi_{\mu\nu}^k$ satisfies property \eqref{item:special_covers_transition_diffeomorphisms} if and only if its restriction to $\varphi_{\nu}^k(B_{rk^{-1}}(p) \cap B_{rk^{-1}}(q))$ is a tranlation by
\begin{equation*}
F(p) - F(q) - (p-q) = F(p)-F(o) + (F(o)-F(q)) - (p-o) - (o-q) = (dF_{\sigma}-\operatorname{id}) v_{\mu\nu}^k + (dF_\tau - \operatorname{id}) w_{\mu\nu}^k,
\end{equation*}
a quantity independent of $p$ and $q$. Choose a Hamiltonian diffeomorphism $\varphi_{\mu\nu}^k$ which is compactly supported in $B$ and whose restriction to the $rk^{-1}$-neighborhood of the image of $F$ is a translation by $(dF_{\sigma}-\operatorname{id}) v_{\mu\nu}^k + (dF_\tau - \operatorname{id}) w_{\mu\nu}^k$. Again, we can arrange $\varphi_{\mu\nu}^k$ to be arbitrarily $C^\infty$ close to $\operatorname{id}$ as $k$ goes to infinity.

It is immediate from the above discussion that, for every sufficiently large $k>0$, the cover $(V_\mu^k)_\mu$ and the Hamiltonian diffeomorphisms $\varphi_\mu^k$ and $\varphi_{\mu\nu}^k$ satisfy properties \eqref{item:special_covers_bound_V}, \eqref{item:special_covers_inclusion} and \eqref{item:special_covers_transition_diffeomorphisms}. Moreover, given $\varepsilon>0$, the last property \eqref{item:special_covers_Cinfty_control} holds for all sufficiently large $k$.

It remains to explain how to modify the above construction in order to also guarantee property \eqref{item:special_covers_Bprime}. Choose the triangulation $\MT$ in such a way that there exists a union of top dimensional simplices of $\MT$ which are contained in $B$ and cover some neighbourhood of $\overline{B}'$. Moreover, choose the map $f:T_{0}\rightarrow B$ to be the identity on the vertices of these simplices. Now add $B'$ to the open cover constructed above and set $\varphi_{B'}\coloneqq\operatorname{id}$. We need to construct diffeomorphisms $\varphi_{B',V_\mu^k}$ for every color $\mu$. The interesting case is when $B'$ and $V_\mu^k$ have non-empty intersection. It follows from the choice of $\MT$ and $f$ that in this case we have $\varphi_\mu^k=\operatorname{id}$ for all $k$ sufficiently large. But this allows us to simply set $\varphi_{B',V_{\mu}^k} \coloneqq \operatorname{id}$.
\end{proof}

We also need a version of Proposition~\ref{prop:special_covers} for manifolds with boundary. This is the content of the following result.

Let us fix $n>0$ and two balls $B'\Subset B\Subset \BR^{2n}$. Let $B'_+$ and $B_+$ be the corresponding half balls. Let $\BT\coloneqq \BR/\BZ$ denote the circle. We regard $B'_+$ and $B_+$ as subsets of $[0,\infty)\times \BT\times \BR^{2n-2}$. Here we equip $[0,\infty) \times \BT \times \BR^{2n-2}$ with the symplectic form whose lift to $[0,\infty) \times \BR \times \BR^{2n-2}$ is the restriction of the standard symplectic form on $\BR^{2n}$.

Let $C'$ and $C$ denote the images of $B_+'$ and $B_+$ under the natural projection to $\BR^{2n-2}$, respectively. Clearly, $C'\Subset C$ are balls centered at the origin. Fix a ball $C'\Subset D\Subset C$. Given a real number $r>0$, we abbreviate $A_r\coloneqq [0,r)\times \BT$. We view the product $A_r\times C$ as a symplectic manifold with boundary given by $\left\{ 0 \right\}\times \BT \times C$. We fix numbers $r<R$ both larger than the radius of $B$.

\begin{proposition}
\label{prop:special_open_covers_boundary}
For every $\varepsilon>0$, there exist:
\begin{itemize}
\item a finite collection $\MV$ of open subsets of $A_R\times C$ covering $\overline{A_{r}\times D}$;
\item a compactly supported Hamiltonian diffeomorphism $\varphi_V\in \operatorname{Ham}_c(A_R\times C)$ for every $V\in \MV$;
\item a compactly supported Hamiltonian diffeomorphism $\varphi_{V,W}\in \operatorname{Ham}_c(B_+)$ for every pair $V,W\in \MV$
\end{itemize}
satisfying the following properties:
\begin{enumerate}
\item \label{item:special_covers_boundary_bound_V} $|\MV|$ is bounded from above by a constant independent of $\varepsilon$;
\item \label{item:special_covers_boundary_Bprime} $B_+'\in \MV$ and $\varphi_{B_+'}=\operatorname{id}$;
\item \label{item:special_covers_boundary_inclusion} for all $V\in \MV$, we have $\varphi_V(V)\subset B_+$;
\item \label{item:special_covers_boundary_transition_diffeomorphisms} for all $V,W\in \MV$, we have $\varphi_{V,W}\circ \varphi_V|_{V\cap W} = \varphi_W|_{V\cap W}$;
\item \label{item:special_covers_boundary_Cinfty_control} for all $V,W\in \MV$, we have $\operatorname{dist}_{C^\infty}(\varphi_{V,W},\operatorname{id})\leq \varepsilon$.
\end{enumerate}
\end{proposition}

\begin{proof}
We slightly adapt the proof of Proposition \ref{prop:special_covers}. Choose an affine triangulation $\MT$ of some neighbourhood $N\subset A_R\times C$ of $\overline{A_r\times D}$. Choose a map $f:T_0\rightarrow B_+$ and extend it to a piecewise affine map $F:N\rightarrow B_+$. If $f$ is chosen generically, then $dF_\sigma$ has full rank for all simplices in $\MT$. We will assume in the following that this is the case. As in the proof of Proposition \ref{prop:special_covers}, in order to guarantee property \eqref{item:special_covers_boundary_Bprime}, we arrange $\MT$ such that there exists some collection of simplices of $\MT$ which are contained in $B_+$ and whose union covers a neighbourhood of $\overline{B}_+'$. Moreover, we set $f$ to be equal to the identity on the vertices of all these simplices. In order to define open sets $V_\mu^k$ and differmorphisms $\varphi_\mu^k$ and $\varphi_{\mu\nu}^k$ as in the proof of Proposition \ref{prop:special_covers}, a little more care is necessary for the top dimensional simplices of $\MT$ that intersect the boundary $\left\{ 0 \right\}\times \BT \times C$. Consider a ball $E\subset A_R\times C$ centered at a point $p = (s,t,x)$ close to $\left\{ 0 \right\}\times \BT\times C$ with radius bigger than $s$ (this means that $E$ is a truncated ball). Suppose that $\varphi\in \operatorname{Ham}_c(A_R\times C)$ is a Hamiltonian diffeomorphism whose restriction to $E$ is a translation. Then this translation must actually be in the direction of $\partial_t$, where $t$ is the coordinate of the $\BT$ factor. Let us therefore assume that there is some small neighbourhood of $\left\{ 0 \right\}\times \BT\times C$ such that for every vertex $p=(s,t,x)$ of one of the simplices of $\MT$ intersecting this neighbourhood, the map $f$ is of the form $f(s,t,x) = (s,*,x)$. With these choices of $\MT$ and $f$ in place, the rest of the proof of Proposition \ref{prop:special_covers} carries over almost verbatim.
\end{proof}

\section{Proof of Theorem~\ref{thm:smooth_perfectness_pairs_of_balls} for pairs of full balls}
\label{sec:smooth_perfectness_full_balls}

The goal of this section is to prove Theorem~\ref{thm:smooth_perfectness_pairs_of_balls} in the case that $(V,U)$ is equal to a tuple of full balls $(B,B')$. In view of the discussion in Section~\ref{sec:reduction_to_balls}, this will conclude the proofs of Theorems~\ref{thm:smooth_perfectness_closed} and~\ref{thm:smooth_perfectness_open} and Corollaries~\ref{cor:homogeneous_quasimorphisms_continuity} and~\ref{cor:smooth_perfectness_open}; see Remark~\ref{rem:manifold_without_boundary_only_require_pairs_of_full_balls}. The case of half balls $(B_+,B_+')$ necessary for the more general Theorem~\ref{thm:smooth_perfectness_general} and Corollary~\ref{cor:smooth_perfectness_general} will be treated in Section~\ref{sec:smooth_perfectness_half_balls} below.

Let us fix $n>0$ and two balls $B'\Subset B\Subset \BR^{2n}$ centered at the origin. We may regard $B'$ and $B$ as subsets of the torus $T^{2n}$. Let $\varepsilon>0$ be arbitrarily small and apply Proposition~\ref{prop:special_covers}. This yields a finite open cover $\MV$ of $T^{2n}$. Moreover, we obtain diffeomorphisms $\varphi_V\in \operatorname{Ham}(T^{2n})$ and $\varphi_{V,W}\in \operatorname{Ham}_c(B)$ for all $V,W\in \MV$. Let $\MU$ be the open cover obtained by applying Lemma~\ref{lem:special_cover_refinement} to the open cover $\MV$. We apply Proposition~\ref{prop:cocycles_torus_fragmented} to this open cover $\MU$. This yields a triangulated surface $(\Sigma,\MT)$ and an assignment of a cocycle $c$ on $\MT$ to every Hamiltonian diffeomorphism $\varphi\in \operatorname{Ham}(T^{2n})$ sufficiently close to the identity.

Let $\varphi\in \operatorname{Ham}_c^0(B')$ be close to the identity. Proposition~\ref{prop:cocycles_torus_fragmented} gives us an associated cocycle $c$ on $\MT$. Our next step is to construct a new triangulated surface $(\Sigma',\MT')$ and a cocycle $c'$ on $\MT'$ which takes values in $\operatorname{Ham}_c(B)$ and evaluates to $\varphi$ on the boundary of $\Sigma'$. It is immediate from property~\eqref{item:cocycles_torus_fragmented_support} in Proposition~\ref{prop:cocycles_torus_fragmented} and property~\eqref{item:special_cover_refinement_inclusion} in Lemma~\ref{lem:special_cover_refinement} that, for every $2$-simplex $a$ of $\MT$, we may fix an open set $V_a\in \MV$ (only depending on $a$) such that the restriction of $c$ to $a$ is supported in $V_a$. We abbreviate $\varphi_a\coloneqq \varphi_{V_a}$. Moreover, for any two $2$-simplices $a$ and $b$, we abbreviate $\varphi_{ab}\coloneqq \varphi_{V_a,V_b}$.

Note that every $1$-simplex of $\MT$ is the face of one or two distinct $2$-simplices. In the former case, it is either contained in the boundary of $\Sigma$ or it is an interior $1$-simplex which is the face of one single $2$-simplex in two different ways. Let us cut the surface $\Sigma$ open along every interior $1$-simplex $e$ which the the face of two distinct $2$-simplices $a$ and $b$ as indicated in Figure~\ref{fig:cut_at_edges}.
\begin{figure}
\includegraphics[width=0.8\textwidth]{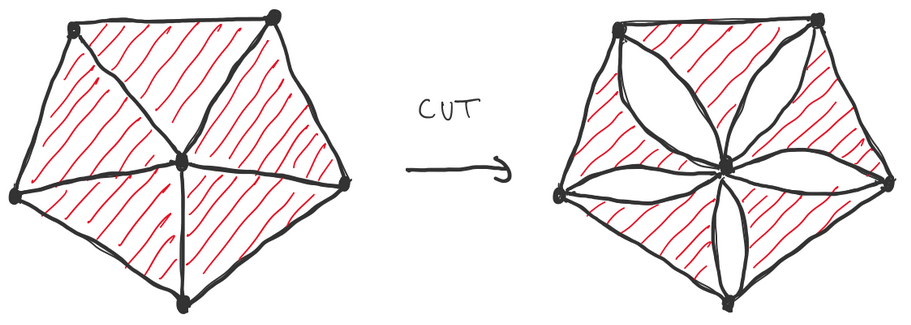}
\caption{We cut along interior $1$-simplices which are the faces of two distinct $2$-simplices.}
\label{fig:cut_at_edges}
\end{figure}
After this process, all $2$-simplices are still held together at their vertices. Let us define a cocycle on the cut surface by declaring its restriction to a $2$-simplex $a$ to be $\varphi_a c|_a \varphi_a^{-1}$. Now consider one of the $1$-simplices $e$ of $\MT$ that was cut. Let $a$ and $b$ denote the two $2$-simplices containing $e$ as a face. Note that since $c(e)$ is supported in the intersection $V_a\cap V_b$, we have
\begin{equation*}
\varphi_a c(e) \varphi_a^{-1} = \varphi_{ab} (\varphi_b c(e) \varphi_b^{-1})\varphi_{ab}^{-1}
\end{equation*}
Let us attach two $2$-simplices in the hole between $a$ and $b$ and extend our cocycle as indicated in Figure~\ref{fig:fill_int_hole}.
\begin{figure}
\includegraphics[width=0.5\textwidth]{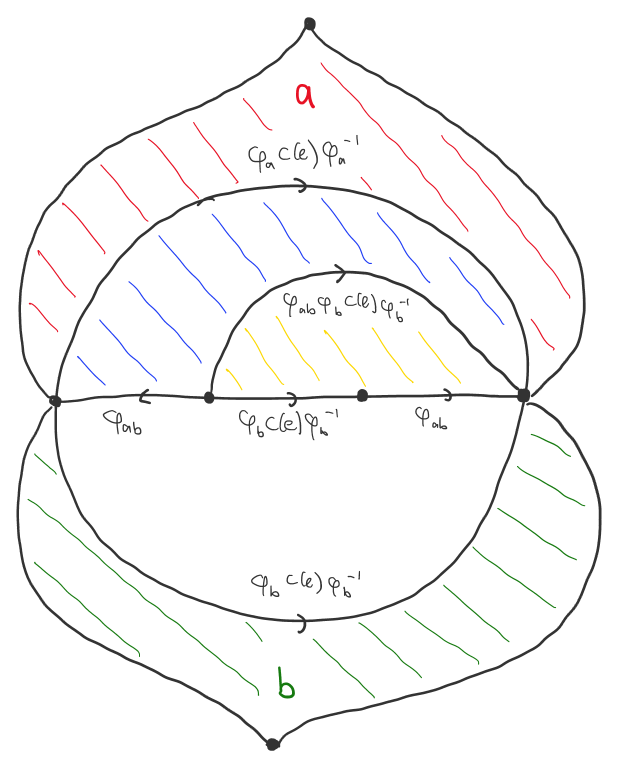}
\caption{We insert two $2$-simplices in the hole obtained by cutting along the $1$-simplex separating $a$ and $b$.}
\label{fig:fill_int_hole}
\end{figure}
The remaining hole is bounded by four $1$-simplices. Let us identify pairs of opposite $1$-simplices via orientation reversing affine maps. It is clear from Figure~\ref{fig:fill_int_hole} that the cocycle descends to the quotient.

Next, let us consider a boundary $1$-simplex $e$ of $\MT$. Let $a$ be the $2$-simplex containing $e$ as a face. By property~\eqref{item:cocycles_torus_fragmented_support_zero_cal} in Proposition~\ref{prop:cocycles_torus_fragmented}, the Hamiltonian diffeomorphism $c(e)$ is supported in $B'\cap V_a$. Unless $c(e)$ is equal to the identity for all $\varphi$, in which case we simply skip $e$ and move on to the next boundary $1$-simplex, this implies that $B'\cap V_a\neq \emptyset$. Let us abbreviate $\varphi_e\coloneqq \varphi_{B',V_a}$. Then we have
\begin{equation*}
c(e) = \varphi_e(\varphi_ac(e)\varphi_a^{-1})\varphi_e^{-1}.
\end{equation*}
Let us attach two $2$-simplices and extend the cocycle as indicated in Figure~\ref{fig:boundary_attachment}.
\begin{figure}
\includegraphics[width=0.5\textwidth]{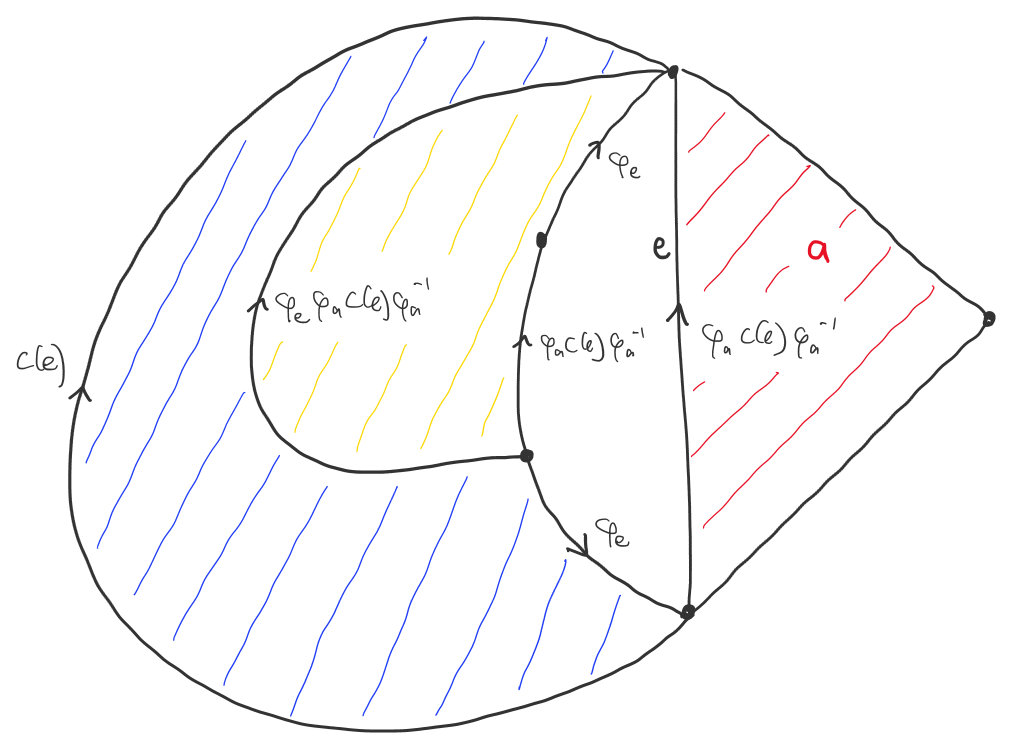}
\caption{We attach two $2$-simplices at the endpoints of $e$.}
\label{fig:boundary_attachment}
\end{figure}
The resulting $\Delta$-complex has a hole bounded by $e$ and three additional $1$-simplices. Again we identify opposite pairs of $1$-simplices via orientation reversing affine maps. It is clear from Figure~\ref{fig:boundary_attachment} that the cocycle descends to the quotient. We repeat this process for all boundary $1$-simplices of $\MT$.

The result of the above cut and paste constructions is a new triangulated surface $(\Sigma',\MT')$ equipped with a cocycle $c'$. It is immediate from our construction that $c'$ is supported in $B$. Moreover, $\Sigma'$ has exactly one boundary component and $c'$ evaluates to $\varphi$ on this boundary component. It is clear from property~\eqref{item:cocycles_torus_fragmented_smooth_dependence} in Propsition~\ref{prop:cocycles_torus_fragmented} and the construction that $c'$ depends smoothly on $\varphi$. 

The number of simplices in $\MT'$ is bounded by a constant only depending on the number of simplices in $\MT$. Recall from Proposition~\ref{prop:cocycles_torus_fragmented} that the number of simplices of $\MT$ is bounded by a constant only depending on $|\MU|$ and $n$. Since $|\MU|$ can be bounded by a constant only depending on $|\MV|$ by property~\eqref{item:special_cover_refinement_cardinality} in Lemma~\ref{lem:special_cover_refinement} and $|\MV|$ can be bounded by a constant only depending on $n$ and pair of balls $(B,B')$ by property~\eqref{item:special_covers_bound_V} in Proposition~\ref{prop:special_covers}, we obtain a bound on the number of simplices in $\MT'$ by a constant only depending on $n$ and $(B,B')$.

Note that by choosing $\varepsilon$ sufficiently small, we can make the $C^\infty$ distance of the diffeomorphisms $\varphi_{ab}$ and $\varphi_e$ to the identity arbitrarily small. This process might require to modify $(\Sigma',\MT')$. The important point, however, is that the number of simplices in $\MT'$ stays bounded by a constant only depending on $n$. We can therefore arrange the image of the identity under the map $\varphi\mapsto c'$ to be arbitrarily close to the identity cocycle while keeping the complexity of $(\Sigma',\MT')$ bounded.

Our next task is to argue that the cocycle $c'$ can be replaced by a cocycle $c''$ on $\MT'$ which is not only supported in $B$, but actually takes values in $\operatorname{Ham}_c^0(B)$. In order to see this, let us choose an intermediate ball $B'\Subset B''\Subset B$. Applying the above discussion to the pair of balls $B'\Subset B''$ instead of the pair $B'\Subset B$, we see that we may assume that $c'$ is actually supported in $B''$. Let us fix an autonomous Hamiltonian $F$ supported in $B\setminus \overline{B}''$ such that $\int_B F =1$. Note that the map
\begin{equation*}
P:\operatorname{Ham}_c(B'') \rightarrow \operatorname{Ham}_c^0(B) \qquad \psi \mapsto \psi \circ \varphi_{-\operatorname{Cal}(\psi)F}^1
\end{equation*}
is a group homomorphism. The reason is that diffeomorphisms $\psi\in \operatorname{Ham}_c(B'')$ and diffeomorphisms of the form $\varphi_{t F}^1$ for $t\in \BR$ commute because they have disjoint supports. Moreover, $P$ restricts to the inclusion on the set $\operatorname{Ham}_c^0(B'')$. Let us define the value of the cocycle $c''$ on an oriented $1$-simplex $e$ of $\MT'$ by the formula
\begin{equation*}
c''(e) \coloneqq P(c'(e)).
\end{equation*}
It is immediate from the properties of $P$ that this takes values in $\operatorname{Ham}_c^0(B)$, is a cocycle and evaluates to $\varphi$ on the boundary of $\Sigma$. Moreover, the assignment $\varphi\mapsto c'$ is smooth and can be arranged to map the identity arbitrarily close to the identity cocycle.

Let $m$ denote the genus of $\Sigma'$. Note that $m$ is bounded from above by a constant only depending on $n$ and $(B,B')$. Let $\gamma$ be the loop which goes around the boundary of $\Sigma'$ once in positive direction. We can write $\gamma$ as a product of $m$ commutators $[\alpha_j,\beta_j]$ in $\pi_1(\Sigma')$. Define $u_j\coloneqq c''(\alpha_j)$ and $v_j\coloneqq c''(\beta_j)$. It is clear from the above discussion that $\varphi\mapsto (u_1,v_1,\dots,u_m,v_m)$ defines a right inverse satisfying all desired properties. \qed

\section{A variant of Herman--Sergeraert's theorem with boundary}
\label{sec:herman_sergeraert_boundary}

Our proof of Theorem~\ref{thm:smooth_perfectness_pairs_of_balls} for pairs of full balls, which is carried out in Section~\ref{sec:smooth_perfectness_full_balls} above, ultimately relies on Herman--Sergeraert's theorem (Theorem~\ref{thm:herman}). Before turning to the proof of Theorem~\ref{thm:smooth_perfectness_pairs_of_balls} for pairs of half balls in Section~\ref{sec:smooth_perfectness_half_balls}, we need to establish a variant of Theorem~\ref{thm:herman} adapted to manifolds with boundary. This is the content of the present section.

Fix an integer $n>0$. Let $B\subset \BR^{2n-2}$ be an open ball. Let $\BT\coloneqq \BR/\BZ$ denote the circle. For every $r>0$, define the half open annulus $A_r\coloneqq [0,r)\times \BT$. Moreover, define $U_r\coloneqq A_r\times B$. We view $U_r$ as a symplectic manifold with boundary given by $\left\{ 0 \right\}\times \BT \times B$. If $r< +\infty$, then $U_r$ is a relatively compact open subset of the symplectic manifold with boundary $M\coloneqq A_\infty \times \BR^{2n-2}$. Let us define the space of symplectic embeddings
\begin{equation*}
\operatorname{Emb}_c(U_r) \coloneqq \left\{ \psi:U_r\rightarrow M \mid \exists R\geq r \exists \varphi \in \operatorname{Ham}_c(U_R): \enspace \psi = \varphi|_{U_r} \right\}.
\end{equation*}

For every $\alpha\in \BT$, let $r_\alpha:\BT\rightarrow \BT$ be the rotation defined by $r_\alpha(x)\coloneqq x+\alpha$. Fix a smooth function $\eta:\BT\rightarrow \BR$ such that
\begin{enumerate}
\item $\eta(x) = -x$ for all $x\in [-1/8,1/8]$;
\item $\operatorname{supp}\eta \subset (-1/4,1/4)$.
\end{enumerate}
Given $\lambda\in \BR$, we define the circle map $c_\lambda:\BT\rightarrow \BT$ by setting $c_\lambda(x) \coloneqq x + \lambda\eta(x)$. If $\lambda$ is sufficiently close to $0$, then $c_\lambda$ is a diffeomorphism. In the following, we will always assume that this is the case.

The circle diffeomorphisms $r_\alpha$ and $c_\lambda$ both lift to symplectomorphisms of the cotangent bundle $T^*\BT\cong\BR\times \BT$. Let $T^*r_\alpha$ and $T^*c_\lambda$ denote these lifts. Observe that $T^*r_\alpha$ is simply given by $T^*r_\alpha(s,t) = (s,t+\alpha)$. The diffeomorphism $T^*c_\lambda$ is supported in $\BR\times (-1/4,1/4)$. On the set $\BR\times (-1/8,1/8)$, it is given by
\begin{equation*}
T^*c_\lambda(s,t) = ((1-\lambda)^{-1}s,(1-\lambda)t).
\end{equation*}
Both lifts preserve the zero section $\left\{ 0 \right\}\times \BT$ and leave the half infinite cylinder $A_\infty$ invariant. We can therefore define symplectomorphisms of $M$ by setting
\begin{equation*}
R_\alpha \coloneqq (T^*r_\alpha)|_{A_\infty} \times \operatorname{id}_{\BR^{2n-2}}
\quad C_\lambda \coloneqq (T^*c_\lambda)|_{A_\infty} \times \operatorname{id}_{\BR^{2n-2}}
\quad R_{\alpha,\lambda} \coloneqq C_\lambda R_\alpha C_\lambda^{-1}.
\end{equation*}

Suppose that $\varphi \in \operatorname{Ham}_c(U_r)$. Even though $R_\alpha$ clearly does not restrict to a diffeomorphism in $\operatorname{Ham}_c(U_r)$, it is straighforward to check that $[R_\alpha,\varphi]\in \operatorname{Ham}_c(U_r)$. Similarly, it is easy to see that $[R_{\alpha,\lambda},\varphi]\in \operatorname{Ham}_c(U_R)$ for all sufficiently large $R$. Given $0<r<R$, $\alpha\in \BT$ and $\lambda\in \BR$ sufficiently close to $0$, we may therefore define a map
\begin{equation}
\label{eq:Falphalambda}
\MF_{\alpha,\lambda} : \operatorname{Ham}_c(U_R)^2 \rightarrow \operatorname{Emb}_c(U_r) \qquad \MF_{\alpha,\lambda}(\varphi_1,\varphi_2)
\coloneqq [R_{\alpha,\lambda}, \varphi_1] [R_\alpha, \varphi_2]\Big|_{U_r}.
\end{equation}
We are now in a position to state the main result of this section. Let $B'\Subset B \subset\BR^{2n-2}$ be a smaller open ball and abbreviate $U'_r \coloneqq A_r\times B'$.

\begin{theorem}
\label{thm:herman_with_boundary}
Let $0<r<R$. For every $\lambda>0$ sufficiently close to $0$ and for every $\alpha\in \BT$ satisfying a Diophantine condition, there exist an open neighbourhood $\MN \subset \operatorname{Emb}_c(U'_r)$ of the inclusion $\iota:U'_r\hookrightarrow M$ and a smooth local right inverse $\MN \rightarrow \operatorname{Ham}_c(U_R)^2$ of $\MF_{\alpha,\lambda}$ which maps $\iota$ to $(\operatorname{id},\operatorname{id})$.
\end{theorem}

The remainder of Section~\ref{sec:herman_sergeraert_boundary} is devoted to the proof of Theorem~\ref{thm:herman_with_boundary}.

\subsection{The Nash--Moser implicit function theorem}

Herman--Sergeraert's theorem can be proved using the Nash--Moser implicit function theorem; see e.g.\ \cite[p.\ 42]{ban97}. Our strategy is to show Theorem~\ref{thm:herman_with_boundary} using the Nash--Moser theorem as well. We refer to \cite{ham82} for a thorough introduction to the necessary background on Fr\'{e}chet spaces and the Nash--Moser theorem. Let us recall here the statement of the relevant version of the theorem; see \cite[\S III.1.1, Theorem 1.1.3]{ham82}.

\begin{theorem}
\label{thm:nash_moser}
Let $F$ and $G$ be tame Fr\'echet spaces. Let $U\subset F$ be an open subset and $P: U\subset F \rightarrow G$ a smooth tame map. Suppose that the derivative $dP : U \times F \rightarrow G$ admits a smooth tame family of right inverses $V: U\times G \rightarrow F$, i.e.\ $dP(x) \circ V(x) = \operatorname{id}_G$ for all $x \in U$. Then $P$ has a local smooth tame right inverse near any point.
\end{theorem}

We point out that in this theorem smoothness is understood in the sense of maps between Fr\'echet spaces as defined in \cite{ham82}. In the cases of interest to us, this notion of smoothness is equivalent to smoothness in the diffeological sense.

We would like to apply this theorem to the map $\MF_{\alpha,\lambda}$ defined in equation~\eqref{eq:Falphalambda}. One immediate problem is that neither $\operatorname{Ham}_c(U_R)$ nor $\operatorname{Emb}_c(U_r)$ are naturally Fr\'echet manifolds.

We rectify this as follows. Let $\MH_0(U_R)$ denote the space of all smooth functions defined on the closure of $U_R$ which vanish on the set $\left\{ 0 \right\}\times \BT \times B$ and which vanish to infinite order on the sets $A_R\times \partial B$ and $\left\{ R \right\}\times \BT \times B$. Together with the sequence of $C^k$ norms, this defines a Fr\'echet space. In fact, it is a tame Fr\'echet space; see \cite[\S II.1]{ham82}. This follows from the arguments in the proofs of \cite[\S II.1.3, Corollaries 1.3.7 \& 1.3.8]{ham82}. Now consider a Hamiltonian $H: [0,1]\times U_R \rightarrow \BR$ such that $H_t \in \MH_0(U_R)$ for all $t$. Such a Hamiltonian generates an isotopy $(\varphi_H^t)_{t\in [0,1]}$ of Hamiltonian diffeomorphisms which agree with the identity to infinite order at $A_R\times \partial B$ and $\left\{ R \right\}\times \BT\times B$. We define $\operatorname{Ham}_0(U_R)$ to be the group of all Hamiltonian diffeomorphisms which arise as time-$1$ maps of such Hamiltonians $H$. Note that $\operatorname{Ham}_c(U_R) \subset \operatorname{Ham}_0(U_R)$. A standard argument involving generating functions shows that $\operatorname{Ham}_0(U_R)$ is a Fr\'echet manifold modeled on the the tame Fr\'echet space $\mathcal{H}_0(U_R)$.

Next, we define
\begin{equation*}
\operatorname{Emb}_0(U_r) \coloneqq \left\{ \psi:U_r\rightarrow M \mid \exists R\geq r \exists \varphi \in \operatorname{Ham}_0(U_R): \enspace \psi = \varphi|_{U_r} \right\}.
\end{equation*}
This completes $\operatorname{Emb}_c(U_r)$ to a Fr\'{e}chet manifold modeled on the tame Fr\'{e}chet space $\ME_0(U_r)$ of all smooth functions $f$ on $\overline{U}_r$ which vanish on $\left\{ 0 \right\}\times \BT\times B$ and all of whose derivatives vanish on $A_r\times \partial B$.

We observe that the map $\MF_{\alpha,\lambda}$ defined in equation~\eqref{eq:Falphalambda} extends to a map
\begin{equation}
\label{eq:Falphalambda_frechet}
\MF_{\alpha,\lambda} : \operatorname{Ham}_0(U_R)^{2} \rightarrow \operatorname{Emb}_0(U_r)
\end{equation}
defined by the same formula. We will show below that this map satisfies the hypotheses in the Nash--Moser theorem (Theorem~\ref{thm:nash_moser}). From this, we will deduce Theorem~\ref{thm:herman_with_boundary} in Subsection~\ref{subsec:proof_of_herman_with_boundary}.

\subsection{Computing the derivative}

The goal of this section is to compute the derivative of $\MF_{\alpha,\lambda}$. It will be convenient to introduce trivializations
\begin{equation*}
\Pi:T\operatorname{Ham}_0(U_R) \cong \operatorname{Ham}_0(U_R)\times \MH_0(U_R) \quad \text{and} \quad \Pi:T\operatorname{Emb}_0(U_r) \cong \operatorname{Emb}_0(U_r) \times \ME_0(U_r)
\end{equation*}
of the tangent bundles of $\operatorname{Ham}_0(U_R)$ and $\operatorname{Emb}_0(U_r)$, respectively. Consider a smooth path $\varphi^t\in \operatorname{Ham}_0(U_R)$. Let $H^t\in \MH_0(U_R)$ be the family of Hamiltonians generating $\varphi^t$. We define the tivialization $\Pi$ of $T\operatorname{Ham}_0(U_R)$ by setting $\Pi(\varphi^0,\partial_t|_{t=0}\varphi^t) \coloneqq (\varphi^0, H^0\circ \varphi^0)$. The trivialization of $T\operatorname{Emb}_0(U_r)$ is defined in an analogous way. Note that if $H^t$ is a family of Hamiltonians generating a path $\varphi^t$ in $\operatorname{Emb}_0(U_r)$, then $H^t$ is defined on the image of $\varphi^t$ and hence $H^t\circ \varphi^t$ is defined in $U_r$.

Using these trivializations, we can then view the linearization of $\MF_{\alpha,\lambda}$ at $\varphi\in \operatorname{Ham}_0(U_R)^2$ as a linear map $d\MF_{\alpha,\lambda}(\varphi) : \MH_0(U_R)^2 \rightarrow \ME_0(U_r)$.

\begin{proposition}
\label{prop:derivative_of_F}
For $\varphi=(\varphi_1,\varphi_2)\in \operatorname{Ham}_0(U_R)^2$ and $H=(H_1,H_2) \in \MH_0(U_R)^2$, we have
\begin{equation}
\label{eq:derivative_of_F}
d\MF_{\alpha,\lambda}(\varphi)H = 
\left( [R_\alpha,\varphi_2]^*(\varphi_1)_* ( (R_{\alpha,\lambda})_* - \operatorname{id})H_1
+ (\varphi_2)_* ( (R_\alpha)_* - \operatorname{id}) H_2 \right)\Big|_{U_r}.
\end{equation}
\end{proposition}

In the proof of Proposition \ref{prop:derivative_of_F}, we will make use of the following elementary lemma, whose proof we omit.

\begin{lemma}
\label{lem:composition_of_flows}
Let $\varphi^t$ and $\psi^t$ be paths of diffeomorphisms generated by time-dependent vector fields $X^t$ and $Y^t$, respectively. Moreover, let $\chi$ be a fixed diffeomorphism. Then the following assertions hold:
\begin{enumerate}
\item The composition $\varphi^t\circ \psi^t$ is generated by $X^t + \varphi^t_*Y^t$.
\item The inverse $(\varphi^t)^{-1}$ is generated by $-(\varphi^t)^*X^t$.
\item The path $\chi\varphi^t\chi^{-1}$ is generated by $\chi_* X^t$.
\end{enumerate}
\end{lemma}

\begin{proof}[Proof of Proposition \ref{prop:derivative_of_F}]
For $i\in \left\{ 1,2 \right\}$, let $\varphi_i^t$ be a Hamiltonian isotopy in $\operatorname{Ham}_0(U_R)$ starting at $\varphi_i^0 = \varphi_i$. Let $H_i^t\in \MH_0(U_R)$ be the Hamiltonians generating $\varphi_i^t$. Abbreviate $R_1\coloneqq R_{\alpha,\lambda}$ and $R_2\coloneqq R_\alpha$. A straightforward application of Lemma \ref{lem:composition_of_flows} shows that $[R_i,\varphi_i^t]$ is generated by the Hamiltonian
\begin{equation*}
(R_i)_*H_i^t - [R_i,\varphi_i^t]_*H_i^t.
\end{equation*}
Applying Lemma \ref{lem:composition_of_flows} again, we wee that $\MF_{\alpha,\lambda}(\varphi_1^t,\varphi_2^t)$ is generated by
\begin{equation*}
(R_1)_*H_1^t - [R_1,\varphi_1^t]_*H_1^t + [R_1,\varphi_1^t]_*((R_2)_*H_2^t - [R_2,\varphi_2^t]_*H_2^t).
\end{equation*}
Combining this with our choice of trivializations of $T\operatorname{Ham}_0(U_R)$ and $T\operatorname{Emb}_0(U_r)$, we see that
\begin{equation*}
d\MF_{\alpha,\lambda}(\varphi)H = \MF_{\alpha,\lambda}(\varphi)^*( (R_1)_*(\varphi_1)_*H_1 - [R_1,\varphi_1]_*(\varphi_1)_*H_1 + [R_1,\varphi_1]_*((R_2)_*(\varphi_2)_*H_2 - [R_2,\varphi_2]_*(\varphi_2)_*H_2) ) \Big|_{U_r}.
\end{equation*}
An easy computation shows that this simplifies to equation~\eqref{eq:derivative_of_F}.
\end{proof}

\subsection{Constructing linear right inverses}

Define $\hat{M}\coloneqq [0,\infty)\times \BR^{2n-2}$ and $\hat{U}_r\coloneqq [0,r)\times B$ for $r>0$. Given a smooth function $f\in C^\infty(M)$, let us define the function $\hat{f}\in C^\infty(\hat{M})$ by taking the average over the $\BT$ factor, i.e. by setting
\begin{equation*}
\hat{f}(s,x) \coloneqq \int_{\BT} f(s,t,x) dt \quad \text{for $(s,x)\in \hat{M}$.}
\end{equation*}
Let $\hat{\MH}_0(U_R)\subset \MH_0(U_R)$ be the subspace consisting of all functions $f$ with vanishing average $\hat{f}$.

\begin{proposition}
\label{prop:right_inverse_D_alpha}
Let $R>0$ and suppose that $\alpha\in \BT$ satisfies a Diophantine condition. Then the linear map
\begin{equation*}
\MD_\alpha:\MH_0(U_R) \rightarrow \hat{\MH}_0(U_R) \quad f \mapsto (R_\alpha)_*f - f
\end{equation*}
possesses a tame linear right inverse.
\end{proposition}

\begin{proof}
It is a well-known and fundamental statement in KAM theory that the map
\begin{equation*}
\hat{C}^\infty(\BT) \rightarrow \hat{C}^\infty(\BT) \qquad f \mapsto f(\cdot + \alpha) - f
\end{equation*}
has a tame linear inverse for all Diophantine $\alpha$, where $\hat{C}^\infty(\BT)$ denotes the space of smooth functions on $\BT$ with zero mean. An inverse can be constructed via Fourier analysis; see e.g.\ \cite[p.\ 45-46]{ban97}. Proposition~\ref{prop:right_inverse_D_alpha} admits an analogous proof, which we briefly sketch below.

We may write $f \in \MH_0(U_R)$ as
\begin{equation*}
f(s,t,x) = \sum_{k\in \BZ} f_k(s,x) e^{2\pi i k t}
\end{equation*}
for functions $f_k$ on $\hat{U}_R$ which vanish on $\left\{ 0 \right\}\times B$ and which vanish to infinite order at $([0,R]\times \partial B) \cup (\left\{ R \right\}\times B)$. Note that $f \in \hat{\MH}_0(U_R)$ if and only if $f_0$ identically vanishes. In terms of Fourier expansions, the operator $\MD_\alpha$ can be written as
\begin{equation*}
\sum_{k\in \BZ} f_k(s,x) e^{2\pi i k t} \qquad \mapsto \qquad \sum_{k\in \BZ} (e^{2\pi i k \alpha}-1) f_k(s,t) e^{2\pi i k t}.
\end{equation*}
One therefore obtains a right inverse of $\MD_\alpha$ via the formula
\begin{equation*}
\sum_{k\in \BZ\setminus \left\{ 0 \right\}} f_k(s,x) e^{2\pi i k t} \qquad \mapsto \qquad \sum_{k\in \BZ\setminus\left\{ 0 \right\}} (e^{2\pi i k \alpha}-1)^{-1} f_k(s,t) e^{2\pi i k t}.
\end{equation*}
Of course one needs to check that this actually defines a tame linear map $\hat{\MH}_0(U_R) \rightarrow \MH_0(U_R)$. This follows from the Diophantine condition satisfied by $\alpha$, which implies that $(e^{2 \pi i k \alpha} -1)^{-1}$ is bounded polynomially in $k$; see \cite[p.\ 46]{ban97}.
\end{proof}

\begin{proposition}
\label{prop:right_inverse_B_lambda}
Let $0<r<R$. If $\lambda>0$ is sufficiently small, then the family of linear maps
\begin{equation*}
\MB_\lambda: \operatorname{Ham}_0(U_R)^2 \times \hat{\MH}_0(U_R)^2 \rightarrow \ME_0(U_r) \qquad \MB_\lambda(\varphi)H = \left( (\varphi_1 C_\lambda)_*H_1 + (\varphi_2)_*H_2 \right) \big|_{U_r}.
\end{equation*}
possesses a smooth tame family of rigth inverses in some open neighbourhood of $\varphi=(\operatorname{id},\operatorname{id})$.
\end{proposition}

\begin{proof}
Let
\begin{equation*}
\operatorname{avg}: C^\infty(M) \rightarrow C^\infty(\hat{M}) \quad \operatorname{avg}f \coloneqq \hat{f}.
\end{equation*}
denote the averaging operator. Let us fix a non-negative smooth function $\rho:\BT\rightarrow \BR$ with integral $\int_{\BT} \rho(t) dt = 1$ and with support contained in $(-1/8,1/8)\subset \BT$. For every $t_0\in \BT$, we abbreviate $\rho_{t_0}\coloneqq \rho(\cdot - t_0)$ and define linear maps
\begin{equation*}
\operatorname{lift}_{t_0} : C^\infty(\hat{M}) \rightarrow C^\infty(M) \quad (\operatorname{lift}_{t_0}f)(s,t,x) \coloneqq f(s,x)\rho_{t_0}(t)
\end{equation*}
and
\begin{equation*}
\operatorname{pr}_{t_0}: C^\infty(M) \rightarrow C^\infty(M) \quad \operatorname{pr}_{t_0}f \coloneqq f - \operatorname{lift}_{t_0}\operatorname{avg}f.
\end{equation*}
Clearly, $\operatorname{lift}_{t_0}$ is a right inverse of $\operatorname{avg}$ and $\operatorname{pr}_{t_0}$ is a projection onto the kernel of $\operatorname{avg}$.

For every $r>0$, let $\MH_0(\hat{U}_r)$ denote the Fr\'{e}chet space of all smooth functions on the closure of $\hat{U}_r$ which vanish on $\left\{ 0 \right\}\times \overline{B}$ and all of whose derivatives vanish on $\left\{ r \right\}\times \overline{B}$ and on $[0,r]\times \partial B$. Similarly, we let $\ME_0(\hat{U}_r)$ denote the Fr\'{e}chet space of all smooth functions on the closure of $\hat{U}_r$ which vanish in $\left\{ 0 \right\}\times \overline{B}$ and all of whose derivatives vanish on $[0,r]\times \partial B$. Note that the averaging operator $\operatorname{avg}$ induces linear maps
\begin{equation*}
\operatorname{avg}:\MH_0(U_r)\rightarrow \MH_0(\hat{U}_r) \quad \text{and}\quad \operatorname{avg}:\ME_0(U_r)\rightarrow \ME_0(\hat{U}_r).
\end{equation*}
The lifting operator $\operatorname{lift}_{t_0}$ induces linear maps
\begin{equation*}
\operatorname{lift}_{t_0}:\MH_0(\hat{U}_r) \rightarrow \MH_0(U_r) \quad \text{and} \quad \operatorname{lift}_{t_0}:\ME_0(\hat{U}_r)\rightarrow \ME_0(U_r).
\end{equation*}
The projection operator $\operatorname{pr}_{t_0}$ yields
\begin{equation*}
\operatorname{pr}_{t_0}:\MH_0(U_r)\rightarrow \hat{\MH}_0(U_r) \quad \text{and}\quad \operatorname{pr}_{t_0}:\ME_0(U_r)\rightarrow \hat{\ME}_0(U_r).
\end{equation*}

Fix $0<r<r'<R$. Define a family of linear maps $\MP:\operatorname{Ham}_0(U_R)\times \hat{\MH}_0(U_R)\rightarrow \ME_0(\hat{U}_{r'})$ by
\begin{equation*}
\MP(\varphi)H \coloneqq (\operatorname{avg}(\varphi C_\lambda)_*H)\Big|_{\hat{U}_{r'}}.
\end{equation*}

\begin{lemma}
\label{lem:right_inverse_P}
The restriction of $\MP$ to some open neighbourhood $\MU\subset \operatorname{Ham}_0(U_R)$ of the identity possesses a smooth tame family of right inverses $\MQ:\MU\times \ME_0(\hat{U}_{r'})\rightarrow \hat{\MH}_0(U_R)$.
\end{lemma}

Before turning to the proof of Lemma \ref{lem:right_inverse_P}, let us first finish the proof of Proposition \ref{prop:right_inverse_B_lambda} assuming the statement of Lemma \ref{lem:right_inverse_P}. Suppose that $\varphi=(\varphi_1,\varphi_2)$ is sufficiently close to $(\operatorname{id},\operatorname{id})$ and let $H\in \ME_0(U_r)$. Our goal is to construct $(H_1,H_2)\in \hat{\MH}_0(U_R)^2$ such that
\begin{equation}
\label{eq:right_inverse_B_lambda_proof_a}
\MB_\lambda(\varphi)(H_1,H_2) = ( (\varphi_1 C_\lambda)_*H_1 + (\varphi_2)_*H_2)\Big|_{U_r} = H.
\end{equation}
Fix a tame linear extension operator $\operatorname{ext}^{U_R}_{U_r}:\ME_0(U_r)\rightarrow \MH_0(U_R)$. If $\varphi_2$ is sufficiently close to $\operatorname{id}$, then equation \eqref{eq:right_inverse_B_lambda_proof_a} is implied by
\begin{equation}
\label{eq:right_inverse_B_lambda_proof_b}
( (\varphi_2^{-1}\varphi_1 C_\lambda)_*H_1 + H_2)\Big|_{U_{r'}} = (\varphi_2^*\operatorname{ext}^{U_R}_{U_r}H)\Big|_{U_{r'}}.
\end{equation}
Note that since $\hat{H}_2 = 0$, we need to find $H_1$ such that
\begin{equation}
\label{eq:right_inverse_B_lambda_proof_requirement_H1}
(\operatorname{avg}(\varphi_2^{-1}\varphi_1C_\lambda)_*H_1)\Big|_{\hat{U}_{r'}} = (\operatorname{avg}\varphi_2^*\operatorname{ext}^{U_R}_{U_r}H)\Big|_{\hat{U}_{r'}}.
\end{equation}
Since the left hand side of equation \eqref{eq:right_inverse_B_lambda_proof_requirement_H1} agrees with $\MP(\varphi_2^{-1}\varphi_1)H_1$ and $\MQ$ is a right inverse of $\MP$, equation \eqref{eq:right_inverse_B_lambda_proof_requirement_H1} holds for
\begin{equation*}
H_1\coloneqq \MQ(\varphi_2^{-1}\varphi_1) \left( (\operatorname{avg}\varphi_2^*\operatorname{ext}^{U_R}_{U_r}H)\Big|_{\hat{U}_{r'}} \right).
\end{equation*}
Given $H_1$ satisfying \eqref{eq:right_inverse_B_lambda_proof_requirement_H1}, it is easy to find $H_2\in \hat{\MH}_0(U_r)$ such that \eqref{eq:right_inverse_B_lambda_proof_b} holds. Simply set
\begin{equation*}
H_2\coloneqq \operatorname{pr}_0\operatorname{ext}^{U_R}_{U_{r'}} \left( (\varphi_2^*\operatorname{ext}^{U_R}_{U_r}H - (\varphi_2^{-1}\varphi_1C_\lambda)_*H_1)\Big|_{U_{r'}} \right).
\end{equation*}
Now define $\MRR(\varphi)H \coloneqq (H_1,H_2)$. The above discussion shows that $\MB_\lambda(\varphi)\MRR(\varphi)H = H$, i.e.\ that $\MRR(\varphi)$ is a right inverse of $\MB_\lambda(\varphi)$ for $\varphi$ sufficiently close to $(\operatorname{id},\operatorname{id})$. The family $\MRR$ is smooth tame because it is the composition of smooth tame families of linear maps. This concludes the proof of Proposition \ref{prop:right_inverse_B_lambda} and it remains to verify Lemma \ref{lem:right_inverse_P}.
\end{proof}

\begin{proof}[Proof of Lemma \ref{lem:right_inverse_P}]
Fix $r'<r''<R$ and a tame linear extension operator $\operatorname{ext}^{\hat{U}_{r''}}_{\hat{U}_{r'}}:\ME_0(\hat{U}_{r'})\rightarrow \MH_0(\hat{U}_{r''})$. If $\lambda$ is sufficiently close to $0$, then
\begin{equation*}
\hat{\MQ}(\varphi) \coloneqq \operatorname{pr}_0 (\varphi C_\lambda)^* \operatorname{lift}_{1/2}\operatorname{ext}_{\hat{U}_{r'}}^{\hat{U}_{r''}}
\end{equation*}
defines a smooth tame family of linear maps $\hat{\MQ}:\MU\times\ME_0(\hat{U}_{r'})\rightarrow \hat{\MH}_0(U_R)$ on some neighbourhood $\MU\subset \operatorname{Ham}_0(U_R)$ of the identity. Our goal is to show that the composition $\MP \hat{\MQ}$ is a family of invertible linear maps and that the family of inverses is smooth tame. Once we know this, we can define a smooth tame family $\MQ$ of right inverses of $\MP$ by $\MQ\coloneqq \hat{\MQ}(\MP\hat{\MQ})^{-1}$. Let us define $\MG\coloneqq \operatorname{id}_{\ME_0(\hat{U}_{r'})} - \MP\hat{\MQ}$. Our strategy is to show that the geometric series $\sum_{j\geq 0}\MG^j$ converges and yields the desired inverse of $\MP\hat{\MQ}$. This requires the following claim.

\begin{claim}
\label{claim:powers_of_G_estimates}
Suppose that the neighbourhood $\MU$ of the identity is sufficiently small. Then for every integer $k>0$, there exist constants $0<\mu<1$ and $C>0$ such that, for all integers $j\geq 0$, all $\varphi\in \MU$ and all $f\in \ME_0(\hat{U}_{r'})$ we have
\begin{equation}
\label{eq:powers_of_G_estimates}
\|\MG(\varphi)^jf\|_k \leq C\mu^j(\|f\|_k + (\|\varphi\|_k+1)\|f\|_1).
\end{equation}
Here $\|\cdot\|_k$ denotes the $C^k$ norm.
\end{claim}

Postponing the proof of this claim, let us first explain how it is used to prove Lemma \ref{lem:right_inverse_P}. Since $\sum_{j\geq 0} \mu^j < \infty$, it follows from estimate \eqref{eq:powers_of_G_estimates} that
\begin{equation*}
\MI(\varphi)\coloneqq \sum_{j\geq 0} \MG(\varphi)^j : \ME_0(\hat{U}_{r'})\rightarrow \ME_0(\hat{U}_{r'})
\end{equation*}
is a well-defined linear operator for all $\varphi\in \MU$ satisfying the estimate
\begin{equation}
\label{eq:right_inverse_P_proof_estimate_I}
\|\MI(\varphi)f\|_k \leq C(k) (\|f\|_k + (\|\varphi\|_k+1)\|f\|_1)
\end{equation}
for all $k$. Clearly, we have
\begin{equation*}
\MI(\varphi) = (\operatorname{id}_{\ME_0(\hat{U}_{r'})}-\MG(\varphi))^{-1} = (\MP(\varphi)\hat{\MQ}(\varphi))^{-1}.
\end{equation*}
It remains to show that $\MI$ is a smooth tame family of linear maps. In view of \cite[\S II.3.1, Theorem 3.1.1]{ham82}, it suffices to show that $\MI$ is continuous and tame. Tameness is immediate from \eqref{eq:right_inverse_P_proof_estimate_I}. Note that estimate \eqref{eq:powers_of_G_estimates} implies that the $C^k$ convergence of $\sum_{0\leq j\leq N}\MG(\varphi)^jf$ to $\MI(\varphi)f$ as $N\rightarrow \infty$ is uniform on every subset of $\MU\times \ME_0(\hat{U}_{r'})$ which is $C^k$ bounded. This implies continuity of $\MI$.\\

It remains to verify Claim \ref{claim:powers_of_G_estimates}. Our derivation of the desired estimate~\eqref{eq:powers_of_G_estimates} will make use of the following interpolation inequality, which can for example be found in \cite[\S II.2.2, Theorem 2.2.1]{ham82}.

\begin{theorem}
Let $0\leq \ell \leq m \leq n$. Then for all $f \in C_c^\infty(\BR^d)$
\begin{equation*}
\|f\|_m^{n-\ell} \leq C \|f\|_n^{m-\ell} \|f\|_\ell^{n-m}.
\end{equation*}
\end{theorem}

More precisely, we will apply several useful estimates on the $C^k$ norms of sums and compositions of functions and inverses of diffeomorphisms, which are derived from this interpolation theorem in \cite[\S II.2.2]{ham82}. We record the relevant estimates here.

\begin{lemma}[\cite{ham82}, \S II.2.2, Corollary 2.2.3]
\label{lem:Ck_norm_product}
Let $k \geq 0$. Then for all $f, g \in C_c^\infty(\BR^d)$
\begin{equation*}
\|f g\|_k \leq C (\|f\|_k \|g\|_0 + \|f\|_0 \|g\|_k).
\end{equation*}
\end{lemma}

\begin{lemma}[\cite{ham82}, \S II.2.2, Lemma 2.3.4]
\label{lem:Ck_norm_composition}
Consider balls $B^\ell \subset \BR^\ell$, $B^m \subset \BR^m$ and $B^n \subset \BR^n$. Let $f : B^m \rightarrow B^n$ and $g : B^\ell \rightarrow B^m$ be smooth maps. Fix a constant $K$ and assume that $\|f\|_1 \leq K$ and $\|g\|_1 \leq K$. Then there exists a constant $C$ depending on $K$ such that
\begin{equation*}
\|f \circ g\|_k \leq C(\|f\|_k + \|g\|_k + 1).
\end{equation*}
\end{lemma}

The following estimate appears as an intermediate inequality in the proof of Lemma~\ref{lem:Ck_norm_composition} given in \cite{ham82}.

\begin{lemma}
\label{lem:Ck_norm_composition_refined}
Consider balls $B^\ell \subset \BR^\ell$, $B^m \subset \BR^m$ and $B^n \subset \BR^n$. Let $f : B^m \rightarrow B^n$ and $g : B^\ell \rightarrow B^m$ be smooth maps. Then for $k\geq 1$
\begin{equation*}
\|d^k(f\circ g)\|_0 \leq C (\|f\|_k \|g\|_1^k + \|f\|_1\|g\|_1^{k-1}\|g\|_k).
\end{equation*}
\end{lemma}

\begin{lemma}[\cite{ham82}, \S II.2.2, Lemma 2.3.6]
\label{lem:Ck_norm_inverse}
Let $B^d \subset \BR^d$ be a ball and consider a smooth embedding $f : B^d \rightarrow \BR^d$ which is a diffeomorphism onto its image and sufficiently $C^1$ close to the inclusion. Then
\begin{equation*}
\|f^{-1}\|_k \leq C (\|f\|_k + 1).
\end{equation*}
\end{lemma}

Our next step is to derive expressions for $\MG(\varphi)$ and $\MG(\varphi)^j$ as certain integral operators. Let us abbreviate $\operatorname{ext}\coloneqq \operatorname{ext}_{\hat{U}_{r'}}^{\hat{U}_{r''}}$ and $\operatorname{res}\coloneqq \operatorname{res}_{\hat{U}_{r'}}$. We compute
\begin{eqnarray*}
\MG(\varphi) & = & \operatorname{id} - \MP(\varphi)\hat{\MQ}(\varphi) \\
& = & \operatorname{id} - \operatorname{res}\operatorname{avg}(\varphi C_\lambda)_*\operatorname{pr}_0(\varphi C_\lambda)^*\operatorname{lift}_{1/2}\operatorname{ext} \\
& = & \operatorname{id} - \operatorname{res}\operatorname{avg}(\varphi C_\lambda)_*(\operatorname{id}-\operatorname{lift}_0\operatorname{avg})(\varphi C_\lambda)^*\operatorname{lift}_{1/2}\operatorname{ext}\\
& = & \operatorname{res}\operatorname{avg}(\varphi C_\lambda)_*\operatorname{lift}_0\operatorname{avg}(\varphi C_\lambda)^*\operatorname{lift}_{1/2}\operatorname{ext}.
\end{eqnarray*}
Here the last equality uses that $\operatorname{lift}_{1/2}$ is a right inverse of $\operatorname{avg}$ and $\operatorname{ext}$ is a right inverse of $\operatorname{res}$. In order to further evaluate this expression, write $\varphi C_\lambda$ and $(\varphi C_\lambda)^{-1}$ in components
\begin{equation*}
\varphi C_\lambda = (S,T,P) \quad \text{and}\quad (\varphi C_\lambda)^{-1} = (\overline{S},\overline{T},\overline{P})
\end{equation*}
with respect to the product $M = [0,\infty)\times \BT \times \BR^{2n-2}$. For $\tau = (\tau^1,\tau^2)\in \BT^2$, we define functions
\begin{align*}
A_\tau: [0,\infty) \times \BR^{2n-2} &\rightarrow [0,\infty) \times \BT \times \BR^{2n-2} &
A_\tau(s,p) &\coloneqq (\overline{S}(s,\tau^1,p) , \tau^2 , \overline{P}(s,\tau^1,p)) \\
B_\tau: [0,\infty) \times \BR^{2n-2} &\rightarrow [0,\infty) \times \BR^{2n-2} &
B_\tau(s,p) &\coloneqq (S(A_\tau(s,p)), P(A_\tau(s,p))) \\
W_\tau: [0,\infty) \times \BR^{2n-2} &\rightarrow \BR &
W_\tau(s,p) &\coloneqq \rho_0(\overline{T}(s,\tau^1,p))\cdot \rho_{1/2}(T(A_\tau(s,p))).
\end{align*}
A direct computation shows that
\begin{equation}
\label{eq:right_inverse_P_proof_first_expression_G}
(\MG(\varphi)f)(s,p) 
= \int_{\BT^2} W_\tau(s,p)\cdot (\operatorname{ext}f)(B_\tau(s,p)) d\tau \qquad \text{for $(s,p)\in [0,r')\times B$}.
\end{equation}
Let us define the subset
\begin{equation*}
U\coloneqq \left\{ (\tau^1,\tau^2)\in \BT^2 \mid -(1-\lambda)/8 < \tau^1 < (1-\lambda)/8 \enspace\text{and}\enspace 3/8 < \tau^2 < 5/8 \right\}
\end{equation*}
If $\varphi$ is sufficiently $C^\infty$ close to the identity, then $W_\tau$ vanishes if $\tau\notin U$ and is $C^\infty$ close to the constant
\begin{equation}
\label{eq:right_inverse_P_proof_Wtau0}
W_\tau^0 \coloneqq \rho_0( (1-\lambda)^{-1} \tau_1) \cdot \rho_{1/2}(\tau_2)
\end{equation}
if $\tau\in U$. The function $B_\tau$ is $C^\infty$ close to
\begin{equation*}
B^0(s,p)\coloneqq ( (1-\lambda)s,p)
\end{equation*}
for all $\tau\in U$. Since this implies that $B_\tau(\hat{U}_{r'}) \subset \hat{U}_{r'}$ for $\tau\in U$, we can simplify \eqref{eq:right_inverse_P_proof_first_expression_G} to
\begin{equation*}
\MG(\varphi)f = \int_U W_\tau \cdot f(B_\tau) d\tau.
\end{equation*}
For $j>0$ and $\tau= (\tau_1,\dots,\tau_j)\in U^j$, let us abbreviate
\begin{equation*}
B_{\tau} \coloneqq B_{\tau_1}\cdots B_{\tau_j}\quad \text{and} \quad
W_{\tau}\coloneqq
W_{\tau_1}(B_{(\tau_2,\dots,\tau_j)}) \cdot W_{\tau_2}(B_{(\tau_3,\dots,\tau_j)}) \cdots W_{\tau_{j-1}}(B_{\tau_j}) \cdot W_{\tau_j}.
\end{equation*}
A simple computation shows that
\begin{equation*}
\MG(\varphi)^jf = \int_{U^j} W_{\tau} \cdot f(B_{\tau}) d\tau.
\end{equation*}

The following claim provides certain estimates on $B_\tau$ and $W_\tau$. When taking their norms $\|B_\tau\|_k$ and $\|W_\tau\|_k$, we regard them as maps $B_\tau:\hat{U}_{r'}\rightarrow \hat{U}_{r'}$ and $W_\tau:\hat{U}_{r'}\rightarrow \BR$, respectively.

\begin{claim}
\label{claim:estimates_W_B}
\begin{enumerate}
\item \label{item:estimates_W_B_integral_W} Let $1-\lambda < \mu < 1$. If $\varphi$ is sufficiently close to the identity, then we have
\begin{equation*}
\int_{U^j} \|W_{\tau}\|_0 d\tau \leq \mu^j \quad \text{for all $j\geq 1$.}
\end{equation*}
\item \label{item:estimates_W_B_Ck_norms} For all $k\geq 0$ and $j>0$, we have
\begin{equation*}
\max\left\{\sup_{\tau\in U^j} \|W_{\tau}\|_k , \sup_{\tau\in U^j} \|B_{\tau}\|_k\right\} \leq C(k,j) (\|\varphi\|_k+1)
\end{equation*}
where the constants $C(k,j)$ only depend on $k$ and $j$.
\item \label{item:estimates_W_B_C1_norm_B} There exists a constant $C>0$ such that for all $\varphi$ sufficiently close to the identity and for all $j\geq 1$ we have
\begin{equation*}
\sup_{\tau\in U^j} \|B_\tau\|_1 < C.
\end{equation*}
\end{enumerate}
\end{claim}

\begin{proof}
We prove assertion~\eqref{item:estimates_W_B_integral_W}. For every $\tau\in U^j$, we have $\|W_\tau\|_0 \leq \prod_{1\leq i\leq j} \|W_{\tau_i}\|_0$ and hence
\begin{equation*}
\int_{U^j}\|W_\tau\|_0 d\tau \leq \int_{U^j} \prod\limits_{i=1}^j \|W_{\tau_i}\|_0 d\tau \leq \left( \int_U \|W_\tau\|_0 d\tau \right)^j.
\end{equation*}
Recall that for every $\tau\in U$, the function $W_\tau$ is $C^\infty$ close to the function $W_\tau^0$ defined in equation~\eqref{eq:right_inverse_P_proof_Wtau0}. We have $\int_U\|W_\tau^0\|d\tau = \int_UW_\tau^0d\tau = 1-\lambda$. If $\varphi$ is sufficiently close to the identity, we therefore have $\int_U\|W_\tau\| d\tau \leq \mu$. This shows assertion~\eqref{item:estimates_W_B_integral_W}.

Next, we prove assertion~\eqref{item:estimates_W_B_Ck_norms}. Note that for all $\varphi$ sufficiently close to $\operatorname{id}$ and for all $\tau\in U$, we have $\|B_\tau\|_1\leq C$ for some constant independent of $\varphi$ and $\tau$. Repeated application of Lemma~\ref{lem:Ck_norm_composition} thus yields
\begin{equation*}
\sup_{\tau\in U^j} \|B_\tau\|_k \leq C(k,j) ( \sup_{\tau\in U} \|B_\tau\|_k + 1).
\end{equation*}
Recall that $B_\tau$ is defined as a composition of components of the maps $\varphi C_\lambda$ and $(\varphi C_\lambda)^{-1}$. Using Lemmas~\ref{lem:Ck_norm_composition} and~\ref{lem:Ck_norm_inverse}, we can estimate $\sup_{\tau\in U}\|B_\tau\|_k \leq C(k)(\|\varphi\|_k + 1)$. This yields the desired upper bound on $\sup_{\tau\in U^j}\|B_\tau\|_k$.

We have a uniform upper bound on $\|W_\tau\|_1$ for all $\varphi$ sufficiently close to $\operatorname{id}$ and all $\tau\in U$. Using Lemma~\ref{lem:Ck_norm_product} repeatedly, we can hence estimate for $\tau\in U^j$
\begin{equation*}
\|W_\tau\|_k \leq  C(k,j) \sum_{1\leq i \leq j} \|W_{\tau_i}(B_{(\tau_{i+1},\dots,\tau_j)})\|_k.
\end{equation*}
Using Lemma~\ref{lem:Ck_norm_composition} and the upper bound on $\|B_{\tau}\|_k$ proved above, we can further estimate
\begin{equation*}
\sup_{\tau\in U^j} \|W_\tau\|_k \leq C(k,j) ( \|\varphi\|_k + 1 + \sup_{\tau\in U}\|W_\tau\|_k).
\end{equation*}
Since the definition of $W_\tau$ only involves compositions of the components of $\varphi C_\lambda$ and $(\varphi C_\lambda)^{-1}$ and the fixed functions $\rho_0$ and $\rho_{1/2}$, Lemmas~\ref{lem:Ck_norm_product},~\ref{lem:Ck_norm_composition} and~\ref{lem:Ck_norm_inverse} can be used to show
\begin{equation*}
\sup_{\tau\in U} \|W_\tau\|_k \leq C(k,j)(\|\varphi\|_k+1).
\end{equation*}
The desired estimate on $\sup_{\tau\in U^j}\|W_\tau\|_k$ follows. This proves assertion~\eqref{item:estimates_W_B_Ck_norms}.

Finally, we turn to assertion~\eqref{item:estimates_W_B_C1_norm_B}. Fix an arbitrary sequence $(\tau_i)_{i\geq 0}$ in $U$. Let $(s_0,p_0)\in [0,r')\times B$ and let $(\sigma_0,\xi_0)$ be a tangent vector at $(s_0,p_0)$ such that $\max\left\{ |\sigma_0|,|\xi_0| \right\}\leq 1$. For $i\geq 0$, we recursively define
\begin{equation*}
(s_{i+1},p_{i+1})\coloneqq B_{\tau_i}(s_{i},p_{i}) \quad \text{and} \quad (\sigma_{i+1},\xi_{i+1})\coloneqq dB_{\tau_i}(s_{i},p_{i})(\sigma_{i},\xi_{i}).
\end{equation*}
Our goal is to show that there exists a constant $C>0$ which is independent of the sequence $(\tau_i)_i$, the point $(s_0,p_0)$, the vector $(\sigma_0,\xi_0)$ and the diffeomorphism $\varphi\in \MU$ such that
\begin{equation*}
\sup_i \max\left\{ |\sigma_i|,|\xi_i| \right\} \leq C.
\end{equation*}
Fix a small constant $\delta>0$. Assume that $\delta$ is sufficiently small such that $\alpha\coloneqq 1-\lambda+\delta < 1$. After shrinking the neighbourhood $\MU\subset \operatorname{Ham}_0(U_R)$ of the identity if necessary, we can assume that $\|B_\tau - B^0\|_2 <\delta$ for all $\tau\in U$. We claim that
\begin{equation}
\label{eq:estimates_W_B_si}
s_i \leq \alpha^ir' \quad \text{for all $i$.}
\end{equation}
Indeed, since $P(0,t,p)= p$ for all $t$ and $p$, a short computation shows that $B_\tau$ restricts to the identity on $\left\{ 0 \right\}\times B$ for all $\tau\in U$. Moreover, we observe that $\partial_sB^0 = (1-\lambda)\partial_s$ where $s$ is the coordinate of the first factor of $\hat{U}_{r'} = [0,r')\times B$. Together with $\|B_\tau - B^0\|_2<\delta$, this yields $s_{i+1}\leq \alpha s_i$ and therefore the desired estimate \eqref{eq:estimates_W_B_si}.

Let us write
\begin{equation}
\label{eq:estimates_W_B_definition_epsilon}
dB_\tau(s,p) =
\left(\begin{matrix}
1-\lambda & 0 \\ 0 & \operatorname{id}
\end{matrix}\right)
+ 
\left(\begin{matrix}
\varepsilon_\tau^{11}(s,p) & \varepsilon_\tau^{12}(s,p) \\ \varepsilon_\tau^{21}(s,p) & \varepsilon_\tau^{22}(s,p)
\end{matrix}\right).
\end{equation}
The error terms satisfy $\|\varepsilon_\tau^{ij}\|_1 \leq \delta$. Since $B_\tau$ restricts to the identity on $\left\{ 0 \right\}\times B$ as observed above, we have
\begin{equation*}
dB_\tau(0,p) =
\left( \begin{matrix}
* & 0 \\ * & \operatorname{id}
\end{matrix} \right).
\end{equation*}
This implies that $\varepsilon_\tau^{i2}(0,p) = 0$. Together with $\|\varepsilon_\tau^{ij}\|_1\leq \delta$, we obtain
\begin{equation}
\label{eq:estimates_W_B_epsilon_i2}
|\varepsilon_\tau^{i2}(s,p)| \leq \delta s.
\end{equation}
Consider an integer $N>0$ and assume that $\max_{0\leq i \leq N} |\xi_i| \leq 2$. For $i\leq N$, we can estimate
\begin{eqnarray}
|\sigma_{i+1}| & = & |(1-\lambda+\varepsilon_{\tau_{i}}^{11}(s_{i},p_{i}))\sigma_{i} + \varepsilon_{\tau_{i}}^{12}(s_{i},p_{i})\xi_{i}| \nonumber\\
& \leq & \alpha|\sigma_{i}| + \delta s_{i} |\xi_{i}| \nonumber \\
& \leq & \alpha|\sigma_{i}| + 2\delta r'\alpha^{i}. \label{eq:estimates_W_B_sigmai}
\end{eqnarray}
Here the equality follows from \eqref{eq:estimates_W_B_definition_epsilon}. The first inequality uses $\|\varepsilon_\tau^{ij}\|_1 \leq \delta$ and estimate \eqref{eq:estimates_W_B_epsilon_i2}. The last inequality uses \eqref{eq:estimates_W_B_si} and the assumption $|\xi_i|\leq 2$.

Set $\beta\coloneqq 1-\lambda+2\delta$ and assume that $\delta$ is sufficiently small such that $\beta<1$. We claim that for $i\leq N+1$, we have
\begin{equation}
\label{eq:estimates_W_B_sigmai_absolute}
|\sigma_i| \leq 4(r'+1)\beta^i.
\end{equation}
Indeed, define a sequence of positive real numbers $(a_i)_i$ by
\begin{equation*}
a_0\coloneqq 1 \quad \text{and} \quad a_{i+1}\coloneqq \alpha a_i + 2\delta r' \alpha^i.
\end{equation*}
It follows from estimate \eqref{eq:estimates_W_B_sigmai} and the assumption $|\sigma_0|\leq 1$ that $|\sigma_{i}|\leq a_i$ for $i\leq N+1$. An elementary computation shows that
\begin{equation*}
a_i = \alpha^i + 2\delta r' i \alpha^{i-1} \leq 4 (r'+1) \beta^i,
\end{equation*}
from which we deduce \eqref{eq:estimates_W_B_sigmai_absolute}.

Next, we estimate for $i\leq N$
\begin{eqnarray}
|\xi_{i+1}-\xi_{i}| & = & |\varepsilon_{\tau_{i}}^{11}(s_{i},p_{i})\sigma_{i} + \varepsilon_{\tau_{i}}^{12}(s_{i},p_{i})\xi_{i}| \nonumber\\
& \leq & \delta|\sigma_{i}| + \delta s_{i} |\xi_{i}| \nonumber\\
& \leq & 4(r'+1)\delta\beta^i + 2\delta r' \alpha^i \nonumber\\
& \leq & 6(r'+1)\delta \beta^i. \label{eq:estimates_W_B_xii}
\end{eqnarray}
Here the equality follows from \eqref{eq:estimates_W_B_definition_epsilon}. The first inequality uses that $\|\varepsilon_\tau^{ij}\|_1 \leq \delta$ and estimate \eqref{eq:estimates_W_B_epsilon_i2}. The second inequality uses the assumption that $|\xi_i|\leq 2$ and estimates \eqref{eq:estimates_W_B_si} and \eqref{eq:estimates_W_B_sigmai_absolute}. We can therefore further estimate
\begin{eqnarray*}
|\xi_{N+1}| & \leq & |\xi_0| + |\xi_{N+1}- \xi_0| \\
& \leq & 1 + 6(r'+1)\delta\sum\limits_{i=0}^N \beta^{i} \\
& \leq & 1 + \frac{6(r'+1)}{1-\beta}\delta.
\end{eqnarray*}
Here the second inequality uses the assumption $|\xi_0|\leq 1$ and estimate \eqref{eq:estimates_W_B_xii}. We see that if $\delta$ is chosen sufficiently small, then we have $|\xi_{N+1}|\leq 2$. We deduce that we must have $|\xi_i|\leq 2$ for all $i$. Together with \eqref{eq:estimates_W_B_sigmai_absolute}, this implies that both sequences $|\sigma_i|$ and $|\xi_i|$ are bounded by some constant independent of $(\tau_i)_i$, $(s_0,p_0)$, $(\sigma_0,\xi_0)$ or $\varphi$. This concludes the proof of Claim \ref{claim:estimates_W_B}
\end{proof}

For $j>0$ we estimate
\begin{eqnarray*}
\|\MG(\varphi)^jf\|_0 & \leq & \int_{U^j} \|W_\tau\cdot f(B_\tau)\|_0 d\tau \\
& \leq & \int_{U^j} \|W_\tau\|_0 \|f\|_0 d\tau \\
& \leq & \mu^j \|f\|_0.
\end{eqnarray*}
Here the last inequality uses assertion~\eqref{item:estimates_W_B_integral_W} in Claim \ref{claim:estimates_W_B}. If $k\geq 1$, we estimate
\begin{eqnarray*}
\|\MG(\varphi)^jf\|_k &\leq & \int_{U^j} \|W_{\tau} \cdot f(B_{\tau})\|_k d\tau \\
& \leq & C(k) \int_{U^j} (\|W_{\tau}\|_0\|f(B_{\tau})\|_k + \|W_{\tau}\|_k\|f(B_{\tau})\|_0) d\tau \\
& \leq & C(k) \int_{U^j} (\|W_{\tau}\|_0 \|f\|_k + \|W_\tau\|_0\|B_\tau\|_k\|f\|_1 + \|W_\tau\|_k\|f\|_0) d\tau\\
& \leq & 
\begin{cases}
C(1) \mu^j\|f\|_1 + C(1,j)\|f\|_0 & \text{if $k=1$}\\
C(k) \mu^j\|f\|_k + C(k,j)(\|\varphi\|_k+1)\|f\|_1 & \text{if $k\geq 2$.}
\end{cases}
\end{eqnarray*}
Here the second inequality uses Lemma~\ref{lem:Ck_norm_product}. The third inequality uses Lemma~\ref{lem:Ck_norm_composition_refined} and assertion~\eqref{item:estimates_W_B_C1_norm_B} in Claim~\ref{claim:estimates_W_B}. The last inequality makes use of assertions~\eqref{item:estimates_W_B_integral_W} and~\eqref{item:estimates_W_B_Ck_norms} in Claim~\ref{claim:estimates_W_B}. We emphasize that while the constants $C(k)$ in the last expression are independent of $j$, the constants $C(k,j)$ are allowed to depend on $j$.

Fix $k\geq 2$. Pick $j\geq 1$ such that $\mu^j< 1/3$ and such that $C(\ell)\mu^j < 1/3$ for all $1\leq \ell \leq k$. Then the above estimates on $\|\MG(\varphi)^jf\|_\ell$ for $0 \leq \ell \leq k$ imply
\begin{equation*}
\left( \begin{matrix}
\|\MG(\varphi)^jf\|_0 \\  \\ \vdots \\  \\ \|\MG(\varphi)^jf\|_k
\end{matrix} \right)
\leq
\left( \begin{matrix}
\frac{1}{3} & & & & \\
C & \frac{1}{3} & & & \\
0 & C(\|\varphi\|_2+1) & \frac{1}{3} & & \\
\vdots & \vdots & & \ddots & \\
0 & C(\|\varphi\|_k+1) & & & \frac{1}{3}
\end{matrix} \right)
\left( \begin{matrix}
\|f\|_0 \\ \\ \vdots \\ \\ \|f\|_k
\end{matrix} \right)
\end{equation*}
for some constant $C>0$. Here the inequality is interpreted componentwise. Let $A=\frac{1}{3}\operatorname{id} + N$ be the matrix on the right hand side of this inequality. For every $m>0$, we obtain the inequality
\begin{equation}
\label{eq:right_inverse_P_matrix_inequality}
\left( \begin{matrix}
\|\MG(\varphi)^{mj}f\|_0 \\ \vdots \\ \|\MG(\varphi)^{mj}f\|_i
\end{matrix} \right) \leq A^m
\left( \begin{matrix}
\|f\|_0 \\ \vdots \\ \|f\|_k
\end{matrix} \right).
\end{equation}
Using that $N^3 = 0$, we estimate
\begin{eqnarray}
A^m & = & \sum\limits_{i=0}^m  \binom{m}{i} \frac{1}{3^{m-i}}N^i \nonumber\\
& = & \sum\limits_{i=0}^2 \binom{m}{i} \frac{1}{3^{m-i}}N^i \nonumber\\
& \leq & C \frac{1}{2^m}(\operatorname{id} + N + N^2) \label{eq:right_inverse_P_matrix_power_estimate}
\end{eqnarray}
where the inequality is again interpreted componentwise. Note that the matrix $N+N^2$ is strictly lower triangular. Moreover, all its non-vanishing entries are contained in the first two columns and are bounded from above by $C(\|\varphi\|_k+1)$. In combination with inequalities \eqref{eq:right_inverse_P_matrix_inequality} and \eqref{eq:right_inverse_P_matrix_power_estimate}, this yields
\begin{equation*}
\|\MG(\varphi)^{mj}f\|_k \leq C\frac{1}{2^m}(\|f\|_k + (\|\varphi\|_k+1)\|f\|_1)
\end{equation*}
for some constant $C$ independent of $m$. This implies the desired estimate \eqref{eq:powers_of_G_estimates} for the choice $\mu \coloneqq (1/2)^{1/j}$.
\end{proof}

\subsection{Proof of Theorem \ref{thm:herman_with_boundary}}
\label{subsec:proof_of_herman_with_boundary}

Fix $0<r<R$, a number $\lambda>0$ sufficiently close to $0$ and $\alpha\in \BT$ satisfying a Diophantine condition. Our goal is to show that the derivative $d\MF_{\alpha,\lambda}$ of the map $\MF_{\alpha,\lambda}$ in equation~\eqref{eq:Falphalambda_frechet} possesses a smooth tame family of right inverses in some open neighbourhood of $(\operatorname{id},\operatorname{id})\in \operatorname{Ham}_0(U_R)^2$. Once we have such a family of right inverses, we are in a position to apply the Nash--Moser theorem.

Let $\varphi=(\varphi_1,\varphi_2)\in \operatorname{Ham}_0(U_R)^2$ be sufficiently close to $(\operatorname{id},\operatorname{id})$ and let $H\in \ME_0(U_r)$. We need to construct $(H_1,H_2)\in \MH_0(U_R)^2$ such that $d\MF_{\alpha,\lambda}(\varphi)(H_1,H_2) = H$. By Proposition \ref{prop:derivative_of_F}, this is equivalent to
\begin{equation}
\label{eq:herman_with_boundary_proof_a}
[R_\alpha,\varphi_2]^*(\varphi_1)_*( ( (R_{\alpha,\lambda})_* - \operatorname{id})H_1 + (\varphi_2)_* ( (R_\alpha)_* - \operatorname{id})H_2) \Big|_{U_r} = H.
\end{equation}
Fix $r<r'<R$ and choose a tame linear extension operator $\operatorname{ext}_{U_{r}}^{U_R}:\ME_0(U_r)\rightarrow \MH_0(U_R)$. If $\varphi$ is sufficiently close to $(\operatorname{id},\operatorname{id})$, then equation \eqref{eq:herman_with_boundary_proof_a} is implied by
\begin{equation}
\label{eq:herman_with_boundary_proof_b}
((\varphi_1C_\lambda)_*( (R_{\alpha})_* - \operatorname{id})C_\lambda^{*}H_1 + (\varphi_1\varphi_2)_*( (R_\alpha)_* - \operatorname{id})H_2)\Big|_{U_{r'}}  = ([R_\alpha,\varphi_2]_* \operatorname{ext}_{U_r}^{U_R} H)\Big|_{U_{r'}}.
\end{equation}
Fix $r'<R'<R$. Applying Proposition \ref{prop:right_inverse_B_lambda} with $r'<R'$ in place of $r<R$, we obtain a smooth tame family of right inverses $\tilde{\MB}_\lambda(\varphi):\ME_{0}(U_{r'})\rightarrow\hat{\MH}_0(U_{R'})^2$. Moreover, Proposition \ref{prop:right_inverse_D_alpha} yields a tame linear right inverse $\tilde{\MD}_\alpha:\hat{\MH}_0(U_{R'})\rightarrow \MH_0(U_{R'})$. Set
\begin{equation*}
(H_1,H_2)\coloneqq ( (C_\lambda)_*\tilde{\MD}_\alpha,\tilde{\MD}_\alpha)\tilde{\MB}_\lambda(\varphi_1,\varphi_1\varphi_2)([R_\alpha,\varphi_2]_* \operatorname{ext}_{U_r}^{U_R} H)\Big|_{U_{r'}}.
\end{equation*}
If $\lambda$ is sufficiently close to $0$, then $(C_\lambda)_*$ can be regarded as a linear map $\MH_0(U_{R'})\rightarrow \MH_0(U_R)$. Hence both $H_1$ and $H_2$ are contained in $\MH_0(U_R)$. Clearly, \eqref{eq:herman_with_boundary_proof_b} holds for $(H_1,H_2)$. Therefore, we can define a right inverse of $d\MF_{\alpha,\lambda}(\varphi)$ by setting $\MRR(\varphi)H\coloneqq (H_1,H_2)$ for all $\varphi$ sufficiently close to $(\operatorname{id},\operatorname{id})$. The family of right inverses $\MRR$ is smooth tame because it is the composition of smooth tame families of linear maps. This concludes the proof of the claim that the derivative of the map $\MF_{\alpha,\lambda}$ in equation~\eqref{eq:Falphalambda_frechet} possesses a smooth tame family of right inverses.

We can now deduce Theorem~\ref{thm:herman_with_boundary}. Recall that we need to construct a smooth local right inverse for the map
\begin{equation}
\label{eq:proof_of_herman_with_boundary_a}
\MF_{\alpha,\lambda} : \operatorname{Ham}_c(U_R)^2 \rightarrow \operatorname{Emb}_c(U_r)
\end{equation}
which is defined in some open neighbourhood $\MN \subset \operatorname{Emb}_c(U'_r)$ of the inclusion $\iota$ and maps $\iota$ to $(\operatorname{id},\operatorname{id})$. Fix $r<S<R$. It follows from the Nash--Moser theorem that the map
\begin{equation*}
\MF_{\alpha,\lambda} : \operatorname{Ham}_0(U'_S)^2 \rightarrow \operatorname{Emb}_0(U'_r)
\end{equation*}
has a local right inverse which is smooth as a map between Fr\'echet manifolds (and hence in the diffeological sense) and maps $\iota$ to $(\operatorname{id},\operatorname{id})$. We pre- and post-compose such a right inverse with the inclusions $\operatorname{Emb}_c(U'_r) \subset \operatorname{Emb}_0(U'_r)$ and $\operatorname{Ham}_0(U'_S)^2 \subset \operatorname{Ham}_c(U_R)^2$. This yields the desired smooth local right inverse of the map in equation~\eqref{eq:proof_of_herman_with_boundary_a}.\qed

\section{Proof of Theorem \ref{thm:smooth_perfectness_pairs_of_balls} for pairs of half balls}
\label{sec:smooth_perfectness_half_balls}

In this section we show Theorem~\ref{thm:smooth_perfectness_pairs_of_balls} for pairs of half balls. We closely follow the proof in the case of pairs of full balls, which is explained in Sections~\ref{sec:commutators_and_surfaces} to~\ref{sec:smooth_perfectness_full_balls}. The main difference is that Herman--Sergeraert's theorem (Theorem~\ref{thm:herman}) is replaced by its variation for manifolds with boundary, Theorem~\ref{thm:herman_with_boundary}. Moreover, Proposition~\ref{prop:special_covers} on the existence of a special open cover of the torus is replaced by its analogue, Proposition~\ref{prop:special_open_covers_boundary}.

Throughout this section, let us fix an integer $n>0$. Let $D'\Subset D\Subset \BR^{2n-2}$ be two open balls centered at the origin. Fix $0<r<R$. Recall that for $s>0$ we abbreviate $A_s \coloneqq [0,s)\times \BT$. Moreover, we regard $A_s\times D$ as a symplectic manifold with boundary given by $\left\{ 0 \right\}\times \BT \times D$. 

We begin with the following result, which we deduce from Theorem~\ref{thm:herman_with_boundary} and Corollary~\ref{cor:smooth_perfectness_open}. Note that the proof of Corollary~\ref{cor:smooth_perfectness_open} is already complete since we established Theorem~\ref{thm:smooth_perfectness_pairs_of_balls} in the case of pairs of full balls in Section~\ref{sec:smooth_perfectness_full_balls}; see Remark~\ref{rem:manifold_without_boundary_only_require_pairs_of_full_balls}.

\begin{proposition}
\label{prop:smooth_perfectness_half_open_annulus}
There exists an integer $m>0$ such that the following is true. Consider the map
\begin{equation*}
\Phi:\operatorname{Ham}_c^0(A_R\times D)^{2m} \rightarrow \operatorname{Ham}_c^0(A_R\times D) \quad (u_1,v_1,\dots,u_m,v_m) \mapsto \prod\limits_{j=1}^m [u_j,v_j].
\end{equation*}
Then there exist an open neighbourhood $\MN\subset \operatorname{Ham}_c^0(A_r\times D')$ of the identity and a smooth local right inverse
\begin{equation*}
\Psi:\MN\rightarrow \operatorname{Ham}_c^0(A_R\times D)^{2m}
\end{equation*}
of $\Phi$. We can choose $\MN$ and $\Psi$ such that $\Psi(\operatorname{id})$ is arbitrarily close to the tuple $(\operatorname{id},\dots,\operatorname{id})$.
\end{proposition}

\begin{proof}
Fix real numbers $r'$ and $R'$ such that $0<r'<r<R'<R$. Moreover, fix an open ball $D'\Subset E \Subset D$. Let $\lambda>0$ be a small positive number. Let $\alpha\in \BT$ be a number close to zero satisfying a Diophantine condition. Consider the map
\begin{equation*}
\MF_{\alpha,\lambda}: \operatorname{Ham}_c(A_{r}\times E)^2\rightarrow \operatorname{Emb}_c(A_{r'}\times E)
\end{equation*}
defined in equation~\eqref{eq:Falphalambda}. By Theorem~\ref{thm:herman_with_boundary} we may choose a smooth local right inverse $\MG$ of $\MF_{\alpha,\lambda}$ defined in an open neighbourhood in $\operatorname{Emb}_c(A_{r'}\times D')$ of the inclusion $\iota$ and mapping $\iota$ to $(\operatorname{id},\operatorname{id})$.

Now suppose that $\varphi\in \operatorname{Ham}_c^0(A_r\times D')$ is sufficiently close to the identity. Set
\begin{equation*}
(\varphi_1,\varphi_2)\coloneqq \MG(\varphi|_{A_{r'}\times D'}) \in \operatorname{Ham}_c(A_r\times E)^2.
\end{equation*}
It is not hard to see that there exists a smooth extension map
\begin{equation*}
\operatorname{ext}:\operatorname{Ham}_c(A_{r}\times E) \rightarrow \operatorname{Ham}_c^0(A_{R'}\times E)
\end{equation*}
such that $\operatorname{ext}(\operatorname{id}) = \operatorname{id}$. Set $v_j\coloneqq \operatorname{ext}(\varphi_j)$ for $j\in \left\{ 1,2 \right\}$. Fix Hamiltonian diffeomorphisms $u_1,u_2\in \operatorname{Ham}_c^0(A_{R'}\times D)$ such that
\begin{equation*}
u_1|_{A_{r}\times E} = R_{\alpha,\lambda}|_{A_{r}\times E} \quad \text{and} \quad u_2|_{A_{r}\times E} = R_\alpha|_{A_{r}\times E}
\end{equation*}
and such that $u_j(A_{R'}\times E) =  A_{R'}\times E$ for $j\in \left\{ 1,2 \right\}$. We can assume that $u_1$ and $u_2$ are arbitrarily close to the identity after moving $\lambda$ and $\alpha$ closer to $0$ if necessary.

Note that the commutators $[u_j,v_j]$ are both contained in $\operatorname{Ham}_c^0(A_{R'}\times E)$ because the diffeomorphisms $v_j$ are contained in this group and the diffeomorphisms $u_j$ leave $A_{R'}\times E$ invariant. Moreover, we have
\begin{equation*}
[u_1,v_1][u_2,v_2]\Big|_{A_{r'}\times D} = \varphi|_{A_{r'}\times D}
\end{equation*}
by construction. Set $U\coloneqq (r',R')\times \BT\times E$. We can then write
\begin{equation*}
\varphi = [u_1,v_1][u_2,v_2] \psi
\end{equation*}
with $\psi\in \operatorname{Ham}_c^0(U)$ depending smoothly on $\varphi$ and equal to $\operatorname{id}$ for $\varphi = \operatorname{id}$. Set $M\coloneqq (0,R)\times \BT\times D$. Applying Corollary~\ref{cor:smooth_perfectness_open} to the boundaryless pair $U\Subset M$, we can write
\begin{equation*}
\psi = [u_3,v_3] \cdots [u_m,v_m]
\end{equation*}
where the diffeomorphisms $u_j,v_j\in \operatorname{Ham}_c^0(M)$ depend smoothly on $\psi$ and can be arranged to be arbitrarily close to $\operatorname{id}$ for $\psi = \operatorname{id}$. Moreover, $m$ only depends on the dimension $n$. We define the desired right inverse of $\Phi$ by setting $\Psi(\varphi) \coloneqq (u_1,v_1,\dots,u_m,v_m)$.
\end{proof}

The following result is an analogue of Proposition~\ref{prop:cocycles_torus}.

\begin{proposition}
\label{prop:cocycles_half_open_annulus}
There exists a triangulated surface $(\Sigma,\MT)$ with one boundary component and with one marked $0$-simplex $p_0\in \partial\Sigma$ such that the following is true. For every Hamiltonian diffeomorphism $\varphi\in \operatorname{Ham}_c^0(A_r\times D')$ which is sufficiently close to the identity, there exists a cocycle $c$ on $\MT$ with values in $\operatorname{Ham}_c^0(A_R\times D)$ such that:
\begin{enumerate}
\item We cave $c(\partial\Sigma,p_0) = \varphi$ where $(\partial\Sigma,p_0)$ is interpreted as the loop based at $p_0$ which goes around the boundary once in positive direction.
\item The assignment $\varphi\mapsto c$ is smooth.
\item We can arrange the identity to be mapped arbitrarily close to the identity cocycle.
\item Suppose that $\varphi\in \operatorname{Ham}_c^0(V)$ for some open subset $V\subset A_r\times D'$. Then the restriction of $c$ to $\partial\Sigma$ takes values in $\operatorname{Ham}_c^0(V)$.
\end{enumerate}
\end{proposition}

\begin{proof}
This result can be deduced from Proposition~\ref{prop:smooth_perfectness_half_open_annulus} using the same construction appearing in the proof of Lemma~\ref{lem:commutators_surfaces}.
\end{proof}

Next, we state an analogue of Proposition~\ref{prop:cocycles_torus_fragmented}.

\begin{proposition}
\label{prop:cocycles_half_open_annulus_fragmented}
Let $\MU$ be a finite open covering of $A_R\times D$. Then there exists a triangulated surface $(\Sigma,\MT)$ with exactly one boundary component such that the total number of simplices of $\MT$ is bounded by a constant only depending on $|\MU|$ and such that the following is true: For every Hamiltonian diffeomorphism $\varphi\in \operatorname{Ham}_c^0(A_r\times D')$ sufficiently close to the identity, there exists a cocycle $c$ on $\MT$ with values in $\operatorname{Ham}_c(A_R\times D)$ such that:
\begin{enumerate}
\item We have $c(\partial\Sigma,p_0) = \varphi$. Here $p_0\in \partial\Sigma$ is a marked $0$-simples independent of $\varphi$ and $(\partial\Sigma,p_0)$ refers to the loop based at $p_0$ going around the boundary once in positive direction.
\item For each $2$-simplex $\sigma$ of $\MT$, there exist open sets $U,U'\in \MU$ (only depending on $\sigma$ but not on $\varphi$) such that $U\cap U'\neq \emptyset$ and the restriction of $c$ to $\sigma$ takes values in $\operatorname{Ham}_c(U\cup U')$.
\item The assingment $\varphi\mapsto c$ is smooth.
\item We can arrange the identity to be mapped arbitrarily close to the identity cocycle.
\item Suppose that $V\subset A_r\times D'$ is an open subset and $\varphi\in \operatorname{Ham}_c^0(V)$. Then the restriction of $c$ to $\partial\Sigma$ takes values in $\operatorname{Ham}_c(V)$.
\end{enumerate}
\end{proposition}

\begin{proof}
The fragmentation statements in Section~\ref{sec:fragmentation} (Lemmas~\ref{lem:fragmentation_1d} and~\ref{lem:fragmentation_2d}) also apply for symplectic manifolds with boundary. Therefore the proof of Proposition~\ref{prop:cocycles_torus_fragmented} carries over almost verbatim, with the only difference that Proposition~\ref{prop:cocycles_torus} is replaced by Proposition~\ref{prop:cocycles_half_open_annulus}.
\end{proof}

\begin{proof}[Proof of Theorem \ref{thm:smooth_perfectness_pairs_of_balls}, case of half balls]
Let $n>0$ and let $B'\Subset B \Subset \BR^{2n}$ be a pair of open balls. Let $B'_+ \Subset B_+$ be the corresponding half balls. As in the setting of Proposition~\ref{prop:special_open_covers_boundary}, we regard $B'_+$ and $B_+$ as subsets of $[0,\infty) \times \BT \times \BR^{2n-2}$. Let $C'$ and $C$ denote the images of $B_+'$ and $B_+$ under the natural projection to $\BR^{2n-2}$, respectively. Fix balls $C' \Subset D' \Subset D \Subset C$. Moreover, we fix numbers $r<R<S$ larger than the radius of $B$.

We apply Proposition~\ref{prop:special_open_covers_boundary}. This yields a finite collection $\MV$ of open subsets of $A_S\times C$ which cover $\overline{A_R\times D}$, Hamiltonian diffeomorphisms $\varphi_V$ of $A_S\times C$ moving sets $V\in \MV$ into $B_+$, and Hamiltonian diffeomorphisms $\varphi_{V,W} \in \operatorname{Ham}_c(B_+)$. Next, we apply Lemma~\ref{lem:special_cover_refinement} to obtain a finite open covering $\MU$ of $A_R\times D$ such that $|\MU|$ only depends on $|\MV|$ and such that any two $U,U'\in \MU$ with non-empty intersection are contained in some open set $V\subset \MV$. Now, apply Proposition~\ref{prop:cocycles_half_open_annulus_fragmented}. This yields a triangulated surface $(\Sigma,\MT)$ and an assignment of a cocycle $c$ on $\MT$ with values in $\operatorname{Ham}_c(A_R\times D)$ for every Hamiltonian diffeomorphism $\varphi \in \operatorname{Ham}_c^0(A_r\times D')$ close to the identity.

From this point onwards, the proof is almost identical to the proof in the case of full balls given in Section~\ref{sec:smooth_perfectness_full_balls}. Starting with $\varphi \in \operatorname{Ham}_c^0(B'_+)$, we modify $(\Sigma,\MT)$ and $c$ to obtain a triangulated surface $(\Sigma',\MT')$ with a cocycle $c'$ taking values in $\operatorname{Ham}_c(B_+)$ and evaluating to $\varphi$ on the unique boundary component of $\Sigma'$. Then we further modify $c'$ to obtain a cocycle $c''$ taking values in $\operatorname{Ham}_c^0(B_+)$ and also evaluating to $\varphi$ on the boundary. Just as in Section~\ref{sec:smooth_perfectness_full_balls}, this yields the desired local right inverse $\Psi$ and concludes the proof of Theorem~\ref{thm:smooth_perfectness_pairs_of_balls}.
\end{proof}

\medskip
\medskip
\medskip
\noindent Oliver Edtmair\\
\noindent ETH-ITS, ETH Z\"{u}rich, Scheuchzerstrasse 70, 8006 Z\"{u}rich, Switzerland.\\
{\it Email address:} {\tt oliver.edtmair@eth-its.ethz.ch}

\end{document}